\newcommand{\MA}{\mathfrak{A}}
\newcommand{\MC}{\mathfrak{C}}
\newcommand{\Law}{\mbox{Law}}
\newcommand{\ZZ}{\mathbb{Z}}
\newcommand{\CW}{\mathcal{W}}
\newcommand{\LLa}{\mbox{LlogL}}
\newcommand{\LLn}{\mbox{\footnotesize LlogL}}

\newcommand{\Lve}{\lVert}
\newcommand{\Div}{{\rm div}}
\newcommand{\Rve}{\rVert}

\newcommand{\DeltaA}{A}
\newcommand{\BDG}{Burkholder--Davis--Gundy }

\newcommand{\ul}{\underline}
\newcommand{\ol}{\overline}
\newcommand\delb[1]{}
\newcommand\deln[1]{}

\newcounter{lil11}
\newenvironment{steps1}
{\begin{list} { \bf Step \Roman{lil11}:} { \usecounter{lil11}
\setlength{\leftmargin}{0.0cm}
\setlength{\topsep}{0.2cm} \setlength{\itemsep}{0.0cm}
\setlength{\parsep}{0.1cm} \setlength{\itemindent}{0.8cm}
\setlength{\parskip}{0.0cm}}} {\end{list}}

\newcounter{lil22}
\newenvironment{steps-2}
{\em \begin{list} {  \roman{lil22}:} {\usecounter{lil22} \em
\setlength{\leftmargin}{1.3cm}
\setlength{\topsep}{0.2cm} \setlength{\itemsep}{0.0cm}
\setlength{\parsep}{0.1cm} \setlength{\itemindent}{0.4cm}
\setlength{\parskip}{0.0cm}}} {\end{list}}

\newcommand{\baray}{\begin{array}{rcl}}
\newcommand{\earay}{\end{array}}
\newcommand{\barray}{\begin{array}{rcl}}
\newcommand{\earray}{\end{array}}

\newcommand\dela[1]{}

\newcommand{\bcase}{\begin{cases}}
\newcommand{\ecase}{\end{cases}}

\newcommand\MS{\mathfrak{S}}

\newcommand\MF{\mathfrak{F}}

\newcommand\del[1]{}

\del{

\newcommand\del[1]{}
}

\def\vt{\lambda}

\def\eps{\varepsilon}

\newcommand{\bu}{\bar u_m}

\newcommand{\lk}{\left}
\newcommand{\lqq}{\lefteqn}
\newcommand{\rk}{\right}
\newcommand{\la}{{\langle}}
\newcommand{\ra}{{\rangle}}

\newcommand{\LL}{L} 

\newcommand{\ep} {\varepsilon }

\newcommand{\be} {\begin{enumerate} }
\newcommand{\ee} {\end{enumerate} }
\newcommand{\BB} {E}

\newcommand{\CO}{{{ \mathcal O }}}

\newcommand{\CE}{{{ \mathcal E }}}
\newcommand{\CC}{{{ \mathcal C }}}

\newcommand{\CT}{{{ \mathcal T }}}
\newcommand{\MT}{{{ \mathfrak T }}}
\newcommand{\Mr}{{{ \mathfrak R }}}

\newcommand{\CA}{{{ \mathcal A }}}

\newcommand{\CH}{{{ \mathcal H }}}
\newcommand{\CS}{{{ \mathcal S }}}
\newcommand{\CG}{{{ \mathcal G }}}

\newcommand{\CB}{{{ \mathcal B }}}

\newcommand{\CM}{{{ \mathcal M }}}
\newcommand{\CP}{{{ \mathcal P }}}
\newcommand{\BF}{{{ \mathbb{F} }}}
\newcommand{\CF}{{{ \mathcal F }}}

\newcommand{\CN}{{{ \mathcal N }}}

\newcommand{\RR}{{\mathbb{R}}}

\newcommand{\WW}{{\mathbb{W}}}

\newcommand{\NN}{\mathbb{N}} 

\newcommand{\PP}{{\mathbb{P}}}

\newcommand{\EE}{ \mathbb{E} }

\newcommand{\DEQS}{\begin{eqnarray*}}
\newcommand{\EEQS}{\end{eqnarray*}}
\newcommand{\DEQSZ}{\begin{eqnarray}}
\newcommand{\EEQSZ}{\end{eqnarray}}
\newcommand{\DEQ}{\begin{eqnarray}}
\newcommand{\EEQ}{\end{eqnarray}}

\documentclass[11pt]{amsart}
\usepackage{color}
\usepackage{amsmath}   
\usepackage{amsfonts,amssymb, color}

\usepackage{nicefrac}
\usepackage{latexsym}
\input colordvi

\usepackage{enumitem}

\usepackage{amsfonts}

\theoremstyle{plain}

\newtheorem{theorem}{Theorem}[section]
\newtheorem{notation}{Notation}[section]
\newtheorem{claim}{Claim}[section]
\newtheorem{lemma}[theorem]{Lemma}
\newtheorem{corollary}[theorem]{Corollary}
\newtheorem{hypo}[theorem]{Assumption}
\newtheorem{assumption}[theorem]{Assumption}

\newtheorem{definition}[theorem]{Definition}
\newtheorem{remark}[theorem]{Remark}

\newtheorem{proposition}[theorem]{Proposition}

\numberwithin{equation}{section}

\newcommand\red[1]{{\color{red}#1}}

\numberwithin{equation}{section} \allowdisplaybreaks

\begin{document}

\title[The 1D stochastic Keller--Segel model]{The one-dimensional stochastic
Keller--Segel model with time-homogeneous spatial Wiener processes}

\author{Erika Hausenblas}
   \address{%
   Department of Mathematics,
	Montanuniversitaet Leoben,
	Austria.}

\author[Debopriya Mukherjee]{Debopriya Mukherjee}

\address{%
School of Mathematics and Statistics\\
The University of New South Wales, Sydney 2052, Australia.}


\author{Thanh Tran}
\address{%
School of Mathematics and Statistics\\
The University of New South Wales,
Sydney 2052, Australia.}

\date{\today}
\thanks{The first two authors of the paper are supported by Austrian Science Foundation, project number P 32295. The first author is also supported by Science Visiting Research Fellowship, UNSW Sydney, during the time of her visit to UNSW. The last two authors are partially supported by ARC DP160101755.
}

\del{we study a stochastic version of the classical Patlak--Keller--Segel system
under homogeneous Neumann boundary conditions on an interval $[0,1]$.
Since we are modelling an intrinsic noise,  it is interpreted in the Stratonovich sense.
Given $T>0$, we prove the existence of a martingale solution on $[0,T]$.}

\begin{abstract}
Chemotaxis is a fundamental mechanism of cells and organisms, which is responsible for attracting microbes to food, embryonic cells into developing tissues, or immune cells to infection sites. Mathematically chemotaxis is described by the Patlak--Keller--Segel model. This macroscopic system of equations is derived from the microscopic model when limiting behaviour is studied. However, on taking the limit
and passing from the microscopic equations to the macroscopic equations, fluctuations are neglected. Perturbing the system by a Gaussian random field restitutes the inherent randomness of the system. This gives us the motivation to study the classical Patlak--Keller--Segel system perturbed by random processes.

We study a stochastic version of the classical Patlak--Keller--Segel system
under homogeneous Neumann boundary conditions on an interval $\CO=[0,1]$.
In particular, let $\mathcal{W}_1$, $\mathcal{W}_2$ be two time-homogeneous spatial Wiener
processes over a filtered probability space $\mathfrak{A}$. Let $u$ and $v$ denote the cell density and concentration of the chemical signal.
We investigate the coupled system
\DEQS
 \lk\{ \barray d {u} - \lk( r_u\Delta   u- \chi \Div( u\nabla v)\rk)\, dt =u\circ  d\mathcal W_1, 
\\
d{v} -(r_v \Delta v  -\alpha v)\, dt = \beta u \, dt+ v\circ d\mathcal W_2,
\earray\rk.
\EEQS
with initial conditions $(u(0),v(0))=(u_0,v_0)$.
The positive  terms $r_u$ and $r_v$  are the diffusivity of the  cells and chemoattractant, respectively,
the positive value  
 $\chi$ is the chemotactic sensitivity,
  $\alpha\ge0$ is the so-called damping constant.
  The noise is interpreted in the Stratonovich sense.
Given $T>0$,   we will prove the existence of a martingale solution on $[0,T]$.
\end{abstract}

\maketitle



\textbf{Keywords and phrases:} {Chemotaxis, Keller-Segel model, Stochastic Partial Differential Equations, Stochastic Analysis, Mathematical Biology}

\textbf{AMS subject classification (2002):} {Primary 60H15, 92C17,  35A01;
Secondary 35B65, 35K87, 35K51, 35Q92.}


\section{Introduction}
Chemotaxis is defined as the oriented movement of cells (or an organism) in response to a
chemical gradient. Many sorts of motile cells undergo chemotaxis. For example, bacteria and many
amoeboid cells can move in the direction of a food source. In our bodies, immune cells like
macrophages and neutrophils can move towards invading cells. Other cells, connected with the
immune response and wound healing, are attracted to areas of inflammation by chemical signals.
%

Theoretical and mathematical modeling of chemotaxis dates back to the works of {Patlak  \cite{patlak} in the 1950s}  and Keller and Segel \cite{KSS} in the 1970s.
Let  $A=\Delta$ be the Laplace operator with Neumann boundary conditions (see equation \eqref{eqn:4.3}). The simplest form of the model is
%
\[
\lk\{ \barray
\dot{u}(t)&=&  r_u A u(t) 
-\chi \mbox{div} \big( u(t)\nabla v(t)\big),
\\
\dot{v}(t) &=& r_v A v(t)+\beta u(t) 
-\alpha v(t),
\earray\rk.
\quad t\in \RR_0^+,\phantom{\big|}
\]
with initial conditions $(u(0,x),v(0,x))=(u_0(x),v_0(x))$, where $x\in \CO=[0,1]$.
Here, $u$ denotes the cell density, and $v$ is the concentration of the chemical signal.
The positive terms $r_u$ and $r_v$  are the diffusivity of the cells and chemoattractant, respectively.
The positive constant 
 $\chi$ is the chemotactic sensitivity.
In the signal concentration model,  $\alpha\ge0$ is the so-called damping constant.
For more details we refer to the surveys by Horstmann \cite{horstmann1,horstmann2}, Hillen and
Painter \cite{hillen1}, and  Bellomo et.\ al.\cite{bellomo1}, and the work of Biler \cite{biler2}, or Perthame \cite{perthame}.

This macroscopic system of equations is derived from the limiting behaviour of the microscopic model; see, e.g.\ \cite{stevens1}. Here, one considers the density of the population modelled by  
Fick's laws, in which the reactions are viewed as functions of the density of cells and the concentration of the chemoattractant.
From the microscopic perspective, one interprets the movement of species populations as a result of microscopic irregular movement of single members of the considered population. In the
derivation of the above macroscopic equations, fluctuations around the mean
value are neglected.
For a more realistic model, it is necessary to consider essential features of the natural environment which are non-reproducible. Hence, the model should include random spatio-temporal
forcing.

Another source of randomness is due to the environment or the fluctuation of parameters. Here, the constant $\alpha$ or $\chi$ may fluctuate randomly in time. Biologists (see \cite{noise1} 
or \cite{noise2}) distinguish
between internal (or intrinsic) noise caused by the irregular movement of the cells  and external noise caused by a random environment.
An appropriate mathematical approach to establish more realistic models is the incorporation of randomness in the system.

In this article, we perturb the density of cells $u$ and concentration of chemoattractant $v$ by
time-homogeneous spatial  Wiener processes, i.e.\ an infinite dimensional Wiener process
depending on the time and space variables (For a brief discussion on different  approaches to
infinite dimensional Wiener processes we refer to \cite{dalang}).  The perturbation is
multiplicative, and the integral is interpreted in the Stratonovich sense. As our main result, we show the existence of a martingale solution. We want to emphasise here that there are several obstacles to overcome in the stochastic case, compared to the deterministic one. Due to the nonlinearity in the system,
one has to truncate nonlinear terms and use stopping time arguments. Here, one can only rely on the entities which are controlled by a Lyapunov functional. This  essentially means that in the proof we are not able to work in spaces with high integrability or high regularity.
Hence, we are unable to show the necessary estimates which are required by the Banach Fixed
Point Theorem. Thus we construct an integral operator, and use compactness and a
 stochastic version of the Schauder--Tychonoff Fixed Point Theorem to obtain a solution. More
precisely,
we prove a stochastic version of the Schauder--Tychonoff Fixed Point Theorem which is
specific to our problem.
In this manner, we achieve only the existence of a martingale solution, but not the uniqueness of the solution.

The authors are not aware of any work which treats the simple Keller--Segel model with noise.
We can only find the very recent work \cite{tusheng},  where the deterministic Keller--Segel model is coupled with the stochastic Navier–Stokes equations.

\del{We are interested in the following system
\DEQSZ\label{chemonoise_anfang}
 d {u} & +& \lk( r_u\Delta   u+ \chi \Div( u\nabla v)\rk)\, dt =u \circ d\mathcal W_1, 
\\
d{v}& +&(r_v \Delta v  +\alpha v)\, dt = \beta u \, dt+u\circ  d\mathcal W_2,
\EEQSZ
where  $\CO$ is either a bounded domain with $C^\infty$ boundary, or $\CO=[0,1]$,  $A=\nabla^2$ is the Laplace operator with  Neumann boundary conditions.
The initial conditions are given by $u(0)=u_0$ 
and $v(0)=v_0$. 
}

\section{The Main Result}
%
Let $\CH_1$ and $\CH_2$ be two Hilbert spaces and let $T>0$; later on in Assumption \ref{wiener} we will see that $\CH_1$ and $\CH_2$ are some Bessel potential spaces.
Let $\mathfrak{A}=(\Omega, \CF,\BF,\PP)$ be a complete probability space
and $\BF=(\CF_t)_{t\in[0,T]}$ be a filtration satisfying the usual conditions i.e.,
\begin{itemize}
\item[(i)] $\mathbb{P}$ is complete on $(\Omega, \CF)$,
\item[(ii)] for each $t\geq 0$, $\CF_t$ contains all $(\CF,\mathbb{P})$-null sets and
 \item[(iii)] the filtration $(\CF_t)_{t\in[0,T]}$ is right-continuous.
  \end{itemize}
Let $\mathcal W_1$ and $\mathcal W_2$ be two cylindrical Wiener processes defined over
$\mathfrak{A}$ on Hilbert spaces $\CH_1$ and $\CH_2$, respectively. (This means these Wiener
processes are defined using orthonormal bases of~$\CH_1$ and $\CH_2$.)
Let us define the Laplacian with the Neumann boundary conditions by
\begin{equation}
\label{eqn:4.3} \left\{
\begin{array}{ll}
D(A) &:= \{ u \in H^2(0,1):u_x(0)=u_x(1)=0 \},\cr
A u&:=\Delta u= u_{xx}, \quad u \in D(A),
\end{array}
\right.
\end{equation}
and let us recall that the positive terms $r_u$ and $r_v$  are the diffusivity of the cells and chemoattractant, respectively.
We consider the following system of equations in $\CO=[0,1]$
\DEQSZ\label{chemonoisestrat}
 \lk\{ \barray d {u} - \lk( r_u\DeltaA   u- \chi \Div( u\nabla v)\rk)\, dt =u\circ  d\mathcal W_1,
\\
d{v} -(r_v \DeltaA v  -\alpha v)\, dt = \beta u \, dt+ v\circ d\mathcal W_2,
\earray\rk.
\EEQSZ
with the initial conditions given by $(u(0),v(0))=(u_0,v_0)$. 
The first equation of system \eqref{chemonoisestrat} is nonlinear
in $u$ and $v$, and the second equation is linear in $u$ and $v$.

\medskip

Since we model an intrinsic noise,  the stochastic integral above is interpreted in
the Stratonovich sense. Since the white noise is an approximation of a continuously fluctuating noise with finite memory being much shorter than the dynamical timescales, the representation of the stochastic integral as a Stratonovich stochastic integral is appropriate.
For a detailed explanation of the Stratonovich integral, we refer to the book by Duan and Wang \cite{duan} or to the original work of Stratonovich \cite{stratonovich}.
For completeness we add the definition of strong solutions to system \eqref{chemonoisestrat}. For a Banach space $E$, the space $C_b^{0}([0,T];E)$ is the set of all continuous and bounded functions $u:[0,T]\to E$.
\begin{definition}
We call a pair of processes $(u,v)$ a strong solution to system \eqref{chemonoisestrat} if $u:[0,T] \times \Omega \rightarrow L^2(\CO)$ and $v:[0,T] \times \Omega \rightarrow
H^1(\CO)$
 are $\BF$--progressively  measurable processes\footnote{The process $\xi:[0,T]\times \Omega\to \mathbb{X}$ is said to be progressively measurable over a probability space $(\Omega,\CF,(\CF_t)_{t\in[0,T]},\PP)$ if, for every time $t\in[0,T]$, the map $ [0,t]\times \Omega \to \mathbb {X} $ defined by $(s,\omega )\mapsto \xi_{s}(\omega )$ is $\mathrm {Borel} ([0,t])\otimes {\mathcal {F}}_{t}$-measurable. This implies that $\xi$ is $(\CF_t)_{t\in[0,T]}$-adapted.} such that $\PP$-a.s.\ $(u,v)\in C_b^0([0,T];L^2(\CO)) \times C_b^0([0,T];H^1(\CO))$ and satisfy for all $t\in [0,T]$ and $\PP$--a.s. the integral equation
\begin{align*} 
u(t)
& =e^{t \,r_u\DeltaA} u_0
-\chi\int_0^ t e^{(t-s)r_u\DeltaA}  \Div( u (s)\,\nabla v(s)) \, ds
%
+\int_0^ t e^{(t-s)r_u\DeltaA} u(s) \circ d\mathcal W_1(s),\\
v(t)
& =e^{t(r_v\DeltaA-\alpha I)} v_0
+\beta \int_0^ t e^{(t-s)(r_v\DeltaA-\alpha I)} u(s)\, ds + 
\int_0^ t e^{(t-s)(r_v \DeltaA-\alpha I)} v(s)\circ  d\mathcal W_2(s).
\end{align*}
\end{definition}
%
Since we use compactness methods to show the existence of the solution,
the given probability space will be lost, and we have to construct another probability space.
In this way, we will obtain a martingale solution, which is also known as a weak solution in the probabilistic sense.
%
%
\begin{definition}\label{Def:mart-sol}
A {\sl martingale solution}   to the problem
\eqref{chemonoisestrat} is a tuple
\(
\left(\mathfrak{A}, 
(\mathcal W_1,\mathcal W_2), \,(u,v)\right)
\)
such that
\begin{enumerate}[label={  (\roman*)}]
\item  $\mathfrak{A}=(\Omega ,{{\mathcal{F}}},{\mathbb{F}},\mathbb{P})$ is a complete filtered
probability space with a filtration ${\mathbb{F}}=(\mathcal{F}_t)_{t\in[0,T]}$ satisfying the
usual conditions,
\item $(\mathcal W_1,\mathcal W_2)$  is a pair of cylindrical Wiener processes over $\mathfrak{A}$ on some given Hilbert space $\CH_1\times \CH_2$,
\item $u:[0,T]\times \Omega \to L^2(\CO)$ and $v:[0,T]\times \Omega \to H^1(\CO)$   are  two ${\mathbb{F}}$-progressively
measurable processes being a strong solution to \eqref{chemonoisestrat} over $\mathfrak{A}$.
\end{enumerate}
\end{definition}

Let us consider the complete orthonormal system of the underlying Lebesgue space~$L^2(\CO)$  given by the trigonometric functions (see \cite[p.\ 352]{Garrett1989}) 
\begin{equation}\label{ONS}
\psi_k(x)=
\begin{cases}
{\sqrt{2}} \, \sin\big(2\pi{k} x\big) &\!\!\text{if } k\geq 1,\,x\in\CO, \\
{1} &\!\!\text{if } k = 0,\, x\in\CO\,, \\
{\sqrt{2}}\, \cos\big(2 \pi{k} x\big) &\!\!\text{if } k \leq - 1, x\in\CO.
\end{cases}
\end{equation}
For the proof of the existence of the solution, the Wiener
perturbation needs to satisfy regularity assumptions given in the next hypotheses.
\begin{hypo}\label{wiener}
Let $\CH_1$ and $\CH_2$ be two Bessel potential spaces such that
the embeddings  $\iota_1:\CH_1\hookrightarrow L^2(\CO)$ and $\iota_2:\CH_2\hookrightarrow
H^1(\CO)$ are  Hilbert--Schmidt operators. The Wiener processes $\mathcal W_1$ and $\mathcal W_2$ are two cylindrical processes on  $\CH_1$ and $\CH_2$.
 \end{hypo}

\begin{remark}\label{rem wiener}
%
As an example we can take $\CH_1=H^{\delta_1}(\CO)$ for $\delta_1>1$, and
$\CH_2=H^{\delta_2}(\CO)$ for $\delta_2>2$ where $H^{\delta_i}(\CO)$, $i=1,2$, is to be
defined in Notation~\ref{Hdelta}. The Wiener processes $\mathcal W_1$ and $\mathcal W_2$  are
then given by
\[
\mathcal W_1(t,x)=\sum_{k\in\mathbb{Z}}
\psi^{(\delta_1)} _k(x)\beta^{(1)}_k(t) \quad\text{and}\quad
\mathcal W_2(t,x)=\sum_{k\in\mathbb{Z}}
\psi^{(\delta_2)} _k(x)\beta^{(2)}_k(t)
\]
where $\psi_k^{(\delta_i)}:=(1+({2\pi k})^2)^{-\delta_i/2}\psi_k$, $i=1,2$, with $\psi_k$ defined by
\eqref{ONS}, and where~$\{\beta^{(i)}_k:k\in\mathbb{Z}\}$, $i=1,2,$ are two families of mutually
independent identical distributed standard Brownian motions. We note
that~$\{\psi^{(\delta_i)}_k:k\in\mathbb{Z}\}$ forms an  orthonormal system in
$H^{\delta_i}(\CO)$, $i=1,2$. We will also write
$$
\mathcal W_i(t,x)=\sum_{k\in\mathbb{Z}} \lambda^i_k
\psi _k(x)\beta^{(i)}_k(t),\quad i=1,2,
$$
where, for~$i=1,2$, the sequence $\{\lambda^i_k:k\in\mathbb{Z}\}$, $\lambda^i_k:=(1+{(2\pi
k)^2})^{-\delta_i/2}$, is non-negative with $\lambda^i_k=\lambda^i_{-k}$ for
all $k\in\ZZ$.
For the later course of analysis, for each $i=1,2$, we define
$\gamma_i:=\sum_{k\in\ZZ}\,(\lambda^i_k)^2$.
\end{remark}

Since the solution $(u,v)$ represents the cell density and concentration of the chemical signal,
$u$ and $v$ are non-negative. This immediately implies that the initial conditions $u_0$ and
$v_0$ are non-negative. Besides, one needs to impose some more regularity assumptions on $u_0$
and $v_0.$ 
\begin{hypo}\label{init}
Let $u_0\in L^ 2 (\CO)$ and $v_0\in H^ 1 (\CO)$ be two random variables over $\mathfrak{A}$ such that
\begin{enumerate}[label=(\alph*)]
  \item $u_0\ge 0$ and $v_0\ge 0$;
  \item $(u_0,v_0)$ is  $\CF_0$--measurable;
  \item  $\EE\left[|u_0|_{L^2}^2\right]<\infty$ and $\EE\left[|\nabla v_0|_{L^2}^2\right]<\infty$.
\end{enumerate}
\end{hypo}

\begin{notation}\label{Hdelta}
Let $A_1$ be the positive Laplace operator~$-\Delta$ (as an operator defined on~$L^2(\RR)$)
restricted to functions defined on~$\CO$, namely,
\[
A_1 :=-\Delta, \quad
D(A_1):= H^2(\CO)\cap H^1_0(\CO),
\]
and let $(\psi_j,\rho_j)$ be the eigenfunctions and eigenvalues of $A_1$.
The  Bessel potential space $H^\kappa(\CO)$ is defined by
\[
H^\kappa(\CO)=\{u=\sum_{j}a_j \psi_j \in L^2(\CO):\|u\|_{H^{\kappa}(\CO)}
=\Big(\sum_{j}a_j^{2} \rho_j^{2\kappa}\Big)^{1/2}<\infty.
\}.
\]
\end{notation}
\begin{notation}
Let us  define by
$\LLa(\CO)$ the Zygmund space (see  \cite[Definition 6.1, p.\ 243]{bennet}) consisting of all Lebesgue-measurable functions $f:\CO\to \RR$ for which $\int_\CO|f(x)|\log(|f(x)|)^+\, dx<\infty$. This space is equipped with the norm
\[|f|_{\LLa}:= \int_0^1 f^{\ast}(t) \log(\dfrac{1}{t})dt,\]
where $f^{\ast}$ is defined by
\[ f^{\ast}(t)=\inf\{\lambda: \delta_f(\lambda)\leq t  \}, \quad t \geq 0.
\]
Here $\delta_f(\lambda)$ is the Lebesgue measure of the set $\{x \in \CO: |f(x)| > \lambda \}, \,\, \lambda \geq 0.$ We note that $f \in \LLa$ iff $\int_{\CO} |f(x)|\log (2+|f(x)|)\,dx < \infty$ (see p. 252 of \cite{bennet}). Note also that by Theorem 6.5 p.\ 247 in \cite{bennet} we have that $\LLa(\CO) \hookrightarrow L^1(\CO)$.
\end{notation}
We now present the main result of the paper.
\begin{theorem} \label{main.thm}
Let Assumption \ref{wiener} be satisfied. 
Then for all initial conditions $(u_0,v_0)\in L^2(\CO)\times H^1(\CO)$ satisfying Assumption
\ref{init} and for all $T>0$, there exists a martingale solution~$(\mathfrak{A},(\mathcal
W_1,\mathcal W_2),(u,v))$ to the system \eqref{chemonoisestrat}. Moreover, there exists a
positive constant~$C=C(T, u_0,v_0)$ such that
\DEQSZ\label{eee}
\EE \Big[\sup_{0\le s\le T}|u(s)|_{L^1}\Big]+\EE \Big[\sup_{0\le s\le T}|\nabla v (s)|_{L^2}^2\Big]+\EE\Big[\int_0^T  |\nabla v(s)|_{H^1}^2\, ds \Big]\le C.
\EEQSZ
If in addition $\EE\big[|u_0|_{\LLn}^p\big]<\infty$ and $
\EE\big[|\nabla v_0|_{L^2}^{2p}\big]<\infty$ for~$p\ge1$, then there exists a constant
$C=C(p,T, (u_0,v_0))>0$ such that
\DEQSZ\label{eeep}
\EE \Big[\sup_{0\le s\le T}|u(s)|^p_{L^1}\Big]
+\EE \Big[\sup_{0\le s\le T}|\nabla v(s) |_{L^2}^{2p}\Big]+\EE\Big( \int_0^T |\nabla v(s)|_{H^1}^2\, ds\Big)^p\le C.
\EEQSZ
\end{theorem}

The proof of this main result is presented in the next section.
As a consequence of the theorem, we obtain the following corollary on regularity, the
proof of which is moved to Appendix~\ref{appb}.

\begin{corollary}\label{cor.main}
Let $p>4$ and let $\beta,\delta>0$ be such that $\beta+\delta<\nicefrac{1}{2}-\nicefrac{1}{p}$ and $\varrho=\nicefrac 23$. Let  Assumption \ref{wiener} be satified.
Then, if $(u_0,v_0) \in L^2(\CO)\times H^1(\CO)$ satisfies  Assumption  \ref{init} along with
$\EE\left[|u_0|_{\LLn}^p\right]<\infty$ and $ \EE\left[|\nabla v_0|_{L^2}^{2p}\right]<\infty$,
then the martingale solution $(u,v)$ to the system \eqref{chemonoisestrat} satisfies
$$
\EE\,\|u-e^{(\cdot)(r_u\DeltaA +\gamma I) }u_0\|^{\varrho p}_{C^{\beta}_b([0,T];H^{2\delta})}<\infty
\quad \hbox{and} \quad
\EE\,\|\nabla v-e^{(\cdot)(r_v\DeltaA -\alpha I) }\nabla v_0\|^{ p}_{C^{\beta}_b([0,T];H^{2\delta})}<\infty.
$$
\end{corollary}

Since the proof of Theorem~\ref{main.thm} is quite technical and sophisticated, in the next
section we split it into five main steps.
In Step I, we consider the truncated system and show the existence of solutions by using the
Schauder--Tychonoff Type Fixed Point Theorem. In this way we obtain a family of solutions $\{
(u_n,v_n):n\in\NN\}$ to the truncated system depending on the truncation parameter $n$.
In Step II, we show that the solution of the truncated system is non-negative, provided
the initial conditions are. In Step III, we construct a new pair of processes on the whole
interval $[0,T]$
such that they are a solution of the chemotaxis system up to a stopping time, and satisfy
the heat equation (with lower order terms). Then using an appropriate Lyapunov functional, we
derive in Step IV uniform estimates in the $L^1$-norm of the family of solutions.
In the final step (Step V), we prove the existence of solutions $\PP$-a.s to the original
problem on the whole time interval $[0,T]$.

Due to the fact that we cannot use higher regularity or integrability as given by the
Lyapunov functional, we can only prove existence. However, we assume that one can find local solutions being unique.
%
%

%


\section{Proof of Theorem \ref{main.thm}}\label{proof}

\subsection{Conversion between the It\^o and Stratonovich form of SPDEs} \label{sec.itostrat}
Before starting with the actual proof, we first outline shortly the difference between the Stratonovich integral and the It\^o integral. A real fluctuating forcing term has a finite amplitude and a
finite-timescale. In this way, the so-called  white noise is an idealisation of the delta-function-correlated
noise, and, as mentioned already in the introduction, the white noise can be interpreted as an approximation of a continuously fluctuating noise with finite memory being much shorter than the dynamical timescales.
Wong and Zakai \cite{wongzakai} investigate the convergence of ordinary differential equations (ODEs) which, in particular, involve
piecewise linear approximations of a one-dimensional Wiener process. They
show that the solutions to these ODEs converge almost surely to a solution
of a specific Stratonovitch stochastic differential equation (SDE) and not an
It\^ o SDE. In particular, they show that the proper representation of the white noise as a
stochastic integral in the Stratonovich sense is appropriate. However, a drawback of the Stratonovich
stochastic integral is that one does not have the Burkholder–Davis–Gundy inequality at their
disposal. Hence we use the relation between the Stratonovich and It\^o integral to change the
integral into an It\^o stochastic integral. Therefore, we present the relation between the two stochastic integrals shortly.
For a short discussion from the viewpoint of physicists, we refer to \ \cite{physics2} and
\cite{physics1}.

Let $\mathfrak{A}=(\Omega, \CF,\BF,\PP)$ be a complete probability space
and $\BF=(\CF_t)_{t\in[0,T]}$ be a filtration satisfying the usual conditions.
Let $H$ be a Hilbert space and
\begin{gather*}\label{def.MA.H}
M_{\MA}^2(0,T;H):=\Big\{ \xi:[0,T]\times \Omega\to H \ \Big| \ \xi \text{ is progressively measurable over $\mathfrak{A}$ }
\notag\\
\qquad \qquad \text{ and} \quad  \EE\Big[ \sup_{0\le s\le T}
|\xi(s)|^2_H \Big]<\infty \Big\}.
\end{gather*}
This space is equipped with the norm
\begin{equation}\label{equ:M nor}
\|\xi\|_{M_{\MA}^2} :=
\left( \EE\Big[ \sup_{0\le s\le T} |\xi(s)|^2_H \Big] \right)^{1/2}.
\end{equation}
The cylindrical Wiener process $\mathcal{W}$ is time-homogeneous spatial Wiener processes (see Assumption \ref{wiener} and Remark \ref{rem wiener}) on a given Hilbert space $\CH$ 
which can be written as 
$$
\mathcal{W}(t,x)=\sum_{k\in\mathbb{N}}  \phi_k(x)\beta_k(t), \quad t\in[0,T],
$$
where  $\{\phi_k:k\in\mathbb{Z}\}$ is an orthonormal basis in $\CH$ and  
$\{ \beta_k:k\in\ZZ\}$
is a sequence of
independent standard Brownian motions. %
A stochastic integral can be interpreted in different ways, e.g., as a Stratonovich integral or as an It\^o integral.
For the time being, let us assume that
$\mathcal{W}$ is a cylindrical 
Wiener process on $\CH$, where, for simplicity, we  assume that the embedding
$\CH\hookrightarrow L^2(\CO)$ is a Hilbert--Schmidt operator.
Let $\pi=\{t^\pi_0=0<t^\pi_1<\ldots<t^\pi_{N^\pi}=T\}$ be a partition of $[0,T]$ and let
$|\pi|=\displaystyle\max_{1\le k\le N^\pi}(t^\pi_{k}-t^\pi_{k-1})$.
The Stratonovich integral of a process~$\xi$ is defined by
$$
\int_0^ t \xi(s) \circ d\mathcal{W}(s) := \lim_{|\pi|\to 0}  \sum_{k=1}^{N^\pi}
 \  \xi_k \ \big( \mathcal{W}(t^\pi_k\wedge t)-\mathcal{W}({t^\pi_{k-1}\wedge t}) \big)
$$
with $\xi_k=\xi((t^\pi_k\wedge t+t^\pi_{k-1}\wedge t)/2)$, while
the It\^o integral is defined by
$$
\int_0^ t \xi(s) d \mathcal{W}(s) := \lim_{|\pi|\to 0}  \sum_{k=1}^{N^\pi}  \
\xi_{k} \  \big( \mathcal{W}(t^\pi_k\wedge t)-\mathcal{W}({t^\pi_{k-1}\wedge t}) \big)
$$
with $\xi_k=\xi(t^\pi_{k-1}\wedge t)$.

Let us consider the following stochastic partial differential equation (SPDE) perturbed by  a one-dimensional Wiener process, i.e.\ $\CH=\RR$. The integral is expressed in the sense of the Stratonovich integral, i.e.,
\DEQSZ\label{strat1}
d\xi(t)=A \xi(t)\, dt + \sigma(\xi(t))\circ d \mathcal{W}(t),\quad \xi(0)=\xi_0.
\EEQSZ
Here $A$ is an analytic operator on the Hilbert space $H:=L^2(\CO)$ and $\sigma:H\to
L_{HS}(\CH,H)$ where $L_{\text{HS}}(H_1,H_2)$ denotes the space of linear operators
from~$H_1$ to~$H_2$ equipped with the Hilbert--Schmidt norm.
As mentioned in the introduction, the Stratonovich integral is the natural one for modelling fluctuations.
However, the disadvantage is that the Stratonovich integral is not a martingale; hence, neither
the stochastic It\^o--isometry nor the Burkholder--Davis--Gundy inequality is available.
Nevertheless, one can reformulate~\eqref{strat1} into an SPDE with an
It\^o differential, by adding a correction term. The reason for the correction term is that, in
the definition of the Stratonovitch integral,
the random variable~$\xi_k=\xi((t^\pi_{k-1}+t^\pi_{k})/2)$
and the increment
$\mathcal{W}({t^\pi_{k}})-\mathcal{W}({t^\pi_{k-1}})$ are not independent. The random variable
$\xi_k$ contains some information about the difference
$\mathcal{W}({t^\pi_{k}})-\mathcal{W}({t^\pi_{k-1}})$ because~$\xi$ is evaluated at the midpoint
of the interval. In the case of the It\^o integral one
chooses $\xi_k=\xi(t_{k-1}^\pi)$, the value of $\xi$ at the left endpoint of the interval.
This random variable $\xi_k$ is independent of $\mathcal{W}({t^\pi_{k}})-\mathcal{W}({t^\pi_{k-1}})$.

The correction term can be calculated explicitly  (see  \cite[p.\ 65, Section 4.5.1]{duan}), i.e., the equivalent It\^o equation of \eqref{strat1} is given by
\begin{align}\label{ito eqn noise}
d\xi(t)=A \xi(t)\, dt +\frac 12 D_\xi\sigma(\xi(t))\, \sigma(\xi(t))\, dt + \sigma(\xi(t)) d \mathcal{W}(t),\quad \xi(0)=\xi_0.
\end{align}
Here, $D_{\xi}(\sigma(\xi))$ denotes the Fr\'echet derivative of $\sigma$ with respect to $\xi$.

\medskip

Let us assume that the Wiener process is cylindrical in $H^\delta(\CO)$, $\delta > 1$. Then, the Wiener process is infinite dimensional and  it can be written as a sum of infinitely many scalar Wiener processes, i.e.\
 $\CW(t,x)=\sum_{k\in\ZZ} \psi_k^{(\delta)}(x)\,\beta _k(t)$, where $\{\beta_k:k\in\ZZ\}$ is a
family of independent scalar-valued Wiener processes and 
 $\psi^{(\delta)}_k:=(1+(2\pi k)^2)^{-\delta/2}\psi_k$. Here $\{ \psi_k:k\in\mathbb{Z}\}$  is the
orthonormal basis given in \eqref{ONS}.
  Elementary calculations reveal for $\sigma(\xi)=\xi$
\begin{align}\label{correction term}
 \sum_{k\in\ZZ}  D_\xi \sigma(\xi)[\psi_k^{(\delta)}]\, \sigma(\xi)[\psi_k^{(\delta)}]=\sum_{k\in\ZZ} \xi  \,  \psi^{(\delta)}_k\psi^{(\delta)}_k.
\end{align}
We call $\sigma_k : H \rightarrow L_{HS}(H)$ by $\sigma_k(\xi)=\lambda_k \xi \psi_k$.
Indeed, denoting $(1+(2\pi k)^2)^{-\delta/2}$ by $\lambda_k$ and taking into account the relations $\lambda_k=\lambda_{-k}$, $\psi_k^2 +\psi_{-k}^2 =2$ and $\gamma=\sum_{k\in\ZZ} \lambda^2_k$, we infer
\begin{align*}
\sum_{k\in\ZZ}  D_\xi\sigma_k(\xi)\, \sigma_k(\xi)&=\sum_{k\in\ZZ} \xi  \psi^{(\delta)}_k\psi^{(\delta)}_k
=\sum_{k\in\NN} \lambda^2_k \psi_k\psi_k\xi + \sum_{k\in\NN} \lambda^2_{-k} \psi_{-k}\psi_{-k} \xi +\lambda_0^2 \xi
\notag\\
&
=2 \sum_{k\in\NN} \lambda^2_k \xi+ \lambda_0^2 \xi
= \sum_{k\in\ZZ} \lambda^2_k \xi=\gamma \xi.
\end{align*}

We shortly give the exact form of the Burkholder--Davis--Gundy inequality which will be useful
during the course of analysis. Given a Wiener process $\mathcal{W}$ being cylindrical on $\CH:=H^\delta(\CO)$,
$\delta>1$, over $\MA$, and a progressively measurable process $\xi\in M^2_\MA(0,T;H^\rho(\CO))$,
$\rho\in[0,\frac 12]$,
let us define $\{Y(t):t\in[0,T]\}$ by
 $$
 Y(t):=\int_0^ t \xi(s)\, d \mathcal{W}(s), \quad t\in[0,T].
 $$
Here, for each $t\in[0,T]$, $\xi(t)$ is interpreted as a multiplication operator acting on the
elements of $\CH$, namely, $\xi:\CH\ni\psi \mapsto \xi\psi\in \CS'(\RR)$.
Since for any $\nu>\frac 12$ and for any~$\varphi\in H^\nu(\CO)$
the product~$\xi(t)\varphi$ belongs to~$H^\rho(\CO)$
(see \cite[p.\ 190]{runst}), we can view~$\xi(t)$ as a linear map from~$H^\nu(\CO)$ into~$H^\rho(\CO)$.
It is shown in~\cite[Theorem~1 (4), p.\ 190]{runst} that
\begin{equation}\label{equ:runst}
|\xi(t)\varphi|_{H^\rho}
\le
C
|\xi(t)|_{H^\rho}
|\varphi|_{H^\nu}
\end{equation}
where the constant $C>0$  is independent of~$\varphi$.
Consequently, for any $p\ge 1$, $\delta > 1 $, and any $\rho\in[0,\frac 12]$
\begin{equation*}\label{BDG}
\EE \Big[\sup_{t \in [0, T]} |Y(t)|^p_{H^\rho}\Big] \leq C_p \, \EE \,\Big[ \int_0^T |\xi(t)|^2_{L_\text{HS}(H^{\delta}, H^\rho )}\, dt\Big]^\frac p2,
\end{equation*}
where $|\cdot|_{L_\text{HS}(H^{\delta},H^{\rho})}$ denotes the Hilbert--Schmidt norm from
$H^{\delta}(\CO)$ to $H^{\rho}(\CO)$. 
First, let us note that for $\delta>1$ there exists a $\nu>\frac 12 $ such that the embedding
$H^{\delta}(\CO)\hookrightarrow H^\nu(\CO)$ is a Hilbert--Schmidt operator.
Using the fact that  $\{\psi_k^{(\delta)}:k\in\ZZ\}$ is an orthonormal basis of $H^{\delta}(\CO)$,
and $\psi_k^{(\delta)}=\lambda_k \psi_k$ we obtain, by using~\eqref{equ:runst},
%
\begin{align}\label{e lhs}
|\xi(s)|^2_{L_\text{HS}(H^{\delta}, H^\rho)}
&
=\sum_{k \in \ZZ} \lk|\xi(s)\psi_k^{(\delta)}\rk|^2_{ H^\rho}
\le C\sum_{k \in \ZZ} \lk|\xi(s)\rk|_{H^\rho}^2\lk|\psi_k^{(\delta)}\rk|^2_{H^{\nu}}
=
C |\xi(s)|^2_{ H^\rho}  \sum_{k \in \ZZ}|\psi_k^{(\delta)}|^2_{ H^\nu}
.
\end{align}
If the embedding $H^{\delta}(\CO) \hookrightarrow H^{\rho}(\CO)$ is supposed to be  a Hilbert--Schmidt,
the right hand side of \eqref{e lhs} is finite
and we obtain
\begin{equation}\label{HSnorm}
\EE \Big[\sup_{t \in [0, T]} |Y(t)|^p_{H^\rho}\Big] \leq C \, \EE \,\Big[ \int_0^T
|\xi(t)|^2_{H^\rho}\, dt \Big]^\frac p2.
\end{equation}
In case $\rho\ge \frac12$, we use the fact that $H^\rho(\CO)$ is an algebra. Therefore,  we obtain
\begin{align}\label{eq lhs fin}
|\xi(s)|^2_{L_\text{HS}(H^{\delta}, H^\rho)}
&
=
\sum_{k \in \ZZ} \lk|\xi(s)\psi_k^{(\delta)}\rk|^2_{ H^\rho}
=|\xi(s)|^2_{ H^\rho} \sum_{k \in \ZZ} \lk|\psi_k^{(\delta)}\rk|^2_{ H^\rho}.
\end{align}
If the embedding  $H^\delta(\CO)\hookrightarrow H^{\rho}(\CO)$ is a Hilbert--Schmidt operator,
then the right hand side of \eqref{eq lhs fin} is finite
and we obtain for $\delta>\rho+1$
\begin{equation}\label{HSnorm.1}
\EE \Big[\sup_{t \in [0, T]} |Y(t)|^p_{H^\rho}\Big] \leq C \, \EE \,\Big[ \int_0^T
|\xi(t)|^2_{H^\rho}\, dt \Big]^\frac p2.
\end{equation}
For further details we refer to the linear noise (see \cite[Example 2.1.2, p.\ 22]{roeckner}).
After this short preparatory step we now start the proof.

\subsection{Proof of Main Result}
\begin{proof}[Proof of Theorem \ref{main.thm}]
%
We show the existence of  a solution to the system written in the It\^o form, i.e., to
\begin{align}\label{chemonoise.1}
 \lk\{ \barray d {u}  - \lk( r_u\DeltaA   u- \chi \Div( u\nabla v)  +\gamma u\rk)\, dt =u \, d\mathcal W_1, 
\\
d{v} -(r_v \DeltaA v  -(\alpha-\gamma) v)\, dt = \beta u \, dt+ v \,d\mathcal W_2. \earray\rk. 
\end{align}
Here, the stochastic differential is interpreted as the It\^o stochastic differentials. Adding the correction term calculated in \eqref{correction term} we see that the equation \eqref{chemonoise.1} is equivalent to \eqref{chemonoisestrat}, where the stochastic integral is interpreted as the Stratonovich integral. The advantage of the It\^o stochastic integral is, that the It\^o integral is a local martingale and we can apply the Burkholder--Davis--Gundy inequality. For simplicity we write $\alpha$ as $\alpha-\gamma$.
In this way we end up with the following system
\begin{align}\label{chemonoise}
 \lk\{ \barray d {u}  - \lk( r_u\DeltaA   u- \chi \Div( u\nabla v)  +\gamma u\rk)\, dt =u \, d\mathcal W_1, 
\\
d{v} -(r_v \DeltaA v  -\alpha v)\, dt = \beta u \, dt+ v \,d\mathcal W_2, \earray\rk. 
\end{align}

As mentioned before, we  have at our disposal the Burkholder--Davis--Gundy maximal inequality. By the definition of $\gamma$ and the remark above (see Section \ref{sec.itostrat}) it follows that the systems \eqref{chemonoisestrat} and \eqref{chemonoise} are equivalent.

The proof of Theorem \ref{main.thm} is carried out in several steps. First, we consider the
truncated system and show that the solution of the truncated system exists depending on the
truncation parameter $n$. The immediate next step is to achieve non-negativity of the solution,
provided the initial conditions are. Then using appropriate Lyapunov functional, we obtain a
uniform estimate in the $L^1$ norm of the family of solutions for the truncated system, by which
we obtain the existence of a global solution.

\begin{steps1}
\item
 {\bf The truncated system:}
Let $\psi \in C_c^{\infty}(\mathbb{R})$ be a cut-off function which satisfies
$$
\psi(x):= \bcase 0 &\mbox{ if } |x|\ge 2,
\\ \in [0,1] &\mbox{ if } |x|\in(1,2),
\\1 &\mbox{ if } |x|\le 1,
\ecase
$$
and let $\psi_n(x):= \psi(x/n)$, $x\in\RR,\,n \in \mathbb{N}$.
In addition, for any progressively measurable process $\eta_1$ and $\eta_2$ over $\mathfrak{A}$ belonging to $\CM_\MA ^2(0,T;L^1(\CO))$ and $\CM_\MA ^2(0,T;H^1(\CO))$ respectively,
let us define for $t \in [0,T]$
\DEQSZ \label{h12.def}
\\\nonumber
h^1(\eta_1,t ):= \sup_{0\le s\le t} |\eta_1(s)|_{L^1},\,
 \, h^2(\eta_2,t ):= \sup_{0\le s\le t}  |\nabla \eta_2(s)|_{L^2}, \,
\mbox{ and }\,  h^3(\eta_2,t):=\int_0^t  |\nabla \eta_2(s)|_{H^1}^2\, ds.
\EEQSZ
Let us consider the truncated system given by
\begin{equation}\label{ppp1.n}
\begin{cases}
& du_n(t) -
\Big[ r_u \DeltaA u_n(t) -\chi \psi_n(h^1(u_n,t))\,\psi_n(h^2(v_n,t))\,\psi_n(h^3(v_n,t))
\\
&\hspace{1.1cm}
\times \Div( u_n(t)\,\nabla v_n(t))+\gamma u_n(t)\Big] dt
=  u_n(t) d \mathcal W_1(t),
 \\
& dv_n(t) - \Big[r_v \DeltaA v_n(t)-\alpha v_n(t)\Big]\,dt =\beta u_n(t)dt+
 v_n(t)d\mathcal W_2(t),
 \\ 
& (u_n(0),v_n(0)) = (u_0,v_0).
\end{cases}
\end{equation}
First we show in the following lemma that this truncated system has a {martingale} solution
$(u_n,v_n)$ (in the sense of Definition \ref{Def:mart-sol}) belonging $\PP$-a.s.\ to
$C_b^0([0,T];L^2(\CO))\times C^0_b([0,T];H^1(\CO))$.
%
%
%
%
\begin{lemma}\label{lem.exist.n}
Let us assume that Assumption \ref{wiener} is satisfied.
For each $n \in \mathbb{N}$, and any initial condition $u_0,v_0$ satisfying Assumption \ref{init},  there exists a martingale solution $(\mathfrak{A}_n, (\mathcal{W}^n_1,\mathcal{W}^n_2),\\(u_n,v_n))$, where {$\mathfrak{A}_n=(\Omega_n,\CF_n,(\CF_t^n)_{t\in[0,T]}, \PP_n)$} is a filtered probability space and $(u_n,v_n)$ is a strong solution  to the system \eqref{ppp1.n} such that $(u_n,v_n)$  belongs $\PP_n$--a.s.\ to
$C_b^0([0,T];L^2(\CO))\times C^0_b([0,T];H^1(\CO))$. In addition, for all $n\in\NN$, there exists a constant $C=C(n)>0$ such that
\[
\EE^n \Big[ \sup_{0\le s\le T} |u_n(s)|_{L^2}^2\Big]
+
\EE^n \Big[ \sup_{0\le s\le T} |\nabla v_n(s)|_{L^2}^2 \Big]
\le C(n)\phantom{\Bigg|}.
\]
\end{lemma}
%
%
The proof of Lemma \ref{lem.exist.n} is carried out by
applying the Schauder--Tychonoff Fixed Point Theorem
to the mapping~$\MT_n:\CM_\MA ^2(0,T;L^2(\CO))\to \CM_\MA ^2(0,T;L^2(\CO))$ defined for
each~$n\in\NN$ as follows.
Each process~$\xi\in \CM_\MA ^2(0,T;L^2(\CO))$ is assigned a value~$\MT_n(\xi):=u_n$
and another value~$\Mr(\xi):=v$
where~$(u_n,v)$ is the unique solution to
\begin{equation}\label{ppp1}
\begin{cases}
& du_n(t) -\Big[r_u \DeltaA u_n(t) -\chi \,\psi_n(h^1(\xi,t))\psi_n(h^2(v,t))\,\psi_n(h^3(v,t))
\\
& \qquad \times \Div( \xi (t)\,\nabla v(t))+\gamma u_n(t)\Big] dt=  u_n(t) d \mathcal W_1(t),
\\
& dv(t)-\Big[r_v \DeltaA v(t)-\alpha v(t)\Big]\,dt= \beta \xi(t)dt+v(t) d \mathcal W_2(t),
\\
& (u_n(0),v(0)) = (u_0,v_0).
\end{cases}
\end{equation}

First, let us check that the operators $\MT_n:\CM_\MA ^2(0,T;L^2(\CO))\to \CM_\MA
^2(0,T;L^2(\CO))$ and $\Mr:\CM_\MA ^2(0,T;L^2(\CO))\to \CM_\MA ^2(0,T;L^2(\CO))$ are well
defined, namely we verify that there exists a unique solution $(u_n,v)$ to  system \eqref{ppp1}.
For each $\xi\in  \CM_\MA ^2(0,T;L^2(\CO))$,
the unique existence of $v$ belonging to $L^\infty(0,T;H^1(\CO))\cap L^2(0,T;H^2(\CO))$
can be shown by standard arguments, see e.g. \cite[Theorem 7.2]{pratozab}.
Note that the process $v$ is independent of the cut-off argument; therefore it is
independent of $n\in\NN$. Again, due to \cite[Theorem 7.5, p.\ 197]{pratozab} given $v$ and $\xi$ we can conclude that there exists a unique solution $u_n$ to the first equation in \eqref{ppp1}. This ensures that the operators~$\MT_n$ and~$\Mr$ are well defined.

Next, we investigate some properties of the operator $\MT_n$ which are necessary to formulate a Schauder--Tychnoff Fixed Point Theorem.
In the following we will show that
\begin{itemize}
\item there exists a convex subset $\CC\subset \CM_\MA ^2(0,T;L^2(\CO))$ such that $\MT_n$ maps $\CC$ into itself,
\item  $\MT_n$ maps bounded sets into {precompact sets},
\item $\MT_n$ restricted to bounded sets is continuous.
\end{itemize}
Using the compactness of the 
operator $\MT_n$ defined on a closed convex subset of
$\CM_\MA^2(0,T;L^2(\CO))$, we obtain a  fixed point of $\MT_n$, which is a solution of
\eqref{ppp1.n}.
Since we use compactness and not a contraction principle, this step can be seen as a Schauder--Tychonoff Type Fixed Point Theorem.

Let us start with the first claim.
\begin{claim}\label{claim1}
The following two statements hold.
\begin{enumerate}[label=(\alph*)]
\item
 For any $n\in\NN$ and any initial condition $(u_0,v_0)$ satisfying Assumption \ref{init}, there exists a number $R=R(n,u_0)>0$ such that
$\MT_n: \CC_R \longrightarrow
\CC_R,
$
where
\begin{align}\label{def.CR}
\CC_R:=\{\xi\in \CM^2_\MA(0,T;L^2(\CO)): \|\xi\|_{\CM_\MA(0,T;L^2)}^2\le R\}.
\end{align}
In particular, $R$ is of the form $ C(T)[ \EE|u_0|_{L^2}
+ (2n)^{5/4}]
$.
\item
 For any $n\in\NN$, any initial condition $(u_0,v_0)$ satisfying Assumption \ref{init},  and  for  any $\bar \alpha,\beta\ge 0$ with {$\bar \alpha+2\beta<\frac 12$}, 
\del{\DEQSZ \label{est.u}
\EE \Big[ \sup_{\delta\le s\le t}|(\MT_n \xi )(s)u_0|^2_{{W^\delta_2}} \Big] \le  
C(T)\|\xi\|^2_{\CM_\MA ^2(0,T;L^2)} 
, \quad \forall t \in [0,T].
\EEQSZ
}\begin{itemize}
   \item there exists a  number $r=r(\bar \alpha,T,n)>0$ such that  for any $\xi\in\CM_\MA ^2(0,T;L^2(\CO))$ we have
   \DEQS
\sup_{0\le t\le T}\EE| (\MT_n \xi )(t)-e^{t(r_u\DeltaA+\gamma I)}u_0|^2_{H^{\bar\alpha}} \le r \|\xi \|^2_{\CM_\MA^2(0,T;L^2)},
\EEQS

   \item  there exists a number $\delta=\delta(\bar \alpha,\beta,T,n)>0$ and a number $C=C(\delta,T,n)>0$ such that for any $0< t_1<t_2\le T$ and $\xi\in\CM_\MA ^2(0,T;L^2(\CO))$ we have
 \end{itemize}
\DEQS
\EE| (\MT_n \xi )(t_1)-(\MT_n \xi )(t_2)|_{L^2}^2\le C\, |t_2-t_1|^\delta   \|\xi \|^2_{\CM_\MA^2(0,T;L^2)}.
\EEQS

\end{enumerate}

\end{claim}
\begin{proof}[Proof of Claim \ref{claim1}:]
Let us fix $\xi\in\CM_\MA ^2(0,T;L^2(\CO))$ and let $(u_n,v)$ be a solution to system \eqref{ppp1}.
Firstly, we observe that by the variation of constant formula and the semigroup generated by $B=-r_u A-\gamma I$ with $D(B)=D(A)$, we have
\begin{align}\label{VOCF}
\MT_n\xi(t)=u_n(t)
 &=e^{t(r_u\DeltaA+\gamma I)} u_0
\notag\\
& +\chi \int_0^t  e^{(t-s)(r_u\DeltaA+\gamma I)} \psi_n(h^1(\xi,t))\,\psi_n(h^2(v,t))\,\,\psi_n(h^3(v,s)) \Div( \xi (t)\,\nabla v(s)) \, ds
 \notag\\&\quad+
\int_0^ t e^{(t-s)(r_u\DeltaA+\gamma I)} u_n(s)\, d\mathcal W_1(s),\,\,t \in [0,T].
\end{align}
Secondly, let us note that we have by the smoothing of the semigroup 
\[
|e^{t(r_u \DeltaA+\gamma I)} \Div\, w|_{L^2 }\le (1+t^{-1/2
 })e^{-(r_u \lambda_1+\gamma) t} |w|_{L^2},
\]
for any $w \in H^1(\CO)$, with $\lambda_1$ being the first eigenvalue of  $A$. We define
\DEQSZ\label{def_r}
r(t) := (1+t^{-1/2})e^{-(r_u \lambda_1 +\gamma)t}.
\EEQSZ
Next, for simplicity, let us introduce the abbreviations
\begin{equation}\label{equ:Psi}
\Psi_{n}^1(s):=\psi_n(h^1(\xi,s)),\quad \Psi_{n}^2(s):=\psi_n(h^2(v,s)),
\quad \mbox{and} \quad \Psi_{n}^3(s):=\psi_n(h^3(v,s)),
\end{equation}
where $h^j(\xi,s)$, $j=1,2,3$, are defined in~\eqref{h12.def}.
%
%
%
Also, for simplicity, and since we will need this kind of calculations later on, we omit for the moment the cut-off function and derive a nice representation. Let us note that we have by \cite[p.\ 238]{runst} and interpolation for
$\ep>\frac 12$, $\delta\in(0,\frac 12)$, $\rho=\delta+\frac 12$,  
\DEQSZ
\label{spimp}
\lqq{\qquad \int_0^t  |e^{(t-s)(r_u\DeltaA+\gamma I)}  \Div( \xi (s)\,\nabla v(s))|_{L^2} \, ds}
&&\\
\nonumber
&=&\int_0^ t r(t-s)  |\xi(s) \nabla v(s)|_{L^2}\, ds \le \int_0^ t r(t-s) |\xi(s)|_{H^{-\delta}}|\nabla v(s)|_{H^\rho}\, ds
\\
\nonumber &\le&%
\int_0^t  r(t-s) |\xi(s)|_{H^{-\ep }}^{\delta/\ep}|\xi(s)|_{L^2}^{1-\delta/\ep} \, |\nabla v(s)|_{L^2}^{\frac 12-\delta}|\nabla v(s)|_{H^1}^ {\delta+\frac 12}\, ds
\\
\nonumber
&\le& C t^{\frac {5-2\delta}{4}} \,  \sup_{0\le s\le t}|\xi(s)|_{H^{-\ep }}^{\delta/\ep}\,\sup_{0\le s\le t}|\xi(s)|_{L^2}^{1-\delta/\ep} \,
 \sup_{0\le s\le t}|\nabla v(s)|_{L^2}^{\frac 12-\delta}\Big(
\int_0^ t |\nabla v(s)|_{H^1}^2\, ds \Big)^ \frac{(\delta+\frac 12)}{2}.
\EEQSZ
The constant $C$ depends on $\delta$.
\del{(I suggest to remove the equation below. It repeats (3.17) and is not referred to in the
sequel.)
In particular, there exists some $C=C(\delta)>0$ such that we can write
\DEQSZ
\label{spimp2}
\lqq{\qquad\qquad \int_0^t  |e^{(t-s)(r_u\DeltaA+\gamma I)}  \Div( \xi (s)\,\nabla v(s))|_{L^2} \, ds}
&&\\
&\le &\nonumber
C t^{\frac {3-2\delta}{4}} \,  \sup_{0\le s\le t}|\xi(s)|_{H^{-\ep }}^{\delta/\ep}\,\sup_{0\le s\le t}|\xi(s)|_{L^2}^{1-\delta/\ep} \,
\int_0^ t |\nabla v(s)|_{H^1}^2\, ds \Big)^ {\frac \delta 2+\frac 14}.
\EEQSZ
}

We are now ready to prove item (a) in Claim~\ref{claim1}, i.e., we
find some $R>0$ such that $\MT_n$ maps $\mathcal{C}_R$ onto $\mathcal{C}_R$.
Here, applying the  consideration above, the
Burkholder--Davis--Gundy inequality \eqref{HSnorm},   and taking into account  the cut-offs, we
obtain
\DEQS
\lqq{ \EE\Big[\sup_{0\le s\le t}|u_n(s)|_{L^2} ^2\Big] \le  \EE \Big[\sup_{0\le s\le t} | e^{s(r_u\DeltaA+\gamma I)} u_0|^2_{L^2}\Big]
  +\chi C t^{\frac {3-2\delta}{4}} \, \EE\Bigg[  \sup_{0\le s\le t}(\Psi_n^1(s))^2 |\xi(s)|_{H^{-ep }}^{2\delta/\ep}}
&&\\
 && \times \,\sup_{0\le s\le t}|\xi(s)|_{L^2}^{2-2\delta/\ep} \,
 \sup_{0\le s\le t}(\Psi_n^2(s))^2|\nabla v(s)|^{1-2\delta}_{L^2}\Big(
\int_0^ t \Psi_n^3(s)|\nabla v(s)|_{H^1}^2\, ds \Big)^ {2\delta+1/2}\Bigg]
 \notag\\&&\quad+
\EE \Big[\int_0^ t \sup_{0\le r\le s} | u_n(r)|^2_{L^2}\, ds\Big]. 
\EEQS
Due to the definition of $\Psi^1_n$, $\Psi_n^2$, and $\Psi_n^3$ we can write
\DEQS
 \EE \Big[\sup_{0\le s\le t}|u_n(s)|_{L^2}^2 \Big] &\le&  \EE\Big[\sup_{0\le s\le t}| e^{s(r_u\DeltaA+\gamma I)} u_0|_{L^2}^2
 \Big] +\chi C t^{\frac {7-2\delta}{4}} \, \EE\Bigg[  (2n)^{2\delta/\ep}
  \sup_{0\le s\le t}|\xi(s)|_{L^2}^{2(1-\delta/\ep)} \notag\\&&\quad \quad \times
(2n)^{1-2\delta}  (2n)^ {2\delta+\frac 12}
 +
\EE \Big[\int_0^ t \sup_{0\le r\le s} | u_n(r)|^2_{L^2}\, ds\Big]. 
\EEQS
Applying Gronwall's lemma we obtain
\DEQS
 \EE\Big[\sup_{0\le s\le t}|u_n(s)|^2_{L^2}\Big]
   &\le& C(t) \EE \Big[ |u_0|^2_{L^2}
+\xi (2n)^{2\delta/\ep+\frac 32} (\EE\,\sup_{0\le s\le t}|\xi(s)|^2_{L^2})^{1-\delta/\ep}\Big].
\EEQS
Taking $\delta=\frac \ep 2$, it follows that if $R$ satisfies
\DEQSZ\label{Rsat}
 R> \sqrt{C(t)}\Big[\EE |u_0|_{L^2}
+ (2n)^{\frac 54} \sqrt{R}\Big],
\EEQSZ
then $\MT_n(\xi)\in\mathcal{C}_R$. Therefore, for any $n\in\NN$  there exists a number $R=R(n,u_0)>0$ given by \eqref{Rsat}
such that
$\MT_n$ maps $\CC_R$ to~$\CC_R$. 

To show item (b) in Claim~\ref{claim1},
we use the following result. Suppose that $t>s$ and  $\eta \in  \LL^ q(0,T;L^2(\CO))  $. Let 
$$(\MF _{r_u\DeltaA+\gamma I}\eta) (t)=\int_0^ t e^{(t-s)(r_u\DeltaA+\gamma I)} \eta(s)\, ds,\quad t\in[0,T].
$$
In Lemma 2.4 of \cite{brzezniaGatarek}, compare also with Lemma \ref{L:reg} in Appendix \ref{appb},
it is shown that, for $\rho>\frac 14$ and $\beta>0$, if~$\beta<1-\rho -\frac 1q-\bar
\alpha/2$ and $\eta\in \LL^q(0,T;H^{-\rho}(\CO))$, then $\MF _{r_u\DeltaA+\gamma I}\eta\in C^{\beta}([0,T];H^{\bar \alpha}(\CO))$.
Moreover,
$$
\|\MF_{r_u\DeltaA+\gamma I}\eta\|_{ C^{\beta}([0,T];H^{\bar \alpha})}\le C\, \|\eta\|_{\LL^q(0,T;H^{-2\rho})}.
$$
Let us put $\eta(s)=\phi_n(h^1(\xi,s))\phi_n(h^2(v,s))\phi_n(h^3(v,s))\,\Div( \xi (s)\,\nabla v(s))$, $s\in[0,T]$.
Then,  by the Sobolev embedding $L^1(\CO)\hookrightarrow H^{-(2\rho+1)}(\CO)$ we obtain
$$
|\Div( \xi (s)\,\nabla v(s)) |_{H^{-2\rho }}
\le
| \xi (s)\,\nabla v(s) |_{H^{-(2\rho+1)}}
\le
| \xi (s)\,\nabla v(s) |_{L^1}
\le
|\xi(s) |_{L^2} |\nabla v (s)|_{L^2}.
$$
Due to the definitions of the $h^2(v,\cdot)$, $h^3(v,\cdot)$, and $\eta$ we have  for and $q\ge 2$
\[
\|\eta\|_{\LL^q(0,T;H^{-2\rho})}
\le
2n\sup_{0\le s\le T}|\xi(s)|_{L^2}.
\]
Hence, $\MF _{r_u\DeltaA+\gamma I}\eta\in C^\beta_b([0,T];H^{\bar \alpha}(\CO))$ for $\bar \alpha,\beta\ge 0$ and $\bar \alpha+2\beta<\frac 12$.

It remains to show the continuity for the stochastic term.
Recalling that  $\xi  \in \CM_\MA ^2(0,T;L^2(\CO))$ we define
$$(\MS _{r_u\DeltaA+\gamma I}\xi ) (t)=\int_0^ t e^{(t-s)(r_u\DeltaA+\gamma I)} \xi (s)\, d\mathcal W_1(s),\quad t\in[0,T].
$$
We first show that there exists  $C>0$ such that
\DEQSZ\label{wichtig}
 \EE\Big[\sup_{0< t\le T} |\MS_{r_u\DeltaA+\gamma I} \xi |_{H^\alpha}\Big]\le C\|\xi\|_{\CM_\MA^2(0,T;L^2)}.
\EEQSZ
Indeed, choosing $\beta=0,\, \delta=\bar \alpha/2, \rho=0$, and $p>1$ such that $\bar
\alpha/2+1/p <1/2$ in Lemma \ref{L:reg}, we deduce
\DEQS
 \EE\Big[ |\MS_{r_u\DeltaA+\gamma I} \xi |^p_{C^0_b([0,T];H^{\bar \alpha}(\CO))}\Big]
 \le C \EE\Big[\int_0^T |\xi(s)|^p_{{L_\text{HS}(H, L^2(\CO))}}ds \Big]
 \le C_T \EE\Big[\sup_{0 \leq s \leq T}\|\xi(s)\|^p_{L^2(\CO)}
 \Big].
\EEQS
This essentially yields \eqref{wichtig}.
Elementary calculations reveal
\begin{align*}
\EE \Big|\int_0^ t e^{(t-s)(r_u\DeltaA+\gamma I)} \xi(s)d\mathcal W_1(s)\Big|^2_{H^\alpha} \le
C_{\alpha}\EE \Big[\int_0^ t (t-s)^{-\alpha}|\xi(s)|^2_{L^2}\, ds\Big]
\leq
\dfrac{t^{1-\alpha}}{1-\alpha}\|\xi\|^2_{\CM_\MA^2(0,T;L^2)}.
\end{align*}
For $T\ge t>s>0$, elementary calculations
lead to the following identity
\DEQS
\lqq{ (\MS_{r_u\DeltaA+\gamma I} \xi )(t) - (\MS_{r_u\DeltaA+\gamma I} \xi )(s) =
\int_0 ^{t} e ^ {(t-r)(r_u\DeltaA+\gamma I)}\xi (r)
 d\mathcal W_1(r)
  -  \int_0  ^s   e ^ {(s-r)(r_u\DeltaA+\gamma I)}     \xi (r )\; d\mathcal W_1(r)}
  &&\\& =&
  \int_s ^{t} e ^ {(t-r)(r_u\DeltaA+\gamma I)}
   \xi (r )\;d\mathcal W_1(r)
+ \lk(e ^ {(t-s)(r_u\DeltaA+\gamma I)}  -I\rk)  \int_0  ^s  e ^ {(s-r)(r_u\DeltaA+\gamma I)}   \xi (r)\;d\mathcal W_1(r)
\\
&=&\int_s ^{t} e ^ {(t-r)(r_u\DeltaA+\gamma I)}
   \xi (r)\; d\mathcal W_1(r) + \lk( e ^ {(t-s)(r_u\DeltaA+\gamma I)}  -I\rk) (\MS_{r_u\DeltaA+\gamma I} \xi )(s).
\EEQS
Then, it follows from Remark 1.1 \cite{maxjan} that
$$
\EE\Big[\sup_{s\le  t\le s+h}\Big| \int_s ^{t} e ^ {-(t-r)(r_u\DeltaA+\gamma I)}
   \xi (r)\; d\mathcal W_1(r)\Big|^2_{H^{\bar \alpha}}\Big]
   \le h^{-\bar \alpha} \EE \Big( \int_s ^{s+h}|\xi(s)|_{L^2}^2\, ds\Big)^\frac p2\le \|\xi\|_{\CM_\MA^2(0,T;L^2)}^p\, h^\frac p2.
   $$
   In addition, by standard calculations we know
$$
|\lk( e ^ {(t-s)(r_u\DeltaA+\gamma I)}  -I\rk) (\MS_{r_u\DeltaA+\gamma I} \xi )(s)|_{L^2}^2
\le (t-s)^{\bar \alpha}  |(\MS_{r_u\DeltaA+\gamma I} \xi )(s)|^2_{H^{\bar \alpha}}.
$$
Due to estimate \eqref{wichtig} there exists a number  $\bar \alpha>0$ such that
\DEQS
\EE \Big[ \sup_{0<s<t\le T} \big| (\MS _{r_u\DeltaA+\gamma I}\xi )(t) - (\MS _{r_u\DeltaA+\gamma I}\xi )(s)\big|^2_{H^{\bar \alpha}}\Big] \le C (t-s)^{\bar \alpha}  \|\xi\|^2_{\CM_\MA^2(0,T;L^2)}.
\EEQS
This completes the proof of Claim~\ref{claim1}.
\end{proof}
Next, we show the continuity of the mapping  $\MT_n$.
\begin{claim}\label{continuity}
Under Assumption \ref{wiener}, for any initial condition $(u_0,v_0)$
satisfying Assumption \ref{init} and for any $R>0$, the mapping
$\MT_n: \CC_R \to \CM_\MA ^2(0,T;L^2(\CO))$ is continuous
for each $n \in \mathbb{N}$, where~$\CC_R$ is defined by~\eqref{def.CR}.
\del{In particular, $n \in \mathbb{N}$ 
there exists $T^{\star}>0$ and $k\in(0,1)$ such that
\begin{align} \label{eq.cont}
 |\MT_n(\xi^1)-\MT_n(\xi^2)|_{\CM_{T^{\star}}^2(L^2)} \le k\,|\xi^1-\xi^2|_{\CM_{T^{\star}}^2(L^2)}, \quad \forall \xi^j \in \CM_{T^{\star}}^2(0,T^\star;L^2(\CO)), \,\,j=1,2.
\end{align}}
\end{claim}

\begin{proof}[Proof of Claim \ref{continuity}:]
Let us fix $n\in\NN$ and consider~$\xi^j\in\CM_\MA ^2(0,T;L^2(\CO))$, $j=1,2$. In order to prove
the continuity of $\MT_n$, we will show that there exists a constant $C=C(n)>0$ such that
\begin{equation}\label{equ:xi12}
\|\MT_n(\xi^1)-\MT_n(\xi^2)\|_{M_{\MA}^2(0,T;L^2)}
\le
C
\|\xi^1-\xi^2\|^{2}_{M_{\MA}^2(0,T;L^2)},
\end{equation}
where the $M_{\MA}^2(0,T;L^2(\CO))$-norm is defined by~\eqref{equ:M nor}.

For~$j=1,2$, let $(u_n^j,v^j)=(\MT_n(\xi^j),\Mr(\xi^j))$ be defined by~\eqref{ppp1} with~$\xi$
replace by $\xi^j$.
We define $\Psi_{n,j}^1$, $\Psi_{n,j}^2$, and $\Psi_{n,j}^3$ by~\eqref{equ:Psi} with $\xi$
and~$v$ replaced by $\xi^j$ and~$v^j$, respectively, $j=1,2$. Then it follows from the
definition of~$h^j$ in~\eqref{h12.def} that
\begin{align}\label{Psi1}
|\Psi^1_{n,1}(s)-\Psi^1_{n,2}(s)|
&=
|\psi_n(h^1(\xi^1,s))- \psi_n(h^1(\xi^2,s))|
\le
|\psi'|_{L^{\infty}}\, \frac 1n\,\Big| h^1(v^1,s) - h^1(v^2,s) \Big|
\nonumber\\
&\le
|\psi'|_{L^{\infty}}\, \frac 1n\,\sup_{0 \le r \le s}\Big| \xi^1(r)- \xi^2(r)\Big|_{L^1}.
\end{align}
Next, we have 
\begin{align}\label{equ:Psi2}
|\Psi^2_{n,1}(s)-\Psi^2_{n,2}(s)|
&=
|\psi_n(h^2(v^1,s))- \psi_n(h^2(v^2,s))|
\le
|\psi'|_{L^{\infty}}\, \frac 1n\,\Big| h^2(v^1,s) - h^2(v^2,s) \Big|
\nonumber\\
&\le
|\psi'|_{L^{\infty}}\, \frac 1n\,\sup_{0 \le r \le s}\Big|\nabla v^1(r)-\nabla
v^2(r)\Big|_{L^2}.
\end{align}
By similar calculations we obtain
\begin{align}\label{equ:Psi3}
|\Psi^3_{n,1}(s)-\Psi^3_{n,2}(s)|
&\le
|\psi'|_{L^{\infty}}\, \frac 1n\,\int_0^ s\Big|\nabla v^1(r)-\nabla
v^2(r)\Big|_{H^1}^2\, dr.
\end{align}
The supremum on the right-hand side of~\eqref{equ:Psi2} can be estimated by using the
Burkholder--Davis--Gundy inequality~\eqref{HSnorm.1} as follows
\begin{align*}
&\EE\Big[\sup_{0 \le s \le t} \Big|\nabla v^1(s)-\nabla v^2(s)\Big|^2_{L^2}\Big]
\\
&\le \EE \Big[\sup_{0\le s\le t} \Big( \int_0^s (s-s_1)^{-\frac 12} \, |\xi_1(s_1)-\xi_2(s_1)|_{L^2}\, ds_1\Big)^2\Big]
+\EE\Big[ \sup_{0 \le t \le T}\int_0^t |\nabla (v^1(s)-v^2(s))|^2_{L^2} \,ds \Big]
\\
&\le
4\,T\,\EE\Big[ \sup_{0\le s\le T}
|\xi_1(s)-\xi_2(s)|_{L^2}^2\Big]
+\int_0^t \EE \Big[\sup_{0 \le r \le s}|\nabla (v^1(r)-v^2(r))|^2_{L^2}\Big] \,ds.
\end{align*}
Applying  Gronwall's Lemma yields
\begin{align*}\label{nablav1_2}
\EE \Big[ \sup_{0 \le s \le t}|\nabla (v^1(s)-v^2(s))|_{L^2}^2\Big]
&\le 4T\,\EE \Big[ \sup_{0\le s\le t}  |\xi_1(s)-\xi_2(s)| ^2_{L^2}\Big].
\end{align*}
Finally, the maximum regularity and the \BDG inequality given in \eqref{HSnorm.1} yield
\begin{equation}\label{nabla_v_1}
 \EE \Big[ \int_0^ t  |\nabla v^1(r)-\nabla v^2(r)| ^2_{H^1}\, dr \Big]
\le \EE \Big[\int_0^ t |\xi^1(s)-\xi^2(s)|^2_{L^2}\, ds \Big] \le T\, \EE \Big[\sup_{0\le s\le T} |\xi^1(s)-\xi^2(s)|_{L^2}^2\,\Big] .
\end{equation}
In order to prove~\eqref{equ:xi12}, we estimate the difference $\phi(t) :=
\MT_n \xi^1(t)-\MT_n \xi^2(t) = u^1_n(t)-u^2_n(t)$.
It follows from the variation of constant formula (compare with \eqref{VOCF}) that
\begin{align}\label{sub.u}
\phi(t)
&=\int_0^ t e^{(t-s)(r_u\DeltaA+\gamma I)}\Big[
(\Psi^1_{n,1}(s)\Psi^2_{n,1}(s)\,\Psi^3_{n,1}(s)\,\Div(\xi^1(s)\nabla v^1(s))\notag
\\&\quad \quad
- \Psi^1_{n,2}(s),\Psi^2_{n,2}(s)\,\Psi^3_{n,1}(s)\,\Div(\xi^2(s)\nabla v^2(s))\Big] ds
+\int_0^ t e^{(t-s)(r_u\DeltaA+\gamma I)} \phi(s)d \mathcal W_1(s)\notag
\\
&=
\underbrace{\int_0^ t e^{(t-s)(r_u\DeltaA+\gamma I)}\,
\Div\bigg(\big(\xi^1(s) \Psi^1_{n,1}(s)\, -\xi^2(s) \Psi^1_{n,2}(s)\, \big)\Psi^2_{n,1}(s)\,\Psi^3_{n,1}(s)\, \nabla v^1(s)\bigg)\, ds}_{I_1:=} \notag
\\
&\vspace{-0.14cm}{}+\underbrace{\int_0^ t e^{(t-s)(r_u\DeltaA+\gamma I)}\Psi^1_{n,2}(s)
\Div( \xi^2(s))\Big( \Psi^2_{n,1}(s)\Psi^3_{n,1}(s) \nabla v^1(s)-\Psi^2_{n,2}(s)\,\Psi^3_{n,2}(s)\nabla v^2(s) \Big)\, ds}_{I_2:=}
\notag\\
&\quad+\underbrace{\int_0^ t e^{(t-s)(r_u\DeltaA+\gamma I)} (\xi^1(s)-\xi^2(s))d \mathcal W_1(s)}_{I_3:=}.
\end{align}
Firstly, let us consider $I_1$. Let $\sigma\in(\frac12,\frac 34)$, $\delta\in(0,\frac 12)$ and $\rho=\delta+\frac 12$ be fixed.
It is shown in \cite[p.\ 238]{runst} that
\DEQSZ\label{multiply1}
|\xi_1\xi_2|_{L^2}\le |\xi_1|_{H^{-\delta}}|\xi_2}|_{H^{\rho},\quad \xi_1\in H^{-\delta}(\CO),\, \xi_2\in H^{\rho}(\CO).
\EEQSZ
For any $\ep\in(\frac 12,\frac 34)$, interpolation between $H^{-\ep}(\CO)$ and $L^2(\CO)$
yields 
\DEQSZ\label{multiply2}
 |\xi_1|_{H^{-\delta}}\le  |\xi_1|_{H^{-\ep}}^{\delta/\ep}|\xi_1|_{L^2}^{1-\delta/\ep}\le  |\xi_1|_{L^1}^{\delta/\ep}|\xi_1|_{L^2}^{1-\delta/\ep}, \quad \xi_1 \in L^2(\CO).
\EEQSZ
Proceeding in similar lines as in~\eqref{spimp}, we recall the definition of~$r$ in \eqref{def_r}. Therefore, from the definition of the cut-off
function and the smoothing property of the semigroup, we deduce that
%
%
\begin{align}\label{esti.I2}
\EE \Big[\sup_{0 \le s \le t} |I_1(s)|^2_{L^2}\Big] \notag
&\le \EE \Big[\sup_{0 \le s \le t} \Big|\int_0^ s e^{(s-s_1)(r_u\DeltaA+\gamma I)}\,\Psi^2_{n,1}(s_1)\, \Psi^3_{n,1}(s_1)\,
\nonumber \\
& \quad {} \times \Div((\Psi^1_{n,1}(s_1)\,\xi^1(s_1)-\Psi^1_{n,2}(s_1)\,\xi^2(s_1)) \nabla v^1(s_1))\, ds_1 \Big|^2_{L^2}\Big]
\nonumber \\
&\le \EE \Big[\sup_{0 \le s \le t}
\Big( \int_0^ s r(s-s_1)\, \lk|\Psi^1_{n,1}(s_1)\xi^1(s_1)-\Psi^1_{n,2}(s_1)\xi^2(s_1)
\rk|_{H^{-\delta}}
\nonumber
\\
\nonumber
&\quad {}\times \Psi^2_{n,1}(s_1)\,\Psi^3_{n,1}(s_1) \,
\,|\nabla v^1(s_1)|^2_{H^\rho}\, ds_1\big)^2 \Big]
\\\nonumber
&\le C\, t^{\frac {5-2\delta}4} \EE \Big[\sup_{0 \le s_1 \le t}
   \lk|\Psi^1_{n,1}(s_1)\xi^1(s_1)-\Psi^1_{n,2}(s_1)\xi^2(s_1)
\rk|_{H^{-\delta}}
\\
&\quad \times\Big(\int_0^ t (\Psi^2_{n,1}(s_1))^{\frac{q}{2}}\,(\Psi^3_{n,1}(s_1))^{\frac{q}{2}} \,
\,|\nabla v^1(s_1)|^q_{H^\rho}\, ds_1\Big)^\frac 2q \Big].
\end{align}
%
%
In order to estimate the term involving the integral on the right-hand side of \eqref{esti.I2}, we 
observe that due to interpolation $H^1 \subset H^{\rho} \subset L^2$, we get for $0 \le s_1 \leq t $,
\begin{align*}
|\nabla v^1(s_1)|_{H^{\rho}} \leq C |\nabla v^1(s_1)|^{\theta}_{L^2}|\nabla v^1(s_1)|^{1-\theta}_{H^1},
\end{align*}
with $\theta=1-\rho$ and $C>0$ is the interpolation constant depending on the dimension d and $\rho$.
It follows for $q=\frac 4{1+2\delta}$ 
\begin{align}\label{above.1}
&\lk(\int_0^ t (\Psi^2_{n,1}(s_1))^{\frac{q}{2}}\,(\Psi^3_{n,1}(s_1))^{\frac{q}{2}}|\nabla v^1(s_1)|^{q}_{H^{ \rho}}\, ds_1\rk)^\frac 2q
\notag\\
&\le 
C^2 \sup_{0 \leq s \leq t}\Big(\Psi^2_{n,1}(s_1)) |\nabla v^1(s_1)|_{L^2}^{2 \theta}
\Big)\Big(\int_0^t (\Psi^3_{n,1}(s_1))^{\frac{q}{2}} |\nabla v^1(s_1)|_{H^1}^{(1-\theta)q}\,ds_1
\Big)^\frac 2q
\notag\\
& \leq \theta\,C^{\frac{2}{\theta}}\sup_{0 \leq s \leq t}\Big(\Psi^2_{n,1}(s_1)) |\nabla v^1(s_1)|_{L^2}^{2 \theta}
\Big)^{\frac{1}{\theta}}
+ (1-\theta) \Big(\int_0^t (\Psi^3_{n,1}(s_1))^{\frac{q}{2}} |\nabla v^1(s_1)|_{H^1}^{2}\,ds_1
\Big)^\frac{2(1-\theta)}{q}
\notag\\
& \leq (1-\rho)C^{\frac{2}{1-\rho}}\sup_{0 \leq s \leq t}\Big((\Psi^2_{n,1}(s_1))^{\frac{1}{1-\rho}} |\nabla v^1(s_1)|_{L^2}^{2}
\Big)
+ \rho \Big(\int_0^t (\Psi^3_{n,1}(s_1))^{\frac{q}{2}} |\nabla v^1(s_1)|_{H^1}^{2}\,ds_1
\Big).
\end{align}
Using the definition of $\Psi^2_{n,1}$, we obtain
\begin{align*}
\Psi^2_{n,1}(s_1)=
\begin{cases}
1
&\mbox{ if } \sup_{0\le s_1\le s} |\nabla v^1(s_1)|_{L^2}^2 \le n,
\\[2ex]
\in (0,1)
&\mbox{ if } \sup_{0\le s_1\le s} |\nabla v^1(s_1)|_{L^2}^2 \in (n,2n),
\\[2ex]
0 &\mbox{ otherwise}.
\end{cases}
\end{align*}
Therefore,  $\Psi^2_{n,1}(s_1)|\nabla v^1(s_1)|_{L^2}^{2} \leq 2n $. This eventually implies that
\begin{align}\label{esti psi 2}
(\Psi^2_{n,1}(s_1))^{\frac{1}{1-\rho}} |\nabla v^1(s_1)|_{L^2}^2
\leq \Psi^2_{n,1}(s_1) |\nabla v^1(s_1)|_{L^2}^2 (\Psi^2_{n,1}(s_1))^{\frac{\rho}{\rho-1}} 
\leq 2n.
\end{align}
In similar manner, using the definition of $\Psi^3_{n,1}$, one obtain that
\begin{align}\label{esti psi 3}
\int_0^t (\Psi^3_{n,1}(s_1))^{\frac{q}{2}}|\nabla v^1(s_1)|_{H^1}^2ds_1 \leq 2n.
\end{align}
Using \eqref{esti psi 2}-\eqref{esti psi 3} in \eqref{above.1} we obtain
\begin{align*}
&\lk(\int_0^ t  (\Psi^2_{n,1}(s_1))^{\frac{q}{2}}\,(\Psi^3_{n,1}(s_1))^{\frac{q}{2}}|\nabla v^1(s_1)|^{q}_{H^{ \rho}}\, ds_1\rk)^\frac 2q
 \leq 2n\, (1-\rho)C^{\frac{2}{1-\rho}}
+ 2n\,\rho. 
\end{align*}
We now estimate the term involving the supremum on the right-hand side of~\eqref{esti.I2}.
Let $$\tau_j^u:=\inf_{s>0}\{|\xi^j(s)|_{L^1}\ge n\} \wedge T,
\qquad
\tau_j^o:=\inf_{s>0}\{|\xi^j(s)|_{L^1}\ge 2n\}\wedge T,\,\,j=1,2.$$
We distinguish the following cases:
\begin{align}\label{above.2}
|\xi^1(s)\Psi_{n,1}^1(s)-\xi^2(s)\Psi_{n,2}^1(s)|_{H^{-\delta}}
&\le
\begin{cases}
(2n)^{\delta/\ep} |\xi^1(s)-\xi^2(s)|^{1-\delta/\ep}_{L^2} 
&\mbox{ if } 0\le s\le \ul{\tau}^u,
\\[2ex]
2^{\delta/\ep}n^{\delta/\eps-1} |\psi'|_{L^{\infty}}  
|\xi^1(s)|^{1-\delta/\ep}_{L^2} &
\\[2ex]
\quad \quad \times \sup_{0 \leq r \leq s}|\xi^1(r)-\xi^2(r)|_{L^1}&
\\[2ex]
\quad + (4n)^{\delta/\ep} |\xi^1(s)-\xi^2(s)|^{1-\delta/\ep}_{L^2} 
&\mbox{ if }\ul{\tau}^u < s \le \ol{\tau}^o,
\\[2ex]
0 &\mbox{ if }s>\ol{\tau}^o,
\end{cases}
\end{align}
where $\ul{\tau}^u:= \min(\tau_1^u,\tau_2^u)$ and
$\ol{\tau}^o:=\max(\tau_1^o,\tau_2^o)$.
Elementary calculations reveal that
\begin{align}\label{esti int.r}
\Big( \int_0^s (r(s-s_1))^{\frac{q}{q-1}} ds_1 \Big)^{\frac{2(q-1)}{q}}
\leq 2 \Big(\dfrac{3-2s}{1-2s} \Big) 
\Big(s^{\frac{3-2\delta}{2}}+s^{\frac{1-2\delta}{2}}
\Big).
\end{align}
Substituting~\eqref{above.1} and~\eqref{above.2} into~\eqref{esti.I2}, using \eqref{esti int.r},
and taking into account the fact that $\ul{\tau}^u \leq T$, we obtain 
\begin{align}\label{esti I1 fin}
\EE \Big[\sup_{0 \le s \le t} |I_1(s)|^2_{L^2}\Big]
&\leq 2 C\,n\bigg(\dfrac{3-2\delta}{1-2 \delta}\bigg)\Big( (1-\rho)C^{\frac{2}{1-\rho}}
+ \rho \Big)(t^{\frac{3-2\delta}{2}}+t^{\frac{1-2\delta}{2}})
\notag\\
& \quad \times
\Big\{ \EE\Big[ \sup_{0 \le s \le  \ul{\tau}^u}  (2n)^{2\delta/\ep} |\xi^1(s)-\xi^2(s)|^{2(1-\delta/\ep)}_{L^2}\Big]
\notag\\
& \quad +2\EE \Big[\sup_{ \ul{\tau}^u \le s \le \ol{\tau}^o}
\Big[ 2^{2\delta/\ep} n^{2(\delta/\ep-1)} |\psi'|_{L^{\infty}}^2 |\xi^1(s)-\xi^2(s)|^2_{L^1}  |\xi^1(s)|^{2(1-\delta/\ep)}_{L^2}
\notag\\
& \quad \quad
+ (4n)^{2\delta/\ep} |\xi^1(s)-\xi^2(s)|^{2(1-\delta/\ep)}_{L^2}\Big]\Big] \Big\}
\notag\\
 & \leq 2^{2\delta/\eps} 8C \bigg(\dfrac{3-2\delta}{1-2 \delta}\bigg) \Big( (1-\rho)C^{\frac{2}{1-\rho}}
+ \rho \Big)n^{1+2\delta/\ep}(t^{\frac{3-2\delta}{2}}+t^{\frac{1-2\delta}{2}})
\notag\\
& \quad \times 
\Big\{ \Big(\EE\Big[ \sup_{0 \le s \le  T}  |\xi^1(s)-\xi^2(s)|^{2}_{L^2}\Big]\Big)^{(1-\delta/\ep)}
\nonumber\\
& \quad
+\dfrac{1}{n^2} R^{1-\delta/\eps} \Big(
\EE \Big[\sup_{ 0 \le s \le T} |\xi^1(s)-\xi^2(s)|_{L^2}^{2 \eps/\delta} 
\Big)^{\delta/\eps}\Big\}. 
\end{align}
%
We now treat the term $I_2$.
We fix again $\delta\in(0,\frac 12)$, $\rho=\delta+\frac 12$, and $\ep\in(\frac 12,\frac34)$.
Using the same approach as in the calculation of \eqref{spimp} we obtain
%
\begin{align*} 
\EE \Big[\sup_{0 \le s \le t} |I_2(s)|^2_{L^2}\Big]
& \le
\EE \Big[\sup_{0 \le s \le t} \Big| \int_0^ s e^{(s-s_1)(r_u\DeltaA+\gamma I)}\,\Psi^1_{n,2}(s_1)\,
\\
&\quad \times
\Div( \xi^2(s_1) [ \Psi_{n,1}^2(s_1) \Psi_{n,1}^3(s_1)\, \nabla v^1(s_1)-\Psi_{n,2}^2(s_1) \Psi_{n,2}^3(s_1)\, \nabla v^2(s_1)]\, ds_1 \Big|^2_{L^2}\Big]
\notag
\\\nonumber
& \le
\EE \Big[\sup_{0 \le s \le t} \Big( \int_0^ s r(s-s_1) \Psi_{n,2}^1(s_1)\,
\\
&\quad\times \Big| \xi^2(s_1) [\Psi_{n,1}^2(s_1) \Psi_{n,1}^3(s_1)\, \nabla v^1(s_1)-\Psi_{n,2}^2(s_1) \Psi_{n,2}^3(s_1)\,\nabla v^2(s_1)]\Big|_{L^2}\, ds_1 \Big)^2 \Big] \notag
\\\nonumber
& \le
C  \bigg(\frac{3-2\delta}{1-2 \delta}\bigg)(t^{\frac{3-2\delta}{2}}+t^{\frac{1-2\delta}{2}})\, \EE \Big[\sup_{0 \le s \le t} \Big( \int_0^ s  (\Psi_{n,2}^1(s_1))^q\,
\\
&\quad\times \Big| \xi^2(s_1) [\Psi_{n,1}^2(s_1) \Psi_{n,1}^3(s_1)\, \nabla v^1(s_1)-\Psi_{n,2}^2(s_1) \Psi_{n,2}^3(s_1)\,\nabla v^2(s_1)]\Big|^q_{L^2}\, ds_1 \Big)^\frac 2q \Big]. \notag
\end{align*}
For j=1,2, we define the following stopping times 
 \begin{alignat*}{2}
 \tau_j^u &:= \inf_{s>0}\{|\nabla v^j(s)|_{L^2}\ge n^{1-\rho}\} \wedge T, &\quad
 \tau_j^o &:= \inf_{s>0}\{|\nabla v^j(s)|_{L^2}\ge (2n)^{1-\rho}\} \wedge T
 \\
 \tilde \tau_j^u &:= \inf_{s>0}\{\int_0^s |\nabla v^j(s_1)|_{H^1}^2\, ds_1\ge n^\rho\}\wedge T, &\qquad
 \tilde \tau_j^o &:= \inf_{s>0}\{\int_0^s |\nabla v^j(s_1)|_{H^1}^2\, ds_1\ge (2n)^\rho\}\wedge T,
\\
 \hat \tau_j^u &:= \inf_{s>0}\{\int_0^s |\nabla v^j(s_1)|^q_{H^\rho}\, ds_1\ge n\}\wedge T, &\quad 
 \hat \tau_j^o &:= \inf_{s>0}\{\int_0^s |\nabla v^j(s_1)|^q_{H^\rho}\, ds_1\ge 2n\}\wedge T.
  \end{alignat*}
From \eqref{above.1} it follows that
$
\hat \tau^u_j\ge \ul{\tau}^{u}_{j}:=\min(\tau_j^u,\tilde \tau_j^u)
$
 for $q=\frac 4{1+2\delta}$.
Now,
let us assume that $ 0\le s\le \ul{\tau}^{u}_{j}$. Then we get by \eqref{multiply1} and \eqref{multiply2}
\begin{align*}
& \Big( \int_0^ s  (\Psi_{n,2}^1(s_1))^q\,
 \Big| \xi^2(s_1) [\Psi_{n,1}^2(s_1) \Psi_{n,1}^3(s_1)\, \nabla v^1(s_1)-\Psi_{n,2}^2(s_1) \Psi_{n,2}^3(s_1)\,\nabla v^2(s_1)]\Big|^q_{L^2}\, ds_1 \Big)^{2/q}
\nonumber\\
& \le C \int_0^s \Big((\Psi_{n,2}^1(s_1))^q\,|\xi^2(s_1)|_{L^1}^q
|\nabla v^1(s_1)-\nabla v^2(s_1)|^{q}_{H^\rho}ds_1
\Big)^{2/q} 
\nonumber\\
 &\le C(2n)^{2} \,
 \Big(\sup_{0\le s_1\le s} |\nabla v^1(s_1)-\nabla v^2(s_1)|^{2(1-\rho)}_{L^2}\Big)\,  \lk( \int_0^s |\nabla v^1(s_1)-\nabla v^2(s_1)|^2_{H^1}ds_1\rk)^ {\rho}.
\end{align*}
The Cauchy-Schwarz inequality and estimate \eqref{nabla_v_1} give for $ 0\le s\le \ul{\tau}^{u}_{j}$
\begin{align*}
&\EE \Big[\sup_{0 \le s \le t}\Big( \int_0^ s  (\Psi_{n,2}^1(s_1))^q\,
 \Big| \xi^2(s_1) [\Psi_{n,1}^2(s_1) \Psi_{n,1}^3(s_1)\, \nabla v^1(s_1)-\Psi_{n,2}^2(s_1) \Psi_{n,2}^3(s_1)\,\nabla v^2(s_1)]\Big|^q_{L^2}\, ds_1\Big)^{2/q} \Big]\notag
\\
&\le C(2n)^{2}t^2 \,
\lk(\EE \Big[ \sup_{0\le s \le t} |\xi_1(s)-\xi_2(s)|^{2}_{L^2}\Big] \rk)^{1-\rho}
\lk(\EE \Big[ \sup_{0\le s \le t} |\xi_1(s)-\xi_2(s)|^{2}_{L^2}\Big] \rk)^{\rho}
\\
&\le C(2n)^{2}t^2 \,
\|\xi^1-\xi^2\|_{\CM^2_\MA(0,T;L^2)}^{2}.
\end{align*}
Next,
let us assume that $  \ul{\tau}^{u}_{j}<s\le \ul{\tau}^{o}_{j}:=\max(\tau_j^o,\tilde \tau_j^o)$. Then we get by \eqref{multiply1} and \eqref{multiply2}
\begin{align*}
 &\Big( \int_0^ s  (\Psi_{n,2}^1(s_1))^q\,
 \Big| \xi^2(s_1) [\Psi_{n,1}^2(s_1) \Psi_{n,1}^3(s_1)\, \nabla v^1(s_1)-\Psi_{n,2}^2(s_1) \Psi_{n,2}^3(s_1)\,\nabla v^2(s_1)]\Big|^q_{L^2}\, ds_1\Big)^{2/q}\notag
\\ 
 &\le C(2n)^{2} \Big\{ \sup_{0\le s_1\le s} |\nabla v^1(s_1)-\nabla v^2(s_1)|^{2}_{L^2}
\Big(\int_0^s |\nabla v^1(s_1)|_{H^1}^q ds_1 \Big)^{2/q}\notag
\\ 
 &\quad +
 \sup_{0\le r \le s} 
\Big(\int_0^r |\nabla v^1(s_1)-\nabla v^2(s_1)|_{H^1}^2 ds_1 \Big)
\notag\\ & \quad
+ \sup_{0\le s_1 \le s} |\nabla v^1(s_1)-\nabla v^2(s_1)|_{L^2}^{2(1-\rho)}
\Big(\int_0^s |\nabla v^1(s_1)-\nabla v^2(s_1)|_{H^1}^2 ds_1 \Big)^{2/q}\Big\}.
\end{align*}
Taking expectation, applying the Cauchy-Schwarz inequality, and taking into account the fact
that $\Psi\le 1$, we infer
\begin{align*}
&\EE \Big[\sup_{0 \le s \le t} \Big( \int_0^ s  (\Psi_{n,2}^1(s_1))^q\,
 \Big| \xi^2(s_1)\Big[\Psi_{n,1}^2(s_1) \Psi_{n,1}^3(s_1)\, \nabla v^1(s_1)-\Psi_{n,2}^2(s_1) \Psi_{n,2}
^3(s_1)\,\nabla v^2(s_1)\Big]\Big|^q_{L^2}\, ds_1\Big)^{2/q} \red{\Big]}
\notag\\ 
 & \le C(2n)^{2}t^2 \,
\|\xi^1-\xi^2\|_{\CM^2_\MA(0,T;L^2)}^{2}.
\end{align*}
Taking $C=C(2n)^2t^2$ we achieve
\DEQSZ\label{esti I2 fin}
 \EE \Big[\sup_{0 \le s \le t} |I_2(s)|^2_{L^2}\Big] &\le & C\|\xi^1-\xi^2\|^2_{M_{\MA}^2(0,T;L^2)}.
\EEQSZ
In similar manner, the term $I_3$ can be handled using \eqref{HSnorm}. In particular, we get
\DEQSZ\label{esti I3 fin}
 \ \EE \Big[\sup_{0 \le s \le t} |I_3(s)|^2_{L^2}\Big] \le 4t\|\xi^1-\xi^2\|^2_{M_{\MA}^2(0,T;L^2)}.
\EEQSZ
Combining \eqref{esti I1 fin}, \eqref{esti I2 fin} and \eqref{esti I3 fin} in \eqref{sub.u} appropriately, one can infer~\eqref{equ:xi12}. This completes the proof of {Claim \ref{continuity}}.
\end{proof}

\begin{proof}[Proof of Lemma \ref{lem.exist.n}]
Let $R>0$ be such that $\MT_n$ maps $ \CC_R$ into itself, as given by Claim~\ref{claim1} (a).
By Claim \ref{claim1} (b) we infer that there exists an $\alpha>0$ such that for any $t\in (0,T]$ there exists a constant $K=K(t,\alpha)>0$ with
\DEQSZ\label{compact11}
\EE |\MT_n\xi(t)|^2_{H^\alpha}\le K,\quad \xi\in\CC_R,
\EEQSZ
and a constant $C>0$ such that
\DEQSZ\label{compact22}
\EE\Big[ \sup_{0< t_1<t_2\le T} \frac {|\MT_n\xi(t_2)-\MT_n\xi(t_1)|_{L^2}}{(t_2-t_1)^\delta }\Big]\le C,\quad \xi\in \CC_R.
\EEQSZ
In the sequel we show that there exists a probability space $\MA^\ast$ and  two Wiener processes $\CW_1$ and $\CW_2$ being cylindrical
on $\CH_1$ and $\CH_2$ (let us remind that the spaces are given in Assumption \ref{wiener})
 defined over $\MA^\ast$, and a process $u^\ast_n$ such that $\MT_n(u^\ast)=u^\ast$.
The proof is done in the steps (a) to (f) below.
\medskip

\paragraph{{\bf Step (a):}}
For each $n\in\NN$, let $\{ (u_n^{(k)},v_n^{(k)}):k\in\NN\}$ be a recurrence sequence defined by the
operators~$\MT_n$ and~$\Mr$ via the system~\eqref{ppp1} as follows:
\begin{equation}\label{vsolves}
u_n^{(0)}(t):=0, \ t\in[0,T], \quad
(u_n^{(k)},v_n^{(k)}):=\big(\MT_n(u^{(k-1)}_n),\Mr(u^{(k-1)}_n)\big), \ k\ge1.
\end{equation}

%
First, observe that the laws of the sequence  $\{ (u_n^{(k)},v_n^{(k)}):k\in\NN\}$ is a tight family on
$C_b^0([0,T];L^2(\CO))\times C_b^0([0,T];H^1(\CO))$.
In fact, the set of probability measures $\{\pi_j\circ (u_n^{(k)},v_n^{(k)}):k\in\NN\}$ is
tight, due to Claim \ref{claim1}.
(By $\pi_j$, $j=1,2$, we denote  the projection onto the $j^{th}$ coordinate). To be more precise, by  Claim \ref{claim1} (a) we know that  $u^{(k)}_n\in \CC_R$, and, by Claim  \ref{claim1} (b) we know that \eqref{compact11} and \eqref{compact22} are satisfied. By the Chebyscheff inequality and
the Arzela--Ascoli Theorem (see \cite[Lemma 2.2, p.\ 213]{para}), we infer that the family of probability laws induced by
$\{ \pi_1\circ (u_n^{(k)},v_n^{(k)}):k\in\NN\}$ is tight.
The tightness of the laws of the sequence $\{ v_n^{(k)}:k\in\NN\}$ follows by the following observations.
Invoking Lemma~\ref{L:reg} with $E=H^1(\CO)$ (see Assumption~\ref{assum-1}), $\bar
\alpha=1$, $\gamma=-\frac 12$, $\delta=0$ and $q> 1$, $\beta>0$ sufficiently large so that $\bar
\alpha-\frac 1q <\beta$, we obtain for $\xi\in \CC_R$
\[
\Big\|\int_0^ \cdot e^{ (\cdot-s) (r_uA-\gamma I)}\xi(s)\, ds\Big\|_{C^\beta_b([0,T];H^1)}
\le \|\xi\|_{L^q(0,T;L^2)}.
\]
Secondly, it follows from Corollary \ref{brz:con} (see also \cite[Lemma 3.3]{neerven})
setting $E=H^1(\CO)$, $\delta=0$, $\nu=0$, and $p> 2$, $\beta>0$ such large that
$\beta+\frac 1p<\frac 12$
\[
\EE\Big[\Big\|\int_0^ \cdot e^{ (\cdot-s) (r_vA-\alpha I)}v(s)\,
d\CW_2(s)\Big\|^2_{C^\beta_b([0,T];H^1)}\Big]
\le \EE \Big[\|v\|^2_{L^p(0,T;H^1)}\Big].
\]

Since the marginals of $\{ (u_n^{(k)},v_n^{(k)}):k\in\NN\}$ form tight families, $\{
(u_n^{(k)},v_n^{(k)}):k\in\NN\}$ forms itself a tight family in
$C_b^0([0,T];L^2(\CO))\times C_b^0([0,T];H^1(\CO))$.
Hence, the laws of the sequence  $\{ (u_n^{(k)},v_n^{(k)}):k\in\NN\}$ are a tight family on
$C_b^0([0,T];L^2(\CO))\times C_b^0([0,T];H^1(\CO))$. 

\paragraph{{\bf Step (b):}}
Due to the tightness of the laws of the sequence $\{ (u_n^{(k)},v_n^{(k)}):k\in\NN\}$,
there exists a subsequence, again denoted by $\{ (u_n^{(k)},v_n^{(k)}):k\in\NN\}$, 
and a Borel probability law $\rho^\ast_n$ on $C_b^0([0,T];L^2(\CO))\times C_b^0([0,T];H^1(\CO))$ such that
$$
\Law((u^{(k)}_n,u^{(k)}_n)) \stackrel{k\to\infty}\longrightarrow \rho^\ast_n, 
$$
weakly. Then, by the Skorokhod Lemma, there exists a probability space $\tilde{\mathfrak{A}}_n=(\tilde \Omega_n,\tilde{ \mathcal{F}}_n,\tilde \PP_n)$,
a sequence of $C_b^0([0,T];L^2(\CO))\times C_b^0([0,T];H^1(\CO))$--valued random variables $\{(\tilde u^{(k)}_n,\tilde v^{(k)}_n):k\in\NN\}$ and $(\tilde u_n^\ast,\tilde v^\ast_n)$
such that
\begin{align}
\label{equallaw1}
\Law((\tilde u^{(k)}_n,\tilde v^{(k)}_n)) &= \Law (( u^{(k)}_n, v^{(k)}_n)),
\\
\label{equallaw2}
\Law((\tilde u^{(k)}_n,\tilde v^{(k)}_n)) &= \rho^\ast_n,
\end{align}
and $\tilde \PP_n$--a.s. on $C_b^0([0,T];L^2(\CO))\times C_b^0([0,T];H^1(\CO))$
\[
(\tilde u^{(k)}_n,\tilde v^{(k)}_n)
\longrightarrow (\tilde u^{*}_n,\tilde v^{\ast}_n) \quad \mbox{as} \quad k\to\infty.
\]

\medskip

\paragraph{{\bf Step (c):}}
For each $k\in\NN$, let $\tilde M_{1,n}^{(k)}$ and $\tilde M_{2,n}^{(k)}$ be the processes given by
\DEQSZ\label{qv1}
\\ \nonumber
\tilde  M_{1,n}^{(k)}(t) &:=&\tilde u^{(k)}_n(t)-u_0+\int_0^ t \Big[r_u \DeltaA \tilde u ^{(k)}_n(s)-\chi \psi_n(h^2(\tilde v ^{(k)}_n(s),s))\,\psi_n(h^1(\tilde u ^{(k-1)}_n(s),s))\,\\\nonumber
 &&{}\times  \psi_n(h^3(\tilde v ^{(k)}_n(s),s))\,\Div(\tilde  u ^{(k-1)}_n (s)\,\nabla \tilde v ^{(k)}_n(s))-\gamma\tilde  u ^{(k)}_n(s)\Big]\, ds
 \\
 \label{qv2}
\\ \nonumber
\tilde  M_{2,n}^{(k)}(t) &:=&\tilde v^{(k)}_n(t)-v_0 +\int_0^ t \Big[r_v\DeltaA \tilde v^{(k)}_n(s)-\alpha \tilde v^{(k)}_n(s) \Big] ds 
 -\beta \int_0^ t \tilde u^{(k-1)}_n(s) ds ,
 \EEQSZ
 and let 
 $$
 (\tilde \CG_n^{(k)})_t:=\sigma \lk( \lk\{\,(\tilde u_n^{(k)}(s), \tilde v_n^{(k)}(s),\tilde u_n^{(k-1)}(s)): s\le t, \, \rk\}\rk).
$$
Then,  $\tilde M_{1,n}^{(k)}$ and $\tilde M_{2,n}^{(k)}$  are local martingales over $(\tilde \Omega_n,\tilde \CF_n,\tilde \PP_n)$ with respect to the filtration $ (\CG_{n,t}^{(k)})_{t\in[0,T]}$ and
with quadratic variation
\begin{alignat}{2}
\label{qv11}
\la \tilde M_{1,n}^{(k)}\ra_t
&:=
\la \tilde M_{1,n}^{(k)}(t),\tilde M_{1,n}^{(k)}(t)\ra
& =&
\int_0^ t \sum_{j\in\mathbb{Z}}|\tilde u^{(k),j}_n(s)\psi^{(\delta)}_j|^2\, ds, \quad t\in[0,T],
\\
\label{qv22}
\la \tilde M_{2,n}^{(k)}\ra_t
&:=
\la \tilde M_{2,n}^{(k)}(t),\tilde M_{2,n}^{(k)}(t)\ra
& =&
\int_0^ t \sum_{j\in\mathbb{Z}}|\tilde v^{(k),j}_n(s)\psi^{(\delta)}_j|^2\, ds,\quad t\in[0,T].
\end{alignat}
The representations \eqref{qv11} and \eqref{qv22} {follow} from \eqref{equallaw1} and
{the fact that} $u^{(k)}_n=\MT_n(u^{(k-1)}_n)$; {see \eqref{vsolves}}.
To show that $\la \tilde M_{1,n}^{(k)}\ra$ and $\la \tilde M_{2,n}^{(k)}\ra$ are adapted to  the filtration $ (\CG_{n,t}^{(k)})_{t\in[0,T]}$ we
use the fact that by  \eqref{equallaw1}
for all measurable mapping $\phi_1:C_b^0([0,T];L^2(\CO))\to\RR$,  $\phi_2:C_b^0([0,T];H^1(\CO))\to\RR$, and all $0\le s\le t\le T$ we have
$$\EE \lk[\lk(\tilde M_{1,n}^{(k)}(t)-\tilde M_{1,n}^{(k)}(s)\rk) \phi_1(1_{[0,s)}\tilde u^{(k)}_n) \phi_2(1_{[0,s)}\tilde v^{(k)}_n)\rk]=0,
$$
$$\EE \lk[\lk(\tilde M_{2,n}^{(k)}(t)-\tilde M_{2,n}^{(k)}(s)\rk) \phi_1(1_{[0,s)}\tilde u^{(k)}_n) \phi_2(1_{[0,s)}\tilde v^{(k)}_n)\rk]=0.
$$
In addition, we have for all $w_1,w_2\in H^{\delta}$ for $\delta>1$,
\DEQS
\lqq{\tilde \EE \lk[  \la \tilde M_{1,n}^{(k)}(t),w_1\ra \la \tilde M_{1,n}^{(k)}(t),w_2\ra\rk.}&&
\\
&&{}\lk.- \la\tilde  M_{1,n}^{(k)}(s),w_1\ra \la\tilde  M_{1,n}^{(k)}(s),w_2\ra 
-  \int_s^ t \la \tilde u^{(k)}_n(r),w_1\ra \la\tilde  u^{(k)}_n(r),w_2\ra \, dr\rk]=0,
\EEQS
\DEQS
\lqq{\tilde \EE \lk[  \la\tilde  M_{2,n}^{(k)}(t),w_1\ra \la \tilde M_{2,n}^{(k)}(t),w_2\ra\rk.} &&
\\ &&\lk.{}- \la \tilde M_{2,n}^{(k)}(s),w_1\ra \la \tilde M_{2,n}^{(k)}(s),w_2\ra
-  \int_s^ t \la \tilde v^{(k)}_n(r),w_1\ra \la \tilde v^{(k)}_n(r),w_2\ra \, dr\rk]=0.
\EEQS

\paragraph{{\bf Step (d):}}
Next, {we verify the following statements with the aim to pass the limit}.
\begin{enumerate}
\item There exists a constant $C_n>0$ such that $ \sup_{k\in{\mathbb{N}}} \tilde \EE^n \Big[\sup_{0\le s\le T} |\tilde u_n^{(k)}
(s)|^2_{L^2}\Big]\le C_n$ for all $n\in\NN$ and
\item  for any $r\in (1,2)$ we have
$$\lim_{k\to \infty}\tilde {{\mathbb{E}}}^n\Big[\sup_{0\le s\le T}  |\tilde
 u_n^{(k)}(s)-\tilde u_n^{\ast} |_{L^2}^r\Big] = 0. $$
\end{enumerate}
Let us recall that by our construction which {uses the Skorokhod
Embedding Theorem}, the laws of $ u^{(k)}_n$ and $\tilde  u_n^{(k)}$ on $C^0_b([0,T];L^2(\CO))$ are identical
 for any $k\in{\mathbb{N}}$.
Hence,
$$\EE \Big[\sup_{0\le s\le T}| u_n^{(k)}(s)|^2_{L^2}\Big]= \tilde \EE^n\Big[ \sup_{0\le s\le T} |\tilde u_n^{(k)}(s)
|^2_{L^2}\Big]
$$
 and part (a) easily follows from the fact that $u_n^{(k)}\in\CM_\MA^2(0,T;L^2(\CO))$.
It follows from part (a) that for any $ r<2$  {the process} $ \sup_{0\le s\le T}| \tilde  u_n^{(k)}(s)|_{L^2}^{ r}$ is uniformly integrable with respect to  the probability measure $\tilde{\mathbb{P}}_n$. Since
$ \tilde{\mathbb{P}}_n$-a.s.\ $\tilde  u_n^{(k)}\to
u_n^\ast$ and  $ \sup_{0\le s\le T}| \tilde  u_n^{(k)}(s)|_{L^2}^{r}$ is uniformly integrable
w.r.t. the probability measure $\tilde{ \mathbb{P}}_n$, we obtain part (b) {thanks} to the applicability of the Vitali Convergence Theorem.

\paragraph{{\bf Step (e):}} 
For any $r<2$, since
$ \sup_{0\le s\le T}| \tilde  u_n^{(k)}(s)|_{L^2}^{ r}$ is uniformly integrable with respect to
the probability measure $\tilde{\mathbb{P}}_n$, it is integrable for~$r=2$. Moreover,
\[
\int_0^ t \sum_{j\in\mathbb{Z}}|\tilde u^{(k),j}_n(s)\psi^{(\delta)}_j|^2\, ds
+
\int_0^ t \sum_{j\in\mathbb{Z}}|\tilde v^{(k),j}_n(s)\psi^{(\delta)}_j|^2\, ds
\le
C\,\int_0^ T|\tilde u^{(k)}_n(s)|_{L^2}^2\, ds.
\]
Hence we can invoke the Vitali Convergence Theorem to take the limit in~\eqref{qv1}
and~\eqref{qv2} to define
\[
M_{1,n}^\ast:=\lim_{k\to\infty} M_{1,n}^{(k)} \quad\text{and}\quad
M_{2,n}^\ast:=\lim_{k\to\infty} M_{2,n}^{(k)}.
\]

Next, let
$\tilde {\mathbb{G}}^n=(\tilde {\mathcal{G}}^n_t)_{t\in [0,T]}$
be the
filtration defined by
\begin{equation*} \label{eq:filtrat}
 \tilde  {\mathcal{G}}_t^n= \sigma\left(\sigma( (\tilde u^\ast_n(s),\tilde v^\ast_n(s)) ;\, 0\le s\le t)\cup \mathcal{N}\right),\quad t\in[0,T],
\end{equation*}
where $\mathcal{N}$ denotes the set of null sets of $\tilde {\mathcal{F}}^n$.
Since  $u_n^{(k)}\in L^2(0,T;H^1(\CO))$, the limit is well defined and for $0 \leq s \leq t$
$$\EE \lk[\lk(\tilde M_{1,n}^\ast(t)-\tilde M_{1,n}^\ast(s)\rk) \phi_1(1_{[0,s)}\tilde u^\ast _n) \phi_2(1_{[0,s)}\tilde v^\ast_n)\rk]=0,
$$
and
$$\EE \lk[\lk(\tilde M_{2,n}^\ast (t)-\tilde M_{2,n}^\ast (s)\rk) \phi_1(1_{[0,s)}\tilde u^\ast _n) \phi_2(1_{[0,s)}\tilde v^\ast _n)\rk]=0.
$$
In addition, we have for all $w_1,w_2\in H^{\delta}$ with $\delta>1$,
\DEQS
\tilde \EE \lk[  \la \tilde M_{1,n}^\ast (t),w_1\ra \la \tilde M_{1,n}^\ast (t),w_2\ra- \la \tilde M_{1,n}^\ast (s),w_1\ra \la \tilde M_{1,n}^\ast (s),w_2\ra 
-  \int_s^ t \la \tilde u^\ast _n(r),w_1\ra \la \tilde u^\ast _n(r),w_2\ra \, dr\rk]=0,
\EEQS
\DEQS
\tilde \EE \lk[  \la \tilde M_{2,n}^\ast (t),w_1\ra \la \tilde M_{2,n}^\ast (t),w_2\ra- \la \tilde M_{2,n}^\ast (s),w_1\ra \la \tilde M_{2,n}^\ast (s),w_2\ra
-  \int_s^ t \la \tilde v^\ast _n(r),w_1\ra \la \tilde v^\ast _n(r),w_2\ra \, dr\rk]=0.
\EEQS
By the representation {theorem} \cite[Theorem 8.2]{pratozab}, there exists a filtered probability space
$(\tilde{\tilde{\Omega}}_n,\tilde{\tilde{\CF}}_n,\tilde{\tilde{\PP}}_n)$ with filtration
$(\tilde{\tilde{\CF}}^n_t)_{t\in[0,T]}$,
two cylindrical Wiener
processes $\tilde \CW_1$ and $\tilde \CW_2$ on $\CH_1$ and $\CH_2$ defined on
$\MA^\ast_n:=(\tilde \Omega _n \times \tilde{\tilde{\Omega}}_n,\tilde \CF_n\times \tilde{\tilde{\CF}}_n, 
\tilde \PP_n\times \tilde{\tilde{\PP}}_n)$ and adapted to $(\tilde \CF^n_t\times  \tilde{\tilde{\CF}}^n_t)_{t\in[0,T]}$
such that
$$
M_1(t,\omega,\tilde \omega):=\int_0^ t u_n^\ast(s,\tilde \omega) \, d\tilde \CW_1(s,\tilde \omega,\tilde{\tilde{\omega}}),\quad t\in[0,T],\,(\tilde \omega,\tilde{\tilde{\omega}})\in \tilde \Omega\times \tilde{\tilde{\Omega}},
$$
$$
M_2(t,\omega,\tilde \omega):=\int_0^ t v_n^\ast(s,\tilde \omega)) \, d\tilde \CW_2(s,\tilde \omega,\tilde{\tilde{\omega}}),\quad t\in[0,T],\,(\tilde \omega,\tilde{\tilde{\omega}})\in \tilde \Omega\times \tilde{\tilde{\Omega}}.
$$
In this way, we obtain a pair $(u^\ast_n,v^\ast_n)\in C^0_b([0,t];L^2(\CO))\times C^0_b([0,t];H^1(\CO))$ over $\MA^\ast$ and
two cylindrical Wiener processes $\CW_1$ and $\CW_2$ on $H^{\delta}$ for $\delta>1$ over $\MA^\ast$.

\medskip

\paragraph{{\bf Step (f):}}
In the next step we show that $(u^\ast_n,v^\ast_n)$ over $\MA^\ast$ together with $\CW_1$ and $\CW_2$ is indeed a martingale solution to \eqref{ppp1.n}.
Let us consider the Banach space
${\CM}^2_{\MA^\ast_n} (0,T;L^2(\CO))$ over $\MA^\ast_n$
and the convex subset $\tilde{\CC}_R^{n,\ast}.$
The operator $\tilde{ \MT} _n$ acts now on the space $\tilde{\CC}_R^{n}$ and is defined by system \eqref{ppp1}.
%
%

\medskip

Now, we will show in this step that for any $\ep>0$ we have
(here we denote the expectation over $\MA^\ast_n$ by $\EE^\ast_n$)
\[
\EE_n^\ast  \Big[\sup_{0\le s\le T} | \tilde{ \MT}_n(\tilde u^\ast_n)(s)-\tilde u_n^\ast(s)|_{L^2}\Big]\le \ep. %
\]
\medskip
Fix $\ep>0$.
We can write for all $n\in\NN$, $n>1$,
\DEQSZ\label{sum}
\tilde{ \MT}_n(\tilde u^\ast_n)-\tilde u_n^\ast &=&
\tilde u_n^\ast-\tilde u_n^{(k)}+
\tilde  u_n^{(k)} -
\tilde{\MT}_n (\tilde u_n^{(k-1)}) +\tilde {\MT}_n (u_n^{(k-1)})- 
\tilde{\MT}_n (\tilde u_n^\ast).
\EEQSZ
Note, since $\{\tilde u_n^{(k)}:k\in\NN\}$ converges to $\tilde u_n^\ast $ $\tilde \PP$--a.s.\
there exists a number $K_0$ such that
\DEQSZ\label{tn.01}
\EE^\ast_n \Big[ \sup_{0\le s\le T} |\tilde u_n^\ast(s) -\tilde u_n^{(k)}(s)|^r_{L^2}\Big]
&\le& \frac \ep 3, \quad \forall k\ge K_0.
\EEQSZ
Secondly, we know by \eqref{equallaw1}, \eqref{equallaw2}, and the construction of the sequence $\{\tilde u_n^{(k)}:k\in\NN\}$ that we have for all $k\in\NN$
\DEQSZ\label{tn.0}
 \EE^\ast_n\Big[\sup_{0\le s\le T} |\tilde u_n^{(k)}(s) -\tilde \MT_n(\tilde u_n^{(k-1)})(s)|^m_{L^2} \Big] &=&
 \EE^\ast_n \Big[ \sup_{0\le s\le T} |  u_n^{(k)}(s) - \MT_n(u_n^{(k-1)})(s)|^m_{L^2}\Big]=0.
\EEQSZ
Let us recall that we have shown that $\MT_n$ is continuous on $\CM^2_{{\MA}}(0,T;L^2(\CO))$.
Hence,  
there exists a $\delta>0$ such that for all $\xi,\eta\in \CM^2_{\MA^\ast}(0,T;L^2(\CO))$ with $\|\xi-\eta\|_{\CM^2_{\MA^\ast}(0,T;L^2)}\le \delta,$
we have
\DEQSZ\label{tn.1}
\|\MT_n(\xi)-\MT_n(\eta)\|^r_{\CM^2_{\MA^\ast}(0,T;L^2)} &\le &\frac \ep 3.
\EEQSZ
In addition, it follows from the identity $\Law(\tilde u_n^{(k)})=\Law(u_n^{(k)})$, the
property
$\tilde u_n^{(k)}\in \CM^2_{\tilde {\tilde{\MA_n}}}(0,T;L^2(\CO))$ for all  $k\in\NN$, and the
convergence of
$\{\tilde u_k^{(k)}:k\in\NN\}$ to $\tilde u_n^\ast $ that there exists a number $N_1>0$ such that
\DEQSZ\label{tn.2}
\EE^\ast_n \Big[\sup_{0\le s\le T}|\tilde u_n^{(k-1)}(s) -
\MT_n  (\tilde u_n^\ast )(s)|_{L^2}\Big]
&\le &\delta,
\EEQSZ
 for all $k\ge K_1$. In particular, applying the triangle inequality to \eqref{sum}, then the estimate \eqref{tn.01}, the identity \eqref{tn.0},  and the estimate 
 \eqref{tn.1} together with \eqref{tn.2}  
gives that for all $k\ge \max(K_0,K_1)$ we have 
\DEQS
\EE^\ast_n \Big[ \sup_{0\le s\le T}|\tilde{\MT}_n \tilde u_n^\ast (s)-\tilde u_n^\ast(s) |_{L^2}\Big] \le \ep. %
\EEQS
Hence, for all $t\in[0,T]$ we have $\PP$--a.s.
$$\tilde{\MT_n}(\tilde u_n^\ast )(t)=\tilde u_n^\ast (t),
$$
which is the assertion. This completes the proof of Lemma \ref{lem.exist.n}.
\end{proof}

\del{Let us remind that $M^2_\MA(L^2(\CO))$ is the space of all progressively measurable processes defined over the probability space
$\mathfrak{A}$. Besides, over the probability space $\mathfrak{A}$ are also the Wiener processes $\mathcal{W}_1$ and $\mathcal{W}_2$ defined. Let $\CH_1$ be the Hilbert space given by $ g\in L^2(\CO)$ such that $ \sum_{k\in\ZZ}\la f,\psi_k\ra ^2<\infty$ and $\la f,\psi_k\ra=\la f,\psi_{-k}\ra$, $k\in\NN$. Now, let
$\mathfrak{C}:=C_b^0([0,T];L^2(\CO))\times C_b^0([0,T];\CH_1)\times C_b^0([0,T];\CH_2)$
and let
 $\CP(\MC)$ be the space of probability measures on  given by laws of all triplets $(\xi,\mathcal{W}_1,\mathcal{W}_2)$, such that $\xi\in \CM^2_\MA(L^2(\CO))$ and $\mathcal{W}_1$, $\mathcal{W}_2$ are the given  Wiener processes on $\CH_1$ and $\CH_2$ over $\MA$.

For any $n\in\NN$ the operator $\MT_n$ induces an operator, denoted by  $\CT_n$, on the space of probability measures $\CP(\mathfrak{C})$. In particular, let $\chi\in \CP(\mathfrak{C})$. Then, there exists  a probability space $\tilde {\mathfrak{A}}$, given by  $\tilde{ \mathfrak{A}}=(\tilde \Omega, \tilde \CF,\tilde \PP)$, where $\tilde\Omega=\mathfrak{C}$, $\tilde{\CF}=\mathcal{B}(\mathfrak{C})$, and $\PP=\chi$, such that the projection on the first variable, i.e. $\pi_1(\omega)$, $\omega\in\tilde \Omega$, is a stochastic process, the projection on the second variable i.e. $\pi_2(\omega)$, $\omega\in\tilde \Omega$, is a $\CH_1$--valued Wiener process, and the projection onto the third variable, i.e.\ $\pi_3(\omega)$, $\omega\in\tilde \Omega$, is a $\CH_2$--valued Wiener process.
Let us denote $\tilde \xi(\omega)=\pi_1(\omega)$, $\tilde{\mathcal{W}}_1(\omega)=\pi_2(\omega)$, and $\tilde W_2(\omega)=\pi_3(\omega)$.
Next, let
$\tilde {\mathbb{F}}=(\tilde {\mathcal{F}}_t)_{t\in [0,T]}$
be the
filtration defined by
\begin{equation} \label{eq:filtrat}
 \tilde  {\mathcal{F}}_t= \sigma\left(\sigma( (\tilde \xi(s),\tilde{\mathcal{W}}_1(s), \tilde{\mathcal{W}}_2(s))) ;\, 0\le s\le t)\cup \mathcal{N}\right),\quad t\in[0,T],
\end{equation}
where $\mathcal{N}$ denotes the set of null sets of $\tilde {\mathcal{F}}$.
By similar consideration as done in \cite{reactdiff} one can show that $\tilde \xi$ is indeed a process
belonging to $\CM_{\tilde {\MA}}^2(0,T;L^2(\CO))$, i.e., a progressively measurable process over $\tilde {\mathfrak{A}}$
with respect the filtration $( \tilde  {\mathcal{F}}_t)_{t\in[0,T]}$. The Wiener processes $\tilde{\mathcal{W}}_1$ and $\tilde{\mathcal{W}}_2$ are two
independent Wiener processes satisfying Assumption \ref{wiener}.
Now we can define  $\CT_n(\chi)$ by
$$
\CT_n (\chi)=\Law( (\MT_n(\tilde \xi),\tilde{\mathcal{W}}_1,\tilde{\mathcal{W}}_2)).
$$
We will show in the next step that the operator $\CT_n$ is a compact operator on $\CP(\mathfrak{C})$.

\begin{claim}
For any $n\in\NN$,
there exists a bounded convex subset $\CM\subset \CP(\mathfrak{C})$ on which the operator $\CT_n$ is continuous and compact.
\end{claim}

\begin{proof}
{\bf Let
$$\CM:=\lk\{ \CP(\mathfrak{C}): \int_{\omega\in \MC} 
|\pi_1 (\omega) |^2_{C_b^2}\, \chi(d\omega)\le R\rk\}.
$$
Then, $\CM$ is convex, and  by the norm ... a convex subspace of ..}
By the Prohorov Lemma and the estimate \eqref{compact1} and \eqref{compact2} it is straightforward to verify that the operator
$\CT_n$ maps $\CM$ into a compact subset of $\CM$.

\del{In Claim \ref{claim1}-(a) we have shown that for any $n\in\NN$ there exists a convex bounded map $\MC$ given by a ball of radius $R>0$ such that  $\mathfrak{T}_n$ maps $\MC$ into $\MC$. By the Ascoli theorem (see \cite[p.\ 71, Lemma 1]{simon}) we know that $C_b^\beta([0,T];W^\alpha_2(\CO))\hookrightarrow C_b^0([0,T];L^2(\CO))$ compactly. Now, Claim- \ref{claim1}-(b), and the Prohorov Theorem shows that $\MT_n$ maps bounded subsets into compact subsets of $\CM_\MA ^2(0,T;L^2(\CO))$.
Finally, Claim \ref{continuity} gives that $\MT_n$ is continuous on $\CC$.}
\end{proof}

In particular, we will show that for any $\ep>0$ there exists a number $R>0$ and constants $\alpha,\sigma>0$, $\delta\in\RR$,
such that
$$
\PP\lk(  v^{(n)}\not \in K_R \rk) \le \ep,
$$
where
$$K_R:=\lk\{ \xi \in\mathbb{X}: |\xi|_{\mathbb{X}'}
\le R\rk\}.
$$
Since the embedding   $\mathbb{X}'\hookrightarrow \mathbb{X}$ is compact,
it follows that $K_R$ is compact in $\mathbb{X}$. Secondly, by item (b)
for any $v\in\mathcal{X}$,
\DEQS
\EE \psi(|\mathcal{V}(v)|_{\mathbb{X}'})\le C.
\EEQS
In particular, since $v^{(n)}\in\mathcal{X}$ and $v^{(n)}=\mathcal{V}(v^{(n-1)})$, it follows from assumption (b), that $v^{(n)}\in\mathcal{X}$, $n\in\NN$.
By the Chebycheff inequality it follows that the laws of
the sequence  $\{v^{(n)}:n\in\NN\}$ are tight on $\mathbb{X}$. Due to the fact that every Radon measure is tight and $\Law(W^{(n)})=\Law(W)$,
the laws of the sequence $(W^{(n)}):n\in\NN\}$  are tight on
$\mathbb{D}([0,T];H_1)$. Since the  marginal distributions of $\{(v^{(n)},W^{(n)}):n\in\NN\}$ are tight, we can conclude the  tightness of
the  whole sequence.

\begin{proof}
[Proof of Lemma \ref{lem.exist.n}:]
Now we can start with the proof of Lemma \ref{lem.exist.n}.
By a stochastic Schauder Theorem (see \cite{mejonas}), for any $n\in\NN$, there exists
a martingale solution  $(\mathfrak{A}_n, (W^n_1,W^n_2),(u_n,v_n))$, where $\mathfrak{A}_n$ is a filtered probability space denoted by $\mathfrak{A}_n=(\Omega_n,\CF_n,(\CF_t^n)_{t\in[0,T]}, \PP_n)$, $(W^n_1,W^n_2)$ are two independent Wiener processes,
 and  $(u_n,v_n)$ is a solution  to the system \eqref{ppp1.n} such that $(u_n,v_n)$  belongs $\PP_n$--a.s.\ to
$C_b^0([0,T];L^2(\CO))\times C_0^b([0,T];H^1(\CO))$.
\end{proof}
}

\item
In this  step we will show that the solution $(u_n,v_n)$ is non--negative.
For being short, we omit tedious steps and give a brief outline of the proof. For more details, one can follow e.g.  Theorem 2.3 in \cite{zab}, or Section 2.6 of \cite{roeckner}.
\begin{claim}\label{claim nonneg}
For any $n\in\NN$ the solutions $u_n$ and $v_n$ to the system \eqref{ppp1.n} are non--negative.
\end{claim}
To prove Claim \ref{claim nonneg}, we follow the steps involved in the proof of Theorem 2.6.2 of \cite{roeckner}. For being short and precise, we provide below a brief sketch of the proof.
\begin{proof}
For any $x \in L^2$ we set 
\[x^+:=\max\{x,0\}, \quad \quad x^-:=\max\{-x,0\}.
\]
Let us define for $\delta \in (0,1),$
$$g_\delta(r):=\frac {r^2}{\delta+r},\quad r\in (-\delta,\infty),\quad
\mbox{and} \quad G_\delta(r):=g_\delta((r^-)^2),\quad r\in\RR.
$$
Then, $G_\delta$ belongs to $C^2$ and $G_\delta(r)=G'_\delta(r)=G_\delta ''(r)=0$ for all $r\in[0,\infty)$,
and $|G_\delta'(r)|\le 2r^-$, and $ |G_\delta''(r)|\le 8$ for all $r\in \mathbb{R}$.
Now, let us define $\phi_\delta:L^2(\CO)\to\RR$ by
$$\phi_\delta(u_n)=\int_\CO G_\delta(u_n(x))\, dx,\quad u_n\in L^2(\CO).
$$
Then $\phi_\delta$ is twice Gateau differentiable on $L^2(\CO)$.
We now apply It\^o formula to $\phi_\delta(u_n(t)).$ However, to be precise, one should replace the Laplace operator $\Delta$ by its Yosida approximation $\frac 1 \ep (id-(\ep \DeltaA- id)^{-1})$, apply then the It\^o formula to $\phi_\delta(u_n(t)),$ where $(u_n^{\eps},v_n^{\eps})$ is the solution to system \eqref{ppp1.n} in which the Laplace operator is replaced %
 by its Yosida approximation, and then taking the limit $\ep \to 0$. For convenience we omit the step and apply the
the It\^o formula directly  to $\phi_\delta(u_n(t))$ and get
\begin{align*}
\EE[\phi_\delta(u_n(t))]
&=
 \phi_\delta(u_n(0))+
\EE \Big[\int_0^t \Big \langle D\phi_\delta(u_n(s)),r_u \DeltaA u_n(s)-\chi \Div( u_n(s)\nabla
v_n(s))+\gamma u_n(s)\Big \rangle \, ds \Big] 
\notag\\
&\quad+
\frac 12 \sum_{k\in\mathbb{Z}}\EE\Big[ \int_0^ t\int_\CO G_\delta''(u_n(s,x)) \lk[u_n(s,x)\psi^{(\delta_1)}_k(x)\rk]^2\, dx\, ds\Big].
\end{align*}
As $u_0 \in L^2_{+}(\CO),$ i.e., $u_0 \in L^2(\CO),\,\,u_0 \ge 0$ a.e. So $\phi_\delta(u_n(0))=0.$
Now using Neumann boundary conditions on $(u_n,\,v_n),\,\,G''_\delta(r)=0,\,r \ge 0$,
and integration by parts
we obtain
\begin{align*}
\EE[\phi_\delta(u_n(t))] \le C(r_u,\chi) \EE \Big[ \int_0^t \int_{\CO} | u_n^{-}(s,x)|^2 |\nabla v_n(s,x)|^2 dx\,ds \Big]+(2 \gamma+4) \EE \Big[\int_0^t |u_n^{-}(s)|_{L^2}^2\,ds  \Big].
\end{align*}
Letting $\delta \rightarrow 0$ and using Dominated Convergence Theorem, we have
\begin{align*}
\EE [|u_n^{-}(t)|_{L^2}^2] \le C(r_u,\chi) \EE \Big[ \int_0^t \int_{\CO} | u_n^{-}(s,x)|^2 |\nabla v_n(s,x)|^2 dx\,ds \Big]+(2 \gamma+4) \EE \Big[\int_0^t |u_n^{-}(s)|_{L^2}^2\,ds  \Big].
\end{align*}
Using Grownwall's Lemma we have $  \EE [|u_n^{-}(t)|_{L^2}^2]=0$ and this implies $$u_n^{-}(t)=0\quad \mbox{a.e.}\quad t \in [0,T],\,x \in \CO,\,\, \omega \in \Omega. $$
\dela{
 Taking the limit $\delta\to 0$ gives the assertion.
} Omitting $u_n$, then the positivity of $v_n$ follows again by the same arguments as before (or used by Tessitore and Zabczyk \cite{zab});  the
term $u_n$ can be added by comparison principle (see Kotelenez \cite{KOT}). Another possibility is to follow \cite[Section 2.6]{roeckner}. This completes the proof of Claim \ref{claim nonneg}.
\end{proof}

\item
Here, we will construct a  family of solutions $\{(\bar u_n,\bar v_n):n\in\NN\}$  following the solution to the  original problem until a stopping time $\bar \tau_n$. In particular,  we will introduce for each $n\in\NN$ a new pair of processes $(\bar u_n,\bar v_n)$ following the chemotaxis system up to the stopping time $\bar \tau_n$. Besides, we will have
$(\bar u_n,\bar v_n)|_{[0,\bar \tau_{n})}=(\bar u_{n+1},\bar v_{n+1})|_{[0,\bar \tau_n)}$.

\medskip

Let us start with $n=1$. From Lemma \ref{lem.exist.n}, we know there exists a martingale solution consisting of a probability space $\mathfrak{A}_1=(\Omega_1,\CF_1;(\CF^1_t)_{t\in[0,T]},\PP_1)$, two independent Wiener processes $(\mathcal{W}^1_1,\mathcal{W}^1_2)$
defined over  $\mathfrak{A}_1$, and
 a couple of processes $(u_1,v_1)$ solving $\PP_1$--a.s.\ the system
\[
\begin{cases}
& du_1(t) -\Big[r_u \DeltaA u_1(t) -\chi \,\psi_1(h^1(u_1,t))\psi_1(h^2(v_1,t))\,\psi_1(h^3(v_1,t))
\\
& \hspace{1.1cm} \times \Div( u_1 (t)\,\nabla v_1(t))+\gamma u_1(t)\Big] dt=  u_1(t) d \mathcal{W}^1_1(t),
\\
& dv_1(t) - \Big[r_v \DeltaA v_1(t)-\alpha
 v(t)\Big]\,dt=
 \beta u_1(t)dt+v_1(t) d \mathcal{W}^1_2(t),
\\ 
& (u_1(0),v_1(0)) = (u_0,v_0).
\end{cases}
\]
 Let us define now the stopping times
\begin{align*}
\tau_1^1 &:=\inf \{s\ge 0:|u_1(s)|_{L^1} \ge 1 \}, 
\\
\tau_1^2 &:=\inf \{s\ge 0:|\nabla v_1(s)|_{L^2}^2 \ge 1 \}, 
\\
\tau_1^3 &:=\inf \Big\{s\ge 0:\int_0^ s |\nabla v_1(r)|^2_{H^1}\, dr  \ge 1 \Big\}.
\end{align*}
Put $\tau^\ast_1:=\min(\tau^1_1,\tau_1^2,\tau_1^3)$. Observe, on the time
interval $[0,\tau_1^\ast)$, the pair $(u_1,v_1)$ solves the Chemotaxis system given in
\eqref{chemonoise}. 
Now, we define a new pair of processes $(\bar u_1,\bar v_1)$  following $(u_1,v_1)$  on $[0,\tau_1^\ast)$ and extend this processes
 to the whole interval $[0,T]$ in the following way.
First, we put $\overline{\mathfrak{A}}_1:=\mathfrak{A}_1$ and $\bar {\mathcal{W}}_1^j:=\mathcal{W}_1^j$, $j=1,2$,
let us introduce the processes $y_1$ and $y_2$ 
being a strong solution over $\overline{\mathfrak{A}}_1$ to
\begin{align}\label{eq1}
 y_1(t, u_1(\tau^\ast_1),\sigma) &= e^{t(r_u\DeltaA +\gamma I)} u_1(\tau^\ast_1)
+
 \int_{0}^t  e^{(r_u\DeltaA +\gamma I)(t-s)} y_1(s, u_1( \tau^\ast_1),\sigma)
\,d(\theta_{\sigma} \bar{\mathcal{W}}^1_1)(s)
\end{align}
and
\begin{align}\label{eq2}
 y_2(t, v_1( \tau^\ast _1),\sigma) 
&=
e^{(r_v\DeltaA -\alpha I)t} v_1(\tau^\ast_1)
+ 
\int_{0}^t  e^{(r_v\DeltaA -\alpha I)(t-s)} 
y_2 (s, v_1(\tau^\ast _1),\sigma) \,d(\theta_{\sigma} \bar{\mathcal{W}}^1_2)(s)
\end{align}
%
%
for~$t\ge0$.
Here $\theta_\sigma$ is the shift operator which maps $\mathcal{W}_j(t)$ to $\mathcal{W}_j(t+\sigma)$. 
Since the couple $(u_1,v_1)$  is continuous in $(L^2(\CO),H^1(\CO))$, we know that  $u_1(\tau^\ast_1)$ and $v_1(\tau^\ast_1)$ are well defined random variables belonging $\PP$--a.s.\ to $L^2(\CO)$ and $H^1(\CO)$, respectively.
Since, in addition,  $(e^{t(r_u\DeltaA +\gamma I)})_{t\ge 0}$ and  $(e^{t(r_v\DeltaA -\alpha I)})_{t\ge 0}$ are  analytic semigroups on $L^2 (\CO)$ and  $H^1(\CO)$,
the existence of a unique solution $y_j;\,j=1,2$  over $\overline{\mathfrak{A}}_1$ to \eqref{eq1} and \eqref{eq2}  in $L^2(\CO)$ and  $H^1(\CO)$, respectively,
can be shown by standard methods.
Now, let us define two  processes $\bar u_1 $ and $\bar v_1$
being identical to $u_1$ and $v_1$, respectively,  on the time interval $[0,\tau^\ast_1)$ and
following the heat equation (with lower order terms) with noise, i.e.,   $y_1(\cdot,u_1(\tau^\ast_1),\tau^\ast_1)$ and $y_2(\cdot,v_1(\tau^\ast_1),\tau^\ast_1)$, afterwards.
In particular, let
$$
\bar u_1  (t) = \bcase u_1(t) & \mbox{ for } 0\le t< \tau^\ast_1,\\
y _1(t,u_1(\tau^\ast_1),\tau^\ast_1) & \mbox{ for } \tau^\ast_1\le  t \le T,\ecase
$$
and
$$
\bar v_1  (t) = \bcase v_1 (t) & \mbox{ for } 0\le t< \tau_1^\ast,\\
y_2 (t,v_1(\tau^\ast_1),\tau_1^\ast
)  & \mbox{ for } \tau_1^\ast \le  t \le T.\ecase
$$
Let us now construct the probability space and the processes for the next time interval. First, let $(u_1(\tau^\ast_1),v_1(\tau^\ast_1))$ have probability measure $\mu_1$ on $L^2(\CO)\times H^1(\CO)$. Again, from Lemma \ref{lem.exist.n} we know
there is a martingale solution consisting of a probability space $\mathfrak{A}_2=(\Omega_2,\CF_2,(\CF^2_t)_{t\in[0,T]},\PP_2)$, two independent Wiener processes $(\mathcal{W}_2^1,\mathcal{W}_2^2)$,
 a couple of processes $(u_2,v_2)$ solving $\PP_2$--a.s.\ the system
\[
\begin{cases}
& du_2(t) -\Big[r_u \DeltaA u_2(t) -\chi \,\psi_2(h^1(u_2,t))\psi_2(h^2(v_2,t))\,\psi_2(h^3(v_2,t))
\\ 
& \hspace{1.1cm} \times \Div( u_2 (t)\,\nabla v_2(t))+\gamma u_2(t)\Big] dt=  u_2(t) d \mathcal{W}^2_1(t),
\\
& dv_2(t) - \Big[r_v \DeltaA v_2(t)-\alpha v(t)\Big]\,dt=
 \beta u_2(t)dt+v_2(t) d \mathcal{W}^2_2(t),
\end{cases}
\]
 with initial condition $(u_2(0),v_2(0))$ having law $\mu_1$.
 Let us define now the stopping times on $\mathfrak{A}_2$
\begin{align*}
\tau_2^1 &:=\inf \{s\ge 0:|u_2(s)|_{L^1} \ge 2 \}, 
\\
\tau_2^2 &:=\inf \{s\ge 0:|\nabla v_2(s)|^2_{L^2} \ge 2 \}, 
\\
\tau_2^3 &:=\inf \Big\{s\ge 0:\int_0^ s |\nabla v_2(r)|^2_{H^1}\, dr  \ge 2 \Big\},
\end{align*}
and $\tau^\ast_2:=\min(\tau^1_2,\tau_2^2,\tau_2^3)$.
Let $\overline{\mathfrak{A}}_2:=\overline{\mathfrak{A}}_1\times \mathfrak{A}_2$,
$\overline{\CF}_2:=\overline{\CF}_1\vee \CF_2$,
$$\overline{\CF}^2_t:=\bcase \overline{\CF}^1_t, & \mbox{if} \quad t<\tau_1^\ast,
\\ \CF^2_{t-\tau_1^\ast}, & \mbox{if}\quad  t\ge \tau_1^\ast.
\ecase
$$
Finally, let us put for $j=1,2$
$$\bar{\mathcal{W}}_2^j(t):=\bcase \bar{\mathcal{W}}^j_1(t), & \mbox{if} \quad t<\tau_1^\ast,
\\   \mathcal{W}^j_2({t-\tau_1^\ast})+\bar{\mathcal{W}}^j_1(\tau_1^\ast), & \mbox{if}\quad  t\ge \tau_1^\ast.
\ecase
$$
Now, let us define two  processes $\bar u_2 $ and $\bar v_2$
being identical to $\bar u_1$ and $\bar v_1$, respectively,  on the time interval $[0,\tau^\ast_1)$, being identical to $u_2$ and $v_2$ on the time interval $[\tau^\ast_1,\tau^\ast_1+\tau_2^\ast)$ and
following the heat equation (with lower order terms) with multiplicative noise afterwards.
Let us note, for any initial condition having distribution equal to $u_2(\tau_2^\ast)$ and $v_2(\tau^\ast_2)$ that there exists a strong solutions  $y_1(\cdot ,\cdot ,\tau^\ast_2+\tau^\ast_1)$ and $y_2(\cdot,\cdot ,\tau^\ast_2+\tau^\ast_1)$ of the systems \eqref{eq1} and \eqref{eq2}, respectively, on $\overline{\mathfrak{A}_2}$.
Let for $t\in[0,T]$
$$
\bar u_2  (t) = \bcase \bar u_1(t) & \mbox{ for } 0\le t< \tau^\ast_1,\\
  u_2(t-\tau_1^\ast ) & \mbox{ for }  \tau^\ast_1\le t\le \tau^\ast_1+\tau_2^\ast,\\
y _1(t-(\tau^\ast_1+\tau^\ast_2),u_2(\tau^\ast_2),\tau^\ast_1+\tau^\ast_2) & \mbox{ for } \tau^\ast_2+\tau^\ast_1\le  t \le T,\ecase
$$
$$
\bar v_2  (t) = \bcase v_1 (t) & \mbox{ for } 0\le t< \tau_1^\ast,\\
  v_2(t-\tau_1^\ast ) & \mbox{ for } \tau^\ast_1\le t\le \tau^\ast_1+\tau_2^\ast,\\
y_2 (t-(\tau^\ast_1+\tau^\ast_2),v_2(\tau^\ast_2),\tau_1^\ast+\tau^\ast_2
)  & \mbox{ for } \tau_1^\ast+\tau^\ast_2 \le  t \le T.\ecase
$$
In the same way we will construct for any $n\in\NN$ a probability space $\overline{\mathfrak{A}}_n$, a pair of Wiener processes $(\bar{\mathcal{W}}^1_n,\bar{\mathcal{W}}^2_n)$, over $\overline{\mathfrak{A}}_n$ and a pair of processes $(\bar u_n,\bar v_n)$ starting at $(u_0,v_0)$ and solving system  \eqref{chemonoise} up to time
%
%
%
\DEQSZ\label{stopp_time}\bar \tau_n:=\tau_1^\ast+\cdots +\tau_n^\ast
\EEQSZ and following the heat equation afterwards.
 Besides,  $(\bar u_n,\bar v_n)|_{[0,\bar \tau _{n-1}]}=(\bar u_{n-1},\bar v_{n-1})|_{[0,\bar \tau_{n-1}]}$.
\medskip
\del{
\item In the Step we will establish a uniform estimate of the $L^1$--norm of $\bar u_n$ independent of $n\in\NN$.
\begin{claim}\label{l1est}
For any $p>4$ there exists a constant $C>0$ such that for all $t\in[0,T]$ we have
$$
\EE |\bu(t)|_{L^1}^p\le \EE |u_0|^p_{L^1} +C\,\int_0^t \EE |\bu(s)|_{L^1}^p\, ds, \quad m\in\NN.
$$
\end{claim}

\begin{proof}[Proof of Claim \ref{l1est}:]
We can write by the variation of constant formula
\DEQS
\lqq{ \Big| \int_\CO \bu(t,x)\, dx \Big| \le \Big|\int_\CO (e^{(r_u \DeltaA -\gamma I)t}u_0)(x)\, dx\Big| }
&&
\\
&& +\Big|\int_0^ t\int_\CO  (e^{(r_u \DeltaA -\gamma I)(t-s)})\Div(\bu(s)\nabla \bar v_m(s))(x)\,ds\Big | + C\Big| \int_0^ t e^{(r_u \DeltaA -\gamma I)(t-s)} \bu(s)d\mathcal{W}_1(s)\Big|_{L^2}.
\EEQS
Since $H^{-\frac 12}(\CO)\hookrightarrow L^1(\CO)$, we can write  for $q>2$ and $2q=p$
\DEQS
\EE\Big| \int_0^ t e^{(r_u \DeltaA -\gamma I)(t-s)} \bu(s)d\mathcal{W}_1(s)\Big|^p_{L^2}
&\le &
C(\gamma) \EE  \Big( \int_0^ t (t-s)^{-\frac 12 } \big|\bu(s)\big|^2_{H^{-\frac 12}}\, ds\Big)^\frac p2
\\
\le  C(\gamma,t) \, \EE \Big( \int_0^ t\big|\bu(s)\big|^{2q}_{H^{-\frac 12}}\, ds\Big)^\frac {p}{2q}
&\le & C(\gamma,t) \,  \int_0^ t\EE \big|\bu(s)\big|^{p}_{L^1}\, ds.
\EEQS
Taking into account that $u_0\ge 0$ and that $(e^{-\DeltaA t})_{t\ge0}$ is the heat semigroup, it is easy to see that
$$
\Big|\int_\CO (e^{(r_u \DeltaA -\gamma I)t}u_0)(x)\, dx\Big|\le \lk|u_0\rk|_{L^1}.
$$
In addition, by integration by parts we obtain
$$
\Big|\int_0^ t\int_\CO  (e^{(r_u \DeltaA -\gamma I)(t-s)})\Div(\bu(s)\nabla \bar v_m(s))(x)\,ds\Big| =0, \quad m\in\NN.
$$
Collecting altogether gives the assertion.
\end{proof}
An application of the Grownwall Lemma gives that there exists a constant $C>0$ such that
\DEQSZ\label{l1norm1}
\sup_{0\le t\le T} \EE |\bu(t)|_{L^1}^p\le C(T)\, \EE \lk|u_0\rk|_{L^1}^p, \quad m\in\NN .
\EEQSZ
In a similar  way we can argue that for all $p>4$ there exists constant $C(T)>0$ such that
\DEQSZ\label{l1norm2}
\sup_{0\le t\le T} \EE |\bu(t)|_{H^{-\frac 12}}^p\le C(T)\, \EE \lk|u_0\rk|_{L^1}^p,\quad m\in\NN .
\EEQSZ

From \eqref{l1norm1} and \eqref{l1norm2} we obtain an estimate of the $\sup$-norm in $L^1(\CO)$.
}
\item
In this step, we will give uniform bounds on the processes $\{(\bar u_n,\bar v_n):n\in\NN\}$ by the Lyapunov functional given by
\DEQS 
W (u,v)= \int_\CO \Big(u(x) \log u(x) - \rho u(x)v(x)\Big) dx +
C_1|\nabla v|^2_{L^2} + C_2|v|_{L^2}^2,\quad u\in \LLa(\CO),\,\, v\in H^1(\CO).
\EEQS
In this way we can show the following claim.

\begin{claim}\label{oben}
There exists a constant $C>0$, such that we have for all $n\in\NN$
$$\EE\Big[ \sup_{0\le s\le T} |\bar u_n (s)|_{L^1}^2 \Big] 
\le C,
\quad 
\quad \EE \Big[\sup_{0\le s\le T} |\nabla \bar v_n(s)|_{L^2}^2\Big] \le C,\quad
\EE \Big[\int_0^T |\nabla \bar v_n(s) |_{H^1}^2\, ds \Big]\le C.
$$
\end{claim}

\begin{proof}[Proof of Claim \ref{oben}]
On the time interval $[0,\bar \tau_n ]$ we use the Lyapunov functional $W(\bar u_n,\bar v_n)$ (where $\bar \tau_n$ is defined in \eqref{stopp_time})and on $[\bar \tau_n,T]$
we will use the functional given by
\DEQS 
E (u,v)= \int_\CO u(x) \log u(x)\, dx +
C_1|\nabla v|^2_{L^2} + C_2|v|_{L^2}^2,\quad u\in \LLa(\CO),\,\, v\in H^1(\CO).
\EEQS
We note that $\EE \Big[\sup_{\bar \tau_n \le s \le T}\int_{\CO} \bar u_n(s) \log \bar u_n(s)\Big]$ is bounded below. Hence
$$\EE \Big[ \sup_{0\le s\le T}1_{[\bar \tau_n ,T]}(s) E(\bar u_n (s),\bar v_n (s))\Big]$$
is bounded below. To show $W(\bar u_n(s),\bar v_n(s))$ is bounded below on $[0,\bar \tau_n]$, it suffices to prove $$\EE \Big[
\sup_{0 \le s \le \bar \tau_n}\int_{\CO} \bar u_n(s) \bar v_n(s)  \Big]$$ is bounded below.
First, note that, due to the Young inequality, for any $\ep>0$ there exists $C(\ep)>0$ such that we have
\begin{align*}
\EE\Big[ \sup_{0 \le s \le \bar \tau_n} \int_\CO \bu (s,x) \bar v_n (s,x)\,dx\Big]
&\le  
\EE \Big[\sup_{0 \le s \le \bar \tau_n} |\bu (s)|_{L^1} |\bar v_n (s)|_{L^{\infty}}\Big]
\\
&\le 
C(\ep)  \EE\Big[ \sup_{0 \le s \le \bar \tau_n} |\bu (s)|^2_{L^1}\Big]+ \ep \EE\Big[ \sup_{0 \le
s \le \bar \tau_n} |\bar v_n (s)|_{H^1 }^2\Big].
\end{align*}
Taking $\ep$ small enough, $\EE\Big[\sup_{0 \le s \le \bar \tau_n} |\bar v_n (s)|^2_{H^{1}}\Big]$ can be handled by $\EE\Big[\sup_{0 \le s \le \bar \tau_n}\Big(C_1|\nabla v|^2_{L^2} + C_2|v|_{L^2}^2\Big)\Big]$.
To show the functional $W$ is bounded below, we need to show that there exists a constant $C>0$ such that
\begin{align}\label{esti bar un L1}
\EE\Big[ \sup_{0 \le s \le \bar \tau_n} |\bu (s)|^2_{L^1}\Big] \le C.
\end{align}
It remains to show that there exists a constant $C>0$ such that
\DEQSZ \label{lyp1.1}
\EE\Big[ \sup_{0\le s\le T}1_{[0,\bar \tau_n )}(s) W(\bu (s),\bar v_n (s))\Big]\le C \quad\mbox{and}\quad
\EE \Big[ \sup_{0\le s\le T}1_{[\bar \tau_n ,T]}(s) E(\bu (s),\bar v_n (s))\Big]\le C.
\EEQSZ
Estimate \eqref{lyp1.1} essentially implies \eqref{esti bar un L1} and
\[ \EE \Big[\sup_{0\le s\le T} |\nabla \bar v_n(s)|_{L^2}^2\Big] \le C,\quad
\EE \Big[\int_0^T |\nabla \bar v_n(s) |_{H^1}^2\, ds \Big]\le C.
\]
%
\begin{claim}\label{claim esti}
There exists a constant $C>0$ such that we have for any $n\in\NN$
\begin{align*}
&\EE \Big[\sup_{0\le s\le t\wedge \bar \tau_n }\Big( W(\bar{u}_n (s),\bar{v}_n (s))-W(\bar u_n(0),\bar v_n(0))\Big) \Big]
\notag\\
& \quad + \EE\int_0^ {t\wedge \bar \tau_n } \Bigg( 4 r_u  |\nabla \sqrt{\bar{u}_n (s)}|^ 2 _{L^2}
+(C_1 r_v -(\chi+\rho r_u+\rho r_v +C_1\beta )/ \rho \beta)|\DeltaA \bar{v}_n (s)|_{L^2}^2
\notag\\
&\quad +(C_2 r_v+C_1 \alpha ) |\nabla \bar{v}_n (s)|^2_{L^2} +C_2\alpha  | \bar{v}_n (s)|^2_{L^2}\,%
+{} \frac 12 \rho \beta |\bar{u}_n (s)|^2_{L^2}
+\chi\rho\, \la \bar{u}_n (s), (\nabla \bar{v}_n (s))^2 \ra \Bigg)\, ds \notag\\
&\le  C\Big\{ t \, + \EE  W(\bar u_n (0),\bar v_n (0))
 +C \, \EE\Big[\int_{0}^{t \wedge \bar \tau_n } \Big(\int_\CO |\log(\bar u_n (s,x))+1|\,|\bar u_n (s,x)|\, dx\Big)^2 ds\Big]^{\frac{1}{2}}.
\end{align*}
In particular, there exists a constant $C>0$ such that we have for all $m\in\NN$
\DEQS
\EE W(\bar{u}_n (t),\bar{v}_n (t))=C\Big( t \, + \EE  W(\bar u_n (0),\bar v_n (0))\Big).
\EEQS

%
\end{claim}


%

%
\begin{proof}[Proof of Claim \ref{claim esti}]
Fix $n\in\NN$. We will show the claim by applying the It\^o formula to $W(\bar u_n,\bar v_n)(t)$. To be precise, in equation \eqref{ppp1.n} one should replace before the Laplace operator $\DeltaA$ by its Yosida approximation $\DeltaA_\ep=\frac 1 \ep ((id-\ep \DeltaA)^{-1}-id)$ and the divergence operator by intertwining a mollifier $\phi_\ep$. One can then apply the It\^o formula to $W(\bar u_n^\ep,\bar v_n^\ep)$, where $(\bar u_n^\ep,\bar v_n^\ep)$ are the solutions to system \eqref{ppp1.n} in which the Laplace operator is replaced %
 by its Yosida approximation and then take the limit $\ep \to 0$ (see e.g.\ proof of Lemma 3.2 in \cite{marinelli}). Here, for any $\nu>0$ and $\delta\in\RR$ we know that $|\DeltaA-\DeltaA_\ep|_{L(H^{\delta+\nu},H^{\delta})}\to 0$ as $\ep\to 0$.
We now approximate the divergence operator. Let $\varphi\in C^{(2)}_b(\RR)$ be a function which is compactly supported, $ \int\!\varphi (x)\mathrm {d} x=1$, and $\lim _{\epsilon \to 0}\varphi _{\epsilon }(x)=\lim _{\epsilon \to 0}\epsilon ^{-1}
\varphi (x/\epsilon )=\delta (x)$.
Let $\mbox{div}_\ep(w):= \mbox{div}(\varphi_\ep\star w)$ where $\star$ denotes the convolution. 
Let
\[
\begin{cases}
&d \bar u^\ep_n(t) -
\Big[ r_u \DeltaA_\ep\bar u^\ep_n(t) -\chi \psi_n(h^1(\bar u^\ep_n,t))\,\psi_n(h^2(\bar
v^\ep_n,t))\,\psi_n(h^3(\bar v^\ep_n,t))
\\
&\hspace{1.1cm}
\times \Div_\ep( \bar u^\ep_n(t)\,\nabla\bar  v^\ep_n(t))+\gamma \bar u^\ep_n(t)\Big] dt
= \bar  u^\ep_n(t) d \mathcal W_1(t),
\\
& d \bar v^\ep_n(t) -
\Big[r_v \DeltaA_\ep \bar v^\ep_n(t)-\alpha \bar v^\ep_n(t)\Big]\,dt =\beta \bar u^\ep_n(t)dt+
 \bar v^\ep_n(t)d\mathcal W_2(t),
\\ 
& (\bar u^\ep_n(0),\bar v^\ep_n(0)) = (u_0,v_0).
\end{cases}
\]
Before applying the It\^o-formula we have to reassure that the integrands belong to the appropriate space (see \cite[Theorem 2.4]{brzezniakito} or \cite[Chapter 4.4]{pratozab}).
We note that
$$
W:L^2(\CO)\times H^1(\CO)\to \RR
$$
is well-defined. Due to Lemma \ref{lem.exist.n} for all $n\in\NN$, there exists a constant $C(n)>0$ such that 
\[\EE\Big[\sup_{0\le s\le T} |\bar u^\ep_n(s)|_{L^2}^2\Big]\le C(n), \qquad 
\EE\Big[\sup_{0\le s\le T} |\bar u_n(s)|_{L^2}^2\Big]\le C(n),\]
\[
\EE\Big[\sup_{0\le s\le T} |\nabla\bar  v^\ep_n(s)|_{L^2}^2\Big]\le C(n) \qquad \EE\Big[\sup_{0\le s\le T} |\bar v_n(s)|_{L^2}^2\Big]\le C(n).\] As the integrands satisfy the conditions of Theorem 2.4 in \cite{brzezniakito}, we can apply the It\^o formula to the process $W(\bar u^\ep_n(s), \bar v^\ep_n(s))$.
One can show $W(\bar u_n^\ep,\bar v_n^\ep)(t)\to W(\bar u_n,\bar v_n)(t)$ as $\ep\to 0$ for all $t \in [0,T]$. Indeed, by Claim \ref{claim1}, we know that there exists some $\alpha>0$ and a constant $C>0$ independent of $n\in\NN$ such that $$\EE\Big[\sup_{0\le s\le T} |\bar u^{\ep}_n(s)|_{H^\alpha}^2\Big]\le C ,\qquad 
\EE\Big[\sup_{0\le s\le T} |\bar u_n(s)|_{H^\alpha}^2\Big]\le C.$$
We only deal with the following terms in the expression of $W(\bar u^{\eps}_n, \bar v^{\eps}_n)-W(\bar u_n, \bar v_n)$ namely
\DEQS
\lqq{ \int_\CO ( \log \bar u_n (x)+1) \Big(r_u (\DeltaA_\ep-\DeltaA)  \bar u_n(x)\big)\, dx}
&&
\\
&\le& \Big(\int_\CO ( \log \bar u_n (x)+1)^2\, dx \Big)^\frac 12 \Big(\int_\CO ( r_u (\DeltaA_\ep-\DeltaA)\bar u_n(x))^2\, dx \Big)^\frac 12
\\
&\le & \Big(\int_\CO ( \log \bar u_n (x)+1)^2\, dx \Big)^\frac 12 r_u
|\DeltaA_\ep-\DeltaA|_{L(H^\alpha,L^2)}|\bar u_n|_{H^\alpha} \to 0
\EEQS
and
\DEQS
 \int_\CO ( \log \bar u_n (x)+1) \big( \Div (\bar u_n(x)\nabla \bar v_n(x))-\Div( (\varphi_\ep\star (\bar u_n \nabla \bar v_n)) (x)))\big)\, dx\to 0
\EEQS
as $\ep\to 0$. Hence, $W(\bar u_n^\ep,\bar v_n^\ep)(t)\to W(\bar u_n,\bar v_n)(t)$ as $\ep\to 0$ for all $t \in [0,T]$.
Other terms in the expression $W(\bar u^{\eps}_n, \bar v^{\eps}_n)-W(\bar u_n, \bar v_n)$ can be handled in a similar way. However, for our convenience, we omit the technical details and apply the It\^o formula directly  to $W(\bar u_n ,\bar v_n )(t)$.
Also, we will use the representation of the Wiener process given by
$$
\CW_j(t,x) =\sum_{k\in\mathbb{Z}} \psi^{(\delta_j)}_k(x)\beta^{(j)}_k(t),\quad j=1,2,
$$
where $\psi^{(\delta_j)}_k$; $j=1,2$ is defined in Remark \ref{rem wiener} and $\{\beta^j_k:k\in\mathbb{Z}\}$; $j=1,2$ are two families of mutually independent Brownian motions.
We now start with the following calculations
\DEQSZ\label{startwith}
\\\nonumber
\lqq{ W(\bar u_n ,\bar v_n )(t\wedge \bar \tau_n )-W(\bar u_n ,\bar v_n )(0)}
&&
\\ &=&\underbrace{\int_0^{t \wedge \bar \tau_n }
\nonumber   \la (\log \bar u_n (s)+1),  d\bar u_n (s)\ra}_{I_1:=}  \underbrace{-\rho \int_0^{t \wedge \bar \tau_n }  \la \bar u_n (s), d\bar v_n (s)\ra
 + \la \bar v_n (s) ,d \bar u_n (s)\ra }_{I_2:=}
\\ &&{} 
 +\underbrace{\nonumber \int_0^{t \wedge \bar \tau_n }
\la \nabla \bar v_n (s),\nabla d \bar v_n (s)\ra}_{I_3:=}  + \underbrace{ \int_0^{t \wedge \bar \tau_n }\la \bar v_n (s),d\bar v_n (s)\ra}_{I_4:=}\,
\\\nonumber
&&{} +\frac12 \sum_{k\in\mathbb{Z}} \int_0^{t \wedge \bar \tau_n }\int_\CO \bar u_n(s,x)\, (\psi_k^{(\delta_1)}(x))^2\, dx\, ds
\\
&&{}
+\frac {C_1}2 \sum_{k\in\mathbb{Z}}\int_0^{t \wedge \bar \tau_n }\lk|\nabla\lk( \bar v_n(s) \, \psi_k^{(\delta_2)}\rk)\rk|^2_{L^2}\, ds
+\frac {C_2}2 \sum_{k\in\mathbb{Z}}\int_0^{t \wedge \bar \tau_n }|\bar v_n(s) \, \psi_k^{(\delta_2)}|^2_{L^2}\, ds.\nonumber
\EEQSZ
\del{\\
&&{}
+  \sum_{k\in\mathbb{Z}}\int_0^{t \wedge \bar \tau_n }\int_\CO (1+\ln(\bar u_n(s,x))\bar u_n(s,x) \psi_k^{(\delta_1)}(x)\, d\beta_k^1(s)\nonumber
\\&&{}+  \sum_{k\in\mathbb{Z}}\int_0^{t \wedge \bar \tau_n }\la \nabla \bar v_n(s),\nabla (\bar v_n(s) \psi_k^{(\delta_1)})\ra \, d\beta_k^2(s)\nonumber
\\&&{}
+  \sum_{k\in\mathbb{Z}}\int_0^{t \wedge \bar \tau_n }\la\bar v_n(s),\bar v_n(s) \psi_k^{(\delta_1)}\ra \, d\beta_k^2(s)\nonumber
.
\EEQSZ}
Let us consider the first term in the right hand side of \eqref{startwith}. Here, we obtain
\DEQSZ\label{unter01}
\\\nonumber
I_1:=\lqq{ \int_0^{t \wedge \bar \tau_n } \la( \log \bar u_n (s)+1 ), d\bar u_n (s)\ra}
&&
\\\nonumber
&=&\int_0^{t \wedge \bar \tau_n } \int_\CO ( \log \bar u_n(s,x)+1) \Big(r_u \DeltaA  \bar u_n (s,x)-\chi \,\Div( \bar u_n (s,x)\nabla \bar v_n (s,x)) +\gamma \bar u_n (s,x) \Big)\, dx\, ds
\\ \nonumber
&&{}+\sum_{k\in\mathbb{Z}} \int_0^{t \wedge \bar \tau_n }
\int_\CO ( \log \bar u_n(s,x)+1) \bar u_n (s,x) \, \psi^{(\delta_1)}_k(x) \,dx \,d\beta^{(1)}_k(s)
.
\EEQSZ
Applying integration by parts one gets for the first two terms in \eqref{unter01}
\DEQS
\lqq{ \int_0^{t \wedge \bar \tau_n } \int_\CO ( \log \bar u_n (s,x)+1) \Big(r_u \DeltaA  \bar u_n(s,x)-\chi \,\Div( \bar u_n (s,x)\nabla \bar v_n (s,x))\Big)\, dx\, ds
} &&
\\
&=&
-\int_0^{t \wedge \bar \tau_n } \int_\CO r_u  \bar u_n ^{-1}(s,x) \nabla \bar u_n (s,x) (\nabla   \bar u_n (s,x)-\chi \,\bar u_n (s,x)\nabla \bar v_n (s,x)) \, dx\, ds
\\
&=& -\int_0^ {t \wedge \bar \tau_n } 4r_u  |\nabla \sqrt{\bar u_n (s)}|^ 2 _{L^2}\,ds -\chi \,\int_0^{t \wedge \bar \tau_n } \int_\CO   \bar u_n (s,x) A \bar v_n (s,x)\, dx\, ds.
\EEQS
Applying the Burkholder-Davis-Gundy inequality (see \eqref{HSnorm}) and then the H\"older inequality  give
\DEQS
\lqq{ \EE\Big[ \sup_{0\le s\le {t \wedge \bar \tau_n }}\Big|\sum_{k\in\mathbb{Z}}
 \int_0^{t \wedge \bar \tau_n } \int  \la \log \bar u_n (s)+1, \bar u_n (s)\psi^{(\delta_1)}_k\ra\, d\beta^{(1)}_k(s)\Big|\Big]
}
&&
\\ &\le& 4 \EE \Big[\Big(\sum_{k\in\mathbb{Z}}\int_0^{t \wedge \bar \tau_n }  \, \la \log \bar u_n (s)+1, \bar u_n (s)\,\psi^{(\delta_1)}_k\ra ^2\, ds\Big)^\frac 12 \Big]
\\
&\le &C \, \EE\Big[\int_{0}^{t \wedge \bar \tau_n } \Big(\int_\CO |\log(\bar u_n (s,x))+1|\,|\bar u_n (s,x)|\, dx\Big)^2 ds\Big]^{\frac{1}{2}}
.
\EEQS
Here, we used the fact that for $\delta_1>1$ there exists a $\sigma>\frac 12$ such that $H^{\delta_1}(\CO)\hookrightarrow H^\sigma(\CO)$ is Hilbert--Schmidt and
the embedding $ H^\sigma(\CO)\hookrightarrow L^\infty(\CO)$ is continuous.

We now estimate the next term $I_2$, that is, we deal with $\int_\CO \bar u_n (t\wedge \bar \tau_n ,x) \bar v_n (t\wedge \bar \tau_n ,x)\, dx$.
Applying again the It\^o formula we obtain
\DEQS
\lqq{ I_2=\int_\CO  \bar u_n (t\wedge \bar \tau_n ,x)\bar v_n (t\wedge \bar \tau_n ,x)\ dx -\int_\CO  u_0(x)v_0(x)\ dx}
&&
\nonumber\\
&=& \int_0^ {t\wedge \bar \tau_n } \int_\CO (\bar v_n (s,x) (r_u \DeltaA \bar u_n (s,x)-\chi \, \Div ( \bar u_n (s,x) \nabla \bar v_n (s,x)) -\gamma \bar u_n (s,x))\, dx \, ds
\nonumber\\
&&{}+
\int _0^{t\wedge \bar \tau_n } \int_\CO   \bar u_n (s,x) (r_v  \DeltaA \bar v_n (s,x)-\alpha \bar v_n (s,x)+\beta \bar u_n (s,x))\, dx\, ds
\nonumber\\
&&{} +\sum_{k\in\mathbb{Z}} \int_0^ {t\wedge \bar \tau_n } \la \bar v_n (s) , \bar u_n (s)\,\psi^{(\delta_1)}_k\ra d\beta^{(1)}_k(s)
 + \sum_{k\in\mathbb{Z}}\int_0^ {t\wedge \bar \tau_n } \la \bar u_n (s) , \bar v_n (s)\,\psi^{(\delta_2)}_k\ra d\beta^{(2)}_k(s)
\del{\\
&&{} + \sum_{k\in\mathbb{Z}}\int_0^ {t\wedge \bar \tau_n } \la \bar v_n (s)\psi^{(\delta_2)}_k,\bar u_n (s)\psi^{(\delta_1)}_k\ra \, ds}
.
\EEQS
Using integration by parts and the Neumann boundary conditions, we obtain
\DEQS
 r_u  \int_\CO \bar v_n (s,x) \DeltaA \bar u_n (s,x)\, dx = - r_u \int_\CO \nabla \bar v_n (s,x) \nabla  \bar u_n (s,x)\, dx=r_u \int_\CO A \bar v_n(s,x) \bar u_n(s,x),
\EEQS
\DEQS
\chi \int_\CO  \bar v_n (s,x) \Div ( \bar u_n  (s,x)\nabla \bar v_n (s,x)) \, dx = -\chi \int_{\CO} \bar u_n (s,x) (\nabla \bar v_n (s,x))^2 \, dx .
\EEQS
Proceeding in similar lines we obtain
\begin{align*}
 &
 \int_\CO \bar v_n(-\gamma \bar u_n)+ \bar u_n (s,x)(-\alpha \bar v_n (s,x)+\beta \bar u_n (s,x) )\, dx
\notag\\& 
 =
-(\alpha+\gamma) \int_\CO \bar u_n (s,x)\bar v_n (s,x)\, dx +\beta  \int_\CO \bar u_n ^2(s,x) \, dx.
\end{align*}
As before we know  that for $\delta_1>1$ there exists a $\sigma>\frac 12$ such that $H^{\delta_1}(\CO)\hookrightarrow H^\sigma(\CO)$ is Hilbert--Schmidt and
the embedding $ H^\sigma(\CO)\hookrightarrow L^\infty(\CO)$ is continuous.
This gives 
\DEQS
\lqq{ \EE\Big[\sup_{0\le s\le t\wedge \bar \tau_n } \, \Big|\sum_{k\in\mathbb{Z}} \int_0^ s \la \bar v_n (r) , \bar u_n (r)\,\psi^{(\delta_1)}_k\ra d\beta^{(1)}_k(r) \Big|\Big] }
&&
\\ &\le & \EE\Big[ \lk( \sum_{k\in\mathbb{Z}} \int_0^ {t\wedge \bar \tau_n } \la \bar v_n (s) , \bar u_n (s)\,\psi^{(\delta_1)}_k\ra^2 ds\rk)^\frac 12
\Big] \le  C\Big(\sum_{k\in\mathbb{Z}}|\psi^{(\delta_1)}_k|_{L^\infty}^2 \,\EE\Big[\int_{ 0}^{t\wedge \bar \tau_n } \la \bar v_n (s) , \bar u_n (s)\ra^2 ds
\Big]\Big)^\frac12.
\EEQS
Arguing in similar lines, we treat the following term. 
This gives 
\DEQS
\lqq{ \EE\Big[\sup_{0\le s\le t\wedge \bar \tau_n } \, \Big|\sum_{k\in\mathbb{Z}} \int_0^ s \la \bar u_n (r) , \bar v_n (r)\,\psi^{(\delta_2)}_k\ra d\beta^{(2)}_k(r) \Big|\Big] }
&&
\\ &\le & \EE\Big[ \lk( \sum_{k\in\mathbb{Z}} \int_0^ {t\wedge \bar \tau_n } \la \bar u_n (s) , \bar v_n (s)\,\psi^{(\delta_2)}_k\ra^2 ds\rk)^\frac 12
\Big] \le  C\Big(\sum_{k\in\mathbb{Z}}|\psi^{(\delta_2)}_k|_{L^\infty}^2 \,\EE\Big[\int_{ 0}^{t\wedge \bar \tau_n } \la \bar u_n (s) , \bar v_n (s)\ra^2 ds
\Big]\Big)^\frac 12.
\EEQS
Tackling the terms $I_3$ and $I_4$ in \eqref{startwith} we obtain
\DEQS
\lqq{\frac 12 |\nabla \bar v_n ({t\wedge \bar \tau_n } )|^2_{L^2}-\frac 12|\nabla v _0|^2_{L^2} +r_v\int_0^{t\wedge \bar \tau_n } |\DeltaA \bar v_n (s)|_{L^2}^2\, ds  }
&&
\\
&& +\alpha \int_0^{t\wedge \bar \tau_n }  |\nabla \bar v_n (s)|^2_{L^2} \, ds=-\beta \int_0^{t\wedge \bar \tau_n } \int_\CO A \bar v_n (s,x) \bar u_n (s,x)\, dx \, ds
\\
&&{} +\sum_{k\in\mathbb{Z}} \int_0 ^{t\wedge \bar \tau_n } \la \nabla \bar v_n (s),\nabla (\bar v_n (s)\,\psi^{(\delta_2)}_k)\ra d\beta^{(2)}_k(s)
+\frac 12 \sum_{k=1}^\infty \int_0 ^{t\wedge \bar \tau_n } | \nabla \bar v_n (s)\,\psi_k^{(\delta_2)}|_{L^2}^2\, ds
\EEQS
and
\DEQS
\lqq{ \frac 12| \bar v_n (t)|^2_{L^2}-\frac 12| v_0|^2_{L^2}+r_v \int_0^{t\wedge \bar \tau_n } |\nabla \bar v_n (s)|_{L^2}^2 \, ds +\alpha \int_0^{t\wedge \bar \tau_n } | \bar v_n (s)|^2_{L^2}\, ds }
&&
\\&=&  \beta  \int_0^{t\wedge \bar \tau_n } \la  \bar v_n (s), \bar u_n (s)\ra \, ds
+\sum_{k=1}^\infty \int_0 ^{t\wedge \bar \tau_n } \la  \bar v_n (s), (\bar v_n (s)\, \psi^{(\delta_2)}_k)\ra d\beta^{(2)}_k(s)
\\
&&\quad{}
+\frac 12 \sum_{k=1}^\infty \int_0 ^{t\wedge \bar \tau_n } | \bar v_n (s)\,\psi_k^{(\delta_2)}|_{L^2}^2\, ds
.
\EEQS
Again, by similarly consideration as above and the fact that $\delta_2>\frac 32$, we know 
\DEQS
\lqq{ \EE\Big[\sup_{0\le s\le t\wedge \bar \tau_n } \,\Big| \int_0 ^{s} \la \nabla \bar v_n (s),\nabla (\bar v_n (s)\,\psi^{(\delta_2)}_k)\ra d\beta^{(2)}_k(s)
 \Big|\Big] }
&&
\\ &\le & C\,\EE\,\Big( \sum_{k=1}^ \infty |\psi^{(\delta_2)}_k|_{L^\infty}^2 \int_0^ {t\wedge \bar \tau_n } |\nabla \bar v_n (s)|_{L^2}^4 ds\Big)^\frac 12  
\\ &\le &
  \frac{C_1}{2} \,\EE\Big[\sup_{ 0\le s \le t\wedge \bar \tau_n } |\nabla \bar v_n (s)|_{L^2}^2\Big]
+ C \EE\Big[\int_0^{t\wedge \bar \tau_n } |\nabla \bar v_n (s)|_{L^2}^2 ds\Big]
\EEQS
and
\DEQS
\lqq{ \EE\Big[\sup_{0\le s\le t\wedge \bar \tau_n } \,\Big| \int_0 ^{s} \la  \bar v_n (s), (\bar v_n (s)\,\psi^{(\delta_2)}_k)\ra d\beta^{(2)}_k(s)
 \Big| \Big]}
&&
\\ &\le & \EE\, \Big(\sum_{k=1}^ \infty \int_0^ {t\wedge \bar \tau_n } \la \bar v_n (s) ,  \bar v_n (s)\,\psi^{(\delta_2)}_k\ra^2 ds\Big)^\frac 12  \\ &\le &
  \frac{C_2}{2} \,\EE\Big[\sup_{ 0\le s \le t\wedge \bar \tau_n } | \bar v_n (s)|_{L^2}^2\Big]
+ C \EE\Big[\int_0^{t\wedge \bar \tau_n } | \bar v_n (s)|_{L^2}^2 ds\Big].
\EEQS
Recollecting altogether, we obtain at the end
\DEQSZ\label{fin eqn w}
\lqq{ \EE\Big[\sup_{0\le s\le t\wedge \bar \tau_n }\big[ W(\bar u_n (s),\bar v_n (s))-W(u_0,v_0) \big]\Big] + \EE\int_0^ {t\wedge \bar \tau_n } \Bigg( 4 r_u  |\nabla \sqrt{\bar u_n (s)}|^ 2 _{L^2}
+C_1 r_v |\DeltaA \bar v_n (s)|_{L^2}^2 } &&
\notag\\
&&{} +(C_2 r_v+C_1 \alpha ) |\nabla \bar v_n (s)|^2_{L^2} +C_2\alpha | \bar v_n (s)|^2_{L^2}\,
+{} 
\rho \beta |\bar u_n (s)|^2_{L^2}
+\chi\,\rho\, \la \bar u_n (s), (\nabla \bar v_n (s))^2 \ra \Bigg)\, ds
\notag\\
&\le & \EE\int_0^{t\wedge \bar \tau_n } \Big(   (\rho\alpha+\rho \gamma +C_2\beta) \la \bar u_n (s),\bar v_n (s)\ra  +
(\chi+\rho r_u+\rho r_v+C_1\beta )|\la u_n (s), A \bar v_n (s)\ra| \Big)\, ds
\notag\\
&&{}
+ C \, \EE\Big[\int_{0}^{t \wedge \bar \tau_n } \Big(\int_\CO |\log(\bar u_n (s,x))+1|\,|\bar u_n (s,x)|\, dx\Big)^2 ds\Big]^{\frac{1}{2}}
\notag\\
&&{}+\frac{C_1}{2} \,\EE\Big[\sup_{ 0\le s \le t\wedge \bar \tau_n } |\nabla \bar v_n (s)|_{L^2}^2\Big]
+ C \EE\Big[\int_0^{t\wedge \bar \tau_n } |\nabla \bar v_n (s)|_{L^2}^2 ds\Big] +
C \,\EE\Big[\int_{ 0}^{t\wedge \bar \tau_n } \la \bar v_n (s) , \bar u_n (s)\ra^2 ds
\Big]^\frac12
\notag\\
&&{}
+
\frac{C_2}{2} \,\EE\Big[\sup_{ 0\le s \le t\wedge \bar \tau_n } | \bar v_n (s)|_{L^2}^2\Big]
+ C \EE\Big[\int_0^{t\wedge \bar \tau_n } | \bar v_n (s)|_{L^2}^2 ds\Big]
.
\EEQSZ
However, by the Young inequality we have for $\ep_0=\frac 14 \rho\beta/ (\rho\alpha +\rho \gamma+C_2 \beta)$
$$
\la \bar u_n (s),\bar v_n (s)\ra \le C(\ep_1) |\bar v_n (s)|_{L^2}^2+\ep_0|\bar u_n (s)|_{L^2}^2
$$
and for $\ep_1=\frac 14 \rho\beta/ (\chi+\rho r_u+\rho r_v +C_1\beta )$
$$|\la  \bar u_n (s),A \bar v_n (s)\ra| \le C(\ep_1)|\DeltaA \bar v_n (s)|_{L^2}^2+\ep_1|\bar u_n (s)|_{L^2}^2,$$
where $C(\eps_1)= (\chi+\rho r_u+\rho r_v +C_1\beta )/ \rho \beta$. In particular, the terms in the right hand side of \eqref{fin eqn w} can be estimated in the following way
\DEQS
\lqq{ \EE\Big[ \sup_{0\le s\le t\wedge \bar \tau_n }\big[ W(\bar u_n (s),\bar v_n (s))-W(u_0,v_0) \big] \Big]}
&&
\\
&&{} + \EE\int_0^ {t\wedge \bar \tau_n } \Bigg( 4 r_u  |\nabla \sqrt{\bar u_n (s)}|^ 2 _{L^2}
+C_1 r_v |\DeltaA \bar v_n (s)|_{L^2}^2  +(C_2 r_v +C_1 \alpha ) |\nabla \bar v_n (s)|^2_{L^2}
\\
&&{} +C_2\alpha  | \bar v_n (s)|^2_{L^2}\,
+{} 
\frac 14 \rho \beta |\bar u_n (s)|^2_{L^2}
+\chi\rho\, \la \bar u_n (s), (\nabla \bar v_n (s))^2 \ra \Bigg)\, ds \\
\\
&\le & \EE\int_0^{t\wedge \bar \tau_n } \Big(    C(\ep_0) | \bar v_n (s)|^2_{L^2}  + C(\ep_1) | \DeltaA  \bar v_n (s)|^2_{L^2}
 \Big)\, ds
\\
&&{}
+ C \,({t \wedge \bar \tau_n })^\frac 12  \, \Big( \EE\Big[\sup_{0\le s\le {t \wedge \bar \tau_n }} \Big(\int_\CO |\log(\bar u_n (s,x))+1|^2\,|\bar u_n (s,x)|^2\, dx\Big)^\frac 12\Big]\Big)
\\
&&{}+\frac{C_1}{2}\EE\Big[\sup_{ 0\le s \le t\wedge \bar \tau_n } |\nabla \bar v_n (s)|_{L^2}^2\Big]
+
C\EE\Big[\int_0^{t\wedge \bar \tau_n} |\nabla \bar v_n (s)|_{L^2}^2\, ds\Big]
\\
&&{}
+\frac{C_2}{2}\EE\Big[\sup_{ 0\le s \le t\wedge \bar \tau_n } | \bar v_n (s)|_{L^2}^2\Big]
+
C\EE\Big[\int_0^{t\wedge \bar \tau_n} |\bar v_n (s)|_{L^2}^2\, ds\Big]
.
\EEQS
Using Gronwall inequality we get the follwing 
\DEQS
\lqq{ \EE\Big[ \sup_{0\le s\le t\wedge \bar \tau_n }\big[ W(\bar u_n (s),\bar v_n (s))-W(u_0,v_0) \big] \Big]}
&&
\\
&&{} + \EE\int_0^ {t\wedge \bar \tau_n } \Bigg( 4 r_u  |\nabla \sqrt{\bar u_n (s)}|^ 2 _{L^2}
+(C_1 r_v-C(\eps_1) |\DeltaA \bar v_n (s)|_{L^2}^2  +(C_2 r_v +C_1 \alpha ) |\nabla \bar v_n (s)|^2_{L^2}
\\
&&{} +C_2\alpha  | \bar v_n (s)|^2_{L^2}\,
+{} 
\frac 12 \rho \beta |\bar u_n (s)|^2_{L^2}
+\chi\rho\, \la \bar u_n (s), (\nabla \bar v_n (s))^2 \ra \Bigg)\, ds \\
\\
&\le & \EE\int_0^{t\wedge \bar \tau_n }   C(\ep_0) | \bar v_n (s)|^2_{L^2}  ds +
C\EE\int_0^{t\wedge \bar \tau_n}\Big[ |\nabla \bar v_n (s)|_{L^2}^2\Big] ds
\\
&&{}
+ C \, \EE\Big[\int_{0}^{t \wedge \bar \tau_n } \Big(\int_\CO |\log(\bar u_n (s,x))+1|\,|\bar u_n (s,x)|\, dx\Big)^2 ds\Big]^{\frac{1}{2}}
.
\EEQS
In the above inequality we choose $C_1> (r_u+r_v) /\beta r_v$ and $\rho > (\chi+C_1 \beta)/(C_1 \beta r_v -r_v-r_u)$, then $C_1 r_v-C(\eps_1)>0$.
This completes the proof of Claim \ref{claim esti}.

\end{proof}
It remains to estimate the two entities $|\bar u_n(\cdot)|_{\LLn}$ and $|\nabla \bar v_n(\cdot)|_{L^2}$ in
the time interval $[\bar \tau_n ,T]$. This we can do using  standard estimates for the heat equation
with multiplicative Wiener noise. In particular, applying the It\^o formula we obtain for $u\in \LLa(\CO)$ and $v\in H^1(\CO)$
\DEQS
\CE(u,v):= \int_\CO u(x)\, \log( u(x) )\, dx+C_1|\nabla v|_{L^2}^2,
\EEQS
and by elementary calculations we obtain
\DEQS
\lqq{  \bar u_n(T)\log(\bar u_n(T))=   \bar u_n(\bar \tau_n )\log(\bar u_n(\bar \tau_n )) + \int_{\bar \tau_n}^T \int_\CO \DeltaA \bar u_n(t,x) \, (\log( {\bar u_n}(t,x))+1)\, dx\,dt}
\\
&&{}  +
  \int_{\bar \tau_n}^ t \sum_{k\in\mathbb{Z}} \la  \log(1+\bar u_n(s)) \bar u_n(s)\,\psi^{(\delta_1)}_k\ra  d\beta^{(1)}_k(s)\, +
  \frac 12\sum_{k\in\mathbb{Z}}   \int_{\bar \tau_n}^ t \int_\CO\bar u_n(s,x)(\psi_k^{(\delta_1)}(x))^2\,dxds
.\EEQS
Similar computations works also for $|\nabla \bar v_n(\cdot)|^2_{L^2}$. In particular, we obtain
\DEQS
\lqq{ |\nabla \bar v_n(T)|_{L^2}^2-|\nabla \bar v_n(\bar \tau_n )|_{L^2}^2 - \int_{\bar \tau_n }^T
\la \DeltaA \bar v_n(s), \DeltaA \bar v_n(s) \ra \, ds}
\\& = &\int_{\bar \tau_n}^ T \sum_{k\in\mathbb{Z}}\la \nabla \bar v_n(s),  \nabla \bar v_n(s)\,\psi^{(\delta_2)}_k\ra  d\beta^{(2)}_k(s)\,
  +   \frac 12\sum_{k\in\mathbb{Z}}   \int_{\bar \tau_n}^ T \lk|\nabla (\bar v_n(s)\psi_k^{(\delta_2)})\rk|_{L^2}^2\, ds
.
\EEQS
The stochastic integrals can be estimated again by \eqref{HSnorm} and \eqref{HSnorm.1}.
\medskip

Let us summarize. We have shown that there exists a constant $C>0$ such that for any $t\in [0,T]$ we have
\DEQS
 \EE\Big[ \sup_{0\le s\le t}
 \lk[ \mathcal{E}(\bar u_n  (s), \bar v_n  (s))  -\mathcal{E}(u(0),\bar v_n(0)\rk]\Big]
 &\le &  C\int_0^ t \EE \Big[ \sup_{0\le \sigma \le s}  \CE\lk( \bar v_n  ( \sigma),\bar u_n( \sigma)\rk)\, ds\Big]. 
\EEQS
The Grownwal inequality gives 
$$
\EE\Big[ \sup_{0\le s\le t}  \lk[ |\nabla \bar v_n  (s)|_{L ^2} ^2+|\bar u_n  (s)|_{\LLn} \rk]\Big]\le C \lk(\EE \, \CE(u_0,v_0)+1\rk).
$$
In addition, we know that $\EE \, W(u_0,v_0)\le \EE \lk[ |\nabla v_0|_{L^2}^2+|u_0|_{\LLn}\rk]$.
This completes the proof of Claim \ref{oben}.
\end{proof}
\item
In the final step, we will show that $\PP$--a.s. a martingale  solution to \eqref{eq1} exists.


\begin{claim} There exists a measurable set $\tilde \Omega\subset \Omega$ with $\PP(\tilde \Omega)=1$ such that on $\tilde \Omega$ a martingale solution $(u,v)$ of \eqref{chemonoise} exists.
\end{claim}

\begin{proof}
For any $m\in\NN$
let us define the set
		$$A_m :=
		\lk\{ \omega\in \Omega: \tau^\ast_m\ge T\rk\}.
		$$
		It can be clearly observed that
		there exists a progressively measurable process $(u,v)$ over $\mathfrak{A}$ 
such that $(u,v)$ solves $\PP$--a.s. the integral equation given by
		\eqref{chemonoise} up to time $T$. In particular, we have for the conditioned probability
		$$
		\PP\lk( \{ \mbox{there exists a solution $(u,v)$ to \eqref{chemonoise}} \} \mid A_m\rk)=1.
		$$
		Then{,} it is elementary to verify, that
\[
\PP(\{ \mbox{there exists a solution to \eqref{chemonoise}} \})
=
\lim_{m\to \infty} 
\PP(  \{ \mbox{there exists a solution $(u,v)$ to \eqref{chemonoise}} \} \mid A_m ) \ \PP( A_m).
\]
		%
		Since $ \PP\lk(  \{ \mbox{a solution $(u,v)$ to \eqref{chemonoise} exists} \} \mid A_m\rk)=1$, it remains to show that
		${\lim_{m\to\infty}}$ $\PP( A_m)=1$. Then, as $A_m\supset A_{m+1}$, it follows that
		$$\PP\lk( \lk\{ \mbox{there exists solution to \eqref{chemonoise}}
		\rk\}\rk)=1.
		$$
Since there exists a constant $C(T)>0$ such that
		$$
		\EE\Big[ \sup_{0\le t\le T}|\bar u_n(t)|_{L^1}\Big]
\red{+}
\EE \Big[ \sup_{0\le t\le T} |\nabla \bar v_n(t)|^2_{L^2}\Big]
\red{+}
 \EE \Big[\int_0^T|\nabla \bar v_n(t)|_{H^1}^2\,dt\Big] \le C(T). 
		$$
		%
		we deduce
		$$\PP\lk(\Omega\setminus  A_m\rk) \le \frac {C(T)}{m }\to 0.
$$
		The solution process is well defined on $\tilde \Omega=\lim_{m\to\infty} \CA_m$, where $\PP(\tilde \Omega)=1$.
		\end{proof}

\end{steps1}

This completes the proof of Theorem \ref{main.thm}.

\end{proof}
\appendix

\del{\section{Compactness results}

\del{ A sequence of random variables $\{X_n:n\in\NN\}$ with values in metric space $M$ is called tight,  if for every $\delta>0$, there exists a compact set $K\subset M$  such that $\PP( X_n\in K)\ge 1-\delta$ for all $n\in\NN$.  Here the metric space $M$ may be a Hilbert or Banach space.
 If a sequence is tight, then it follows by the Prohorv Theorem that there exists a element $X^\ast\in M$ such that An example is M = C([0,T];H),withT > 0and H beingaHilbertspace.Moreover,tightnessyieldsacompactnessoftherandomvariableinthesenseofdistribution.LetL(Xn)denotethedistributionof Xn inspaceM.Wehavethefollowingcompactnessresult[94,Theorem2.3]. Theorem 3.14 (Prohorov theorem). Assume thatMis a separable Banach space. The set of probability measures{L(Xn)}on (M,B(M)) is relatively compact if and only if{Xn}is tight. By the Prohorov theorem, if the set of random variables{Xn}is tight, then there is a subsequence{Xnk}and some probability measure μ such that L(Xnk) → μ, weakly as k →∞}

Let $B$ be a separable Banach space, $I \subset \RR_+$ be an interval. Recall that
 $$\LL ^ p(I ;B)=\lk\{ u:I\to B: \mbox{ $u$ measurable and  }\int_I  |u(t)|_{B} ^ p\, dt <\infty\rk\}.
$$
Given $p>1$, $\alpha\in(0,1)$, let $\WW ^ {\alpha,p} (I;B)$ be the Sobolev space
of all $u\in \LL ^ p(\RR_+;B)$ such that
$$
\int_I  \int_{I\cap [t,t+1]}   \,{|u(t)-u(s)|_B ^ p\over |t-s| ^ {1+\alpha p}}\,ds\,dt<\infty.
$$
equipped with the norm
$$
\lk\| u\rk\|_{ \WW^ {\alpha,p}(I ;B)}:=\lk( \int_I  \int_{I\cap [t,t+1]}  \,{u(t)-u(s)|_B ^ p\over |t-s| ^ {1+\alpha p}}\,ds\,dt\rk) ^ \frac 1p.
$$

\begin{theorem}\label{th-gutman}
Let $B_0\subset B\subset B_1$ be Banach spaces, $B_0$ and $B_1$ reflexive, with compact embedding of $B_0$ in $B$. Let $p\in(1,\infty)$ and $\alpha\in(0,1)$ be given. Let $X$ be the space
$$
X=\LL _0^ p([0,T];B_0)\cap \WW _0^{\alpha,p} ([0,T];B_1).
$$
Then the embedding of $X$ in $\LL_0 ^ p( [0,T];B)$
is compact.
\end{theorem}

\begin{lemma}
Let $\{\xi_n:n\in\NN\}$ be a family of progressively measurable processes such that for all $\delta>0$,
\begin{itemize}
  \item there exists a $h_\delta:[0,1]\to\RR^+_0$ such that $\EE|\xi_n(t)-\xi_n(s)|^\gamma\le h(|t-s|)$ for all $n\in\NN$, $t,s\in[\delta,T]$, $|t-s|\le 1$,
  \item there exists a constant $C=C(\delta)>0$ such that $\sup_{t\in[\delta,T]}\EE \xi_n(t)|^{\gamma_2}\le C$ for all $n\in\NN$;
\end{itemize}
if $|\xi(0)|^{\gamma_3}\le \infty$ and
\end{lemma}

}

\section{Proof of Corollary \ref{cor.main}}\label{appb}

As mentioned in the introduction, we move the proof of Corollary \ref{cor.main} to the appendix.
For the convenience of the reader, we will shortly introduce the results of \cite{brzezniak} and \cite{brzezniaGatarek}, respectively, which we use within the proof of the main result. For simplicity, we do not state the results in their most general forms.

First, let us state the setting within our framework.
%

\begin{assumption}\label{assum-1}
 Let
$E$ be a Hilbert  space and let $\CA$ be a linear map on $E$ satisfying the following conditions.
\begin{enumerate}
 \item 
$-\CA $ is a positive operator\footnote{See Section I.14.1 in Triebel's monograph \cite{Triebel_1995}.} on $E$ with compact resolvent. In particular,
there
exists $M>0$ such that
\[ \Lve
(\CA +\lambda)^{-1}\Rve \le \frac{M}{1+\lambda}, \mbox{ for any }\lambda\ge 0;\]
 \item 
infinitesimal generator of an analytic semigroup of contraction
type in $E$.
\item 
 $\CA $ has the BIP (bounded imaginary power) property, i.e. there exist some
 constants $K>0$ and $\vt\in [0,\frac\pi2)$
such that
\begin{equation*}
\Vert \CA ^{is} \Vert \le K e^{\vartheta |s|}, \; s \in \mathbb{R}.
\end{equation*}

\end{enumerate}
\end{assumption}

Let us fix two real numbers $q\in (1,\infty)$ and $\mu\geq 0$.
Similarly to Brze{\'z}niak and G{\c{a}}tarek \cite{brzezniaGatarek} we
define  a linear operator ${{\mathbb{A}}}$ by the formula
\begin{align*}
D({{\mathbb{A}}}) &:= \left\{ u\in \LL ^q(0,T ;\BB): \; \CA u \in \LL ^q(0,T
;E)\right\},
\nonumber\\
 {{\mathbb{A}}}u&:=\left\{(0,T) \ni t \mapsto \CA (u(t)) \in \BB \right\},\quad u\in D({\mathbb{A}}).
\end{align*}
It  is  known, see  \cite{venni}, that  if $\CA+\mu I$ satisfies
Assumption \ref{assum-1}, then so does
${{\mathbb{A}}} + \mu I$.
With $q$ and $\CB$ as above we define two operators ${\mathbb{B}}$ and $\Lambda$, see \cite{brzezniaGatarek}, by
\begin{alignat*}{2}
{\mathbb{B}} u &:= u^\prime, & \quad u\in
D({\mathbb{B}}) &:= H_{0} ^ {1,q}(0,T;\BB).
\\
\Lambda&:= {\mathbb{B}} +{{\mathbb{A}}}, &\quad\quad D(\Lambda) &:=D({\mathbb{B}})\cap
D({{\mathbb{A}}}). \label{2.10}
\end{alignat*}

Next, by
\cite{venni} and \cite{Giga+Sohr_1991}, since $\Lambda={\mathbb{B}}-\mu I +{{\mathbb{A}}} +\mu
I$,  we deduce that $\Lambda $ is a positive
operator. In particular, $\Lambda$ has a bounded inverse. The
domain $D(\Lambda)$ of $\Lambda$ endowed with the `graph' norm
\begin{equation*}
\Vert u\Vert = \left\{ \int_0^T |u^\prime(s)|^q \, ds + \int_0^T
|\CA u(s)|^q \, ds \right\}^{\frac1q} \label{2.11a}
\end{equation*}
is a Banach space.
Let us  present two results on the fractional powers
of the  operator $\Lambda$, see  \cite{brzezniak} for the proof.

\begin{proposition}\label{Prop:2.0} Assume that Assumption \ref{assum-1} is  satisfied.
Then,  for any $ \bar \alpha \in(0, 1]$,  the operator
$\Lambda^{-\bar \alpha}$ is a bounded linear operator in $\LL^q
(0,T;E)$, and for $\eta \in \LL^q(0,T;E)$,
\[
\left(\Lambda^{-\bar \alpha}\eta\right) (t)
=
\frac1{\Gamma (\bar \alpha)}  \int_0^t (t-s)^{\bar \alpha -1}
e^{-(t-s)\CA }\eta (s) \, ds, \quad t \in [0,T].
\]
\end{proposition}

\begin{lemma}\label{L:reg}
(Compare with Lemma  2.4 of \cite{brzezniaGatarek}) Let Assumption \ref{assum-1} be satisfied.
Suppose that the positive numbers $\bar \alpha, \beta,\delta$ and $q>1$
satisfy
\begin{equation*}
\bar \alpha  -\frac1q   +  \gamma>\beta +\delta.
\label{cond:1}
\end{equation*}
    If $T  \in (0,\infty)$, then  the operator
\begin{equation*}
\Lambda ^ {-\bar\alpha} :\mathbb{ L}^q (0,T;D(\CA ^\gamma)) \to
{{C}}^\beta_b ([0,T];D(\CA ^\delta)) \; \label{2.13b}
\end{equation*}
is bounded.
\delb{
If $T=\infty$  and   the semigroup $\{e^{-t\CA }\}_{t\ge 0} $ is
exponentially bounded on $E$, i.e., for some  $ a>0$, $C>0$
\begin{eqnarray}
 |e^{-t\CA }|_{L(E,E)} &\le& Ce^{-at}, \; t \ge 0,
\label{exp-bound}
\end{eqnarray}
then  for any $f \in \LL^q_0(0,T;D(\CA ^\gamma))$ the function
$u =\Lambda ^ {-\alpha} f$ belongs to
${{\mathcal{C}}}^\beta([0,T];D(\CA ^\delta))$.
 Moreover, the operator $\Lambda ^ {-\alpha}$ is a bounded map in the
above spaces.
\begin{equation}
\Lambda ^ {-\alpha} : L^q(0,T;D(\CA ^\gamma)) \to
{{\mathcal{C}}}^\beta([0,T];D(\CA ^\delta)), \label{2.13c}
\end{equation}
is bounded.}
\end{lemma}

Finally, we present slight modifications of Proposition 2.2
and Theorem 2.6 from  \cite{brzezniak}. Corollary \ref{C:comp} is related to Lemma 3.3 of \cite{neerven}.
\begin{theorem}
\label{Th:compact}
Let Assumption \ref{assum-1} be satisfied.
 Assume that
$\bar \alpha \in (0,1]$ and $\delta,\rho\geq  0$ are such that
\begin{equation*}
\bar \alpha-\frac1q +\rho-\delta>-\frac1r.
\label{cond:2}
\end{equation*}
 Then the   operator
 \begin{equation}
 \CA  ^{\delta} \Lambda^{-\bar \alpha}\CA  ^ {-\rho} : L^q (0,T;\BB)
\to  L^r(0,T;\BB)
\label{cond:3}
\end{equation}
is bounded. Moreover, if the operator $\CA ^{-1}:\BB\to \BB$ is   compact, then the
operator in \eqref{cond:3}   is compact.
\end{theorem}
\begin{remark}
In view of Theorem \ref{Th:compact}
$$\Lambda ^ {-1}:\LL ^ p(0,T;E)\to \LL ^ p(0,T;E)
$$
is a well defined bounded linear operator for $p\in [1,\infty)$.
Observe, that $\Lambda^{-1} \eta$ is the solution to the Cauchy problem
$$
\lk\{ \begin{array}{rcl} \dot u(t)&=& \CA  u(t)+\eta(t),\\u(0)&=&0.\end{array}\rk.
$$
\end{remark}
\begin{corollary}\label{C:comp} Assume the first set of assumptions of Theorem
\ref{Th:compact} are satisfied. Assume  that  three  non-negative
numbers $\bar \alpha,  \beta, \delta$ satisfy  the  following condition
\begin{equation*}
\bar \alpha - \frac1q>\beta + \delta.
\label{cond:1a}
\end{equation*}
 Then the operator
$\Lambda^{-\bar \alpha}: \LL^q(0,T;\BB ) \to
{C}^\beta_b([0,T];D(\CA ^\delta))$ is bounded. Moreover, if the
operator $\CA ^{-1}:E\to E$ is  compact, then  the
operator $\Lambda^{-\bar \alpha}: \LL^q(0,T;\BB ) \to {C}^\beta_b([0,T];D(\CA ^\delta))$  is also compact.
\par
{ In particular, if $\bar \alpha > \frac1 q$ and the operator $\CA ^{-1}:\BB\to \BB$ is compact, the map
$\Lambda^{-\bar \alpha}:L^q(0,T;\BB ) \to C_b^0([0,T];\BB )$ is compact.
\delb{- Compare with the beginning of Section 8.}}
\end{corollary}

\medskip

Let $\mathfrak{A}=(\Omega, \CF,\BF,\PP)$ be a complete probability space
and $\BF=(\CF_t)_{t\in[0,T]}$ be a filtration satisfying the usual conditions.
Let $H$ and $H_1$ be two Hilbert spaces, $\CW$ be an $H_1$-valued cylindrical Wiener process,
$L_{HS}(H_1,H)$ be the space of linear operators belonging to $L(H_1,H)$ with finite
Hilbert--Schmidt norm, and
\begin{gather*}\label{def.MA}
\CN_{\MA}^p(0,T;H):=\Big\{ \xi:[0,T]\times \Omega\to H \ \Big| \
\mbox{$\xi$ is progressively measurable over $\mathfrak{A }$ }
\notag\\
\text{ and } \EE\int_0^T |\xi(s)|_{L_{HS}(H_1,H)}\, ds<\infty\Big\}.
\end{gather*}
To handle the stochastic convolution, let us define for a process $\xi\in \CN_{\MA}^2(0,T;H)$
the operator $\mathfrak{S}_A$ by
$$\mathfrak{S}_{\CA}\xi(t)=\int_0^ t e^{-(t-s)\CA} \xi(s)d\CW(s) 
.
$$
\begin {corollary}\label{brz:con}(Compare with Corollary 3.5 of \cite[p.\ 266]{brzezniak} and \cite[Lemma 3.3]{neerven})
If non--negative numbers $\beta,\delta$ and $\rho$ satisfy the following
$$
\beta+\delta+\frac 1p<\frac 12 -\rho,
$$
then a process $x(t)$, $t\in[0,T]$ belonging to  $\CN_{\MA}^q(0,T;\CA ^{-\rho}H)$  possesses a modification $\tilde x(t)$, $t\in[0,T]$ such that
$$
\tilde x \in C^{\beta}_b([0,T];D(\CA ^\delta)), \quad a.s.,
$$
and there exists some constant $C_T>0$ such that
$$
\EE \Big[ \|\tilde x\|^p_{\CC^{\beta}_b([0,T];D(\CA ^\delta))} \Big]
\le C_T 
\EE \Big[ \int_0^T \|\CA ^{-\rho}x(s)\|^p_{L_{HS}(H_1,H)}\, ds \Big].
$$
\end{corollary}

\del{For a function $f\in L^q(0,T;E)$ we define the operator $\Lambda_{T,\alpha}$ by
$$
(\Lambda 
_{T,\alpha} f)(t)=\frac 1 {\Gamma(\alpha)} \int_0^ t(t-s)^{\alpha-1} e^{-(t-s)\CA }f(s)\, ds ,
t\in (0,T).
$$
Observe, that $\Lambda_{T,a} f$ is the solution to the Cauchy problem
$$
\lk\{ \begin{array}{rcl} \dot u(t)&=& \CA  u(t)+f(t),\\u(0)&=&0.\end{array}\rk.
$$

\begin{lemma}
Assume a Banach space $E$ and a linear operator satisfy the condition \ref{assum-1}. Suppose that the positive numbers $\alpha,\beta,\delta$ satisfy
$$
0<\beta<\alpha-\frac 1q +\gamma-\delta.
$$
Then, if $T\in(0,\infty)$ and $f\in \LL ^q(0,T;D(\CA ^\gamma))$, the function $u=\Lambda_{T,\alpha}f$ satisfies
$$u\in\CC^{(\beta)}_b([0,T];D(\CA ^\delta)).
$$
\end{lemma}}

Now we can start with the proof of Corollary \ref{cor.main}.

\begin{proof}[Proof of Corollary \ref{cor.main}]
\del{\begin{corollary}\label{cor.main}
Let $p>6$ be fixed and $\alpha,\beta$ two real numbers such that $\alpha<0$, $\beta\in (0,\frac 12)$,  $\frac \alpha 2 +\beta<0$, and $\beta<\frac 14-\frac 2{3p} $.
Then, under the assumption of Theorem \ref{main.thm}
and the additional assumption $\EE|u_0|_{\LLn}^p<\infty$ and $ \EE|\nabla v_0|_{L^2}^{2p}<\infty$ we have
$$
\EE\|u\|^\frac {2p}3_{C^{\beta}_b([0,T];H^\alpha)}<\infty.
$$
If $\alpha,\beta$ additional satisfy $\beta\in(0,\frac 12)$, $\alpha\in(0,1)$,   $\beta+\alpha/2<\frac 14 $, and $\beta+\alpha/2<\frac 12 -\frac 1p$, then
$$
\EE\|v\|_{C^{\beta}_b([0,T];H^\alpha)}<\infty.
$$
\end{corollary}}

Let us define the operator $\MF_{r_u\DeltaA+\gamma I}:\LL^ q(0,T;L^2(\CO)) \rightarrow \LL^ q(0,T;L^2(\CO))$ for $q>1$ by
$$
(\MF_{r_u\DeltaA+\gamma I} \eta) (t)=\int_0^ t e^{(t-s)(r_u\DeltaA+\gamma I)} \eta(s)\, ds,\quad t\in[0,T],\,\, \eta \in  \LL^ q(0,T;L^2(\CO)).
$$
We observe that the operator $\MF_{r_u\DeltaA+\gamma I}$  is identical to the operator$\Lambda^{-1}$ of Proposition \ref{Prop:2.0}.
We note that the process $u$ is  given by for $t\in[0,T]$
\DEQS
u(t)&=& e^{t(r_u\DeltaA+\gamma I)}u_0+\int_0^ t e^{(t-s)(r_u\DeltaA+\gamma I)} \eta(s)\, ds+\int_0^ t e^{(t-s)(r_u\DeltaA+\gamma I)} u(s)\,dW_1(s)
\\
&=& e^{t(r_u\DeltaA +\gamma I) }u_0 + (\MF_{r_u\DeltaA+\gamma I} \eta) (t)+ \mathfrak{S}_{r_u\DeltaA+\gamma I}(u)(t),
\EEQS
with $\eta=\Div( u (\cdot)\,\nabla v(\cdot))$ and
\DEQS
v(t)&=& e^{t(r_v\DeltaA -\alpha I)}v_0+\int_0^ t e^{(t-s)(r_u\DeltaA-\alpha I)} u(s)\, ds+\int_0^ t e^{(t-s)(r_v\DeltaA-\alpha I)} v(s)\,dW_2(s)
\\
&=& e^{t(r_v\DeltaA -\alpha I)}v_0 + (\MF _{r_u\DeltaA+\gamma I}u) (t)+ (\mathfrak{S}_{r_u\DeltaA-\alpha I}v)(t).
\EEQS
First, we tackle $u$. We note that the operator $r_u\DeltaA+\gamma I$ satisfies Assumption \ref{assum-1}.
In Lemma \ref{L:reg} it can be shown that the operator $\CA$ satisfying the assumptions \ref{assum-1}.
We choose $\bar \alpha,q, \rho,\beta, \delta$ appropriately in Lemma \ref{L:reg} so that the results of
Lemma \ref{L:reg} hold for the operator $\CA=r_u\DeltaA+\gamma I$. Let us choose $\bar \alpha=1, \,1<q, \rho<-3/4$
and $\beta,\delta>0$ such that
\[
0 < \beta +\delta <1+\rho  -\frac 1q.
\]
By Lemma \ref{L:reg}, we note that for $\eta \in \LL^q(0,T;H^{2\rho })$, $\MF_\CA\eta \in  C^{\beta}_b([0,T];H^{2\delta})$, i.e., in particular,
\begin{align}\label{esti eta 12}
\|\MF_\CA\eta\|_{ C^{\beta}_b([0,T];H^{2\delta})}\le C\, \|\eta\|_{\LL^q(0,T;H^{2\rho })}.
\end{align}
As $\eta(\cdot)=\Div( u (\cdot)\,\nabla v(\cdot))$, then,  by the Sobolev embedding $L^1(\CO)\hookrightarrow H^{2\rho +1}(\CO)$, as $2\rho +1<-\frac 12$. Therefore as $H^{\theta}(\CO) \hookrightarrow L^{\infty}(\CO)$ for $\theta>1/2$, and using the interpolation inequality for
$H^{\theta}(\CO)= [H^1(\CO),L^2(\CO)]_{\theta} $
we obtain
\begin{align*}
|\Div( u (s)\,\nabla v(s)) |_{H^{2\rho}}
&\le
| u(s)\,\nabla v(s) |_{H^{2\rho+1}}
\le
| u (s)\,\nabla v(s) |_{L^1} \le |u(s)|_{L^1} |\nabla v(s)|_{L^\infty}\notag\\
&\le |u(s) |_{L^1} |\nabla v (s)|^{1-\theta}_{L^2}|\nabla v (s)|^{\theta}_{H^1}.
\end{align*}
For any $q\in(2,4)$ there exists a constant $C_q>0$ such that the $\LL^q(0,T;H^{2\rho}(\CO))$-norm of the function $\eta$ 
can be estimated by
$$
\|\eta\|_{\LL^q(0,T;H^{2\rho})} \le C_q\, \sup_{0\le s\le T}|u(s)|_{L^1}\, \sup_{0\le s\le T}|\nabla v (s)|^{1-\theta}_{L^2}
\big(\int_0^T |\nabla v (s)|^{2}_{H^1}\, ds\Big)^\frac 1{q},
$$
where $\theta  =2/q$.
In particular, the H\"older inequality gives for $\varrho=2/3$
\DEQS\label{eq ope.1}
\lqq{
\EE \Big[ \|\eta\|_{\LL^q(0,T;H^{{2\rho}})} ^{p\varrho  } \Big] }&&
\\
&\le&
\lk(\EE \Big[ \sup_{0\le s\le T}|u(s)|^p_{L^1} \Big]\rk)^ {\varrho}\, \lk( \EE \Big[\sup_{0\le s\le T}|\nabla v (s)|^{2p}_{L^2}\Big]\rk)^ {{\frac{(1-\theta)\varrho}{2} }} \lk( \EE[\int_0^ T |\nabla v (s)|^{2}_{H^1}\, ds]^p\rk)^{\frac {\varrho}{q}}  .
\EEQS
By estimate \eqref{eee} we know that the right hand side of \eqref{eq ope.1} is bounded for $p\geq 1$.
Using \eqref{esti eta 12} and \eqref{eq ope.1} we obtain
$
\EE \|\MF_\CA\eta\|_{ C^{\beta}_b([0,T];H^{2\delta})}^{p\varrho }$ is finite.
It remains to analyse  the stochastic term $\mathfrak{S}_{\CA}(u)$.
Here, we apply Corollary \ref{brz:con}  for $p>4,$ $\beta,\delta>0$ $\varrho >0$
for $\beta+\delta+\frac 1p<\frac 12-\rho $, $\rho>\frac 14$ and \eqref{HSnorm} to get
\DEQSZ\label{obenu0}
\EE \Big[ \|\mathfrak{S}_{\CA}u\|^{p\varrho}_{C^\beta_b([0,T];H^{2 \delta})} \Big]
\le C\Big(\EE \big[\sup_{0\le s\le T}|u(s)|^p_{ H^{-2\rho}}\big]\Big)^{\varrho}
\le C\Big(\EE \big[\sup_{0\le s\le T}|u(s)|^p_{L^1}\big]\Big)^{\varrho}.
\EEQSZ
By Theorem \ref{main.thm}, estimate \eqref{eeep} we observe that the right hand side of \eqref{obenu0} is finite. Hence we conclude that
$$
\EE\Big[\|u-e^{(\cdot)(r_u\DeltaA +\gamma I) }u_0\|^{\varrho
p}_{C^{\beta}_b([0,T];H^{2\delta})}\Big] <\infty.
$$
This proves the aasertion for $u$.
\medskip

It remains to verify the assertion for $v$. Let us choose $\bar \alpha=1, \,1<q, \rho<-3/4$
and $\beta,\delta>0$ in Lemmma \ref{L:reg} such that
\[
0 < \beta +\delta <1+\rho  -\frac 1q.
\]
We now have the estimate
\begin{align*}
\|\MF_{r_v\DeltaA-\alpha I}\nabla u\|_{ C^{\beta}_b([0,T];H^{2\delta})}\le C\, \|\nabla u\|_{\LL^q(0,T;H^{2\rho})}
\le C\, \|u\|_{\LL^q(0,T;H^{2\rho+1})}
\le C\, \|u\|_{\LL^q(0,T;L^1)}
.
\end{align*}
Hence,
\begin{align*}\label{eq b17}
\EE \Big[ \|\MF_{r_v\DeltaA-\alpha I}\nabla u\|_{ C^{\beta}_b([0,T];H^{2\delta})}^{p } \Big]
\leq C \EE \Big[ \sup_{0\le s\le T}|u(s)|^p_{L^1}\Big].
\end{align*}
To analyse the stochastic term, we apply Corollary \ref{brz:con}
for $p>4,$ $\beta,\delta>0$
for $\beta+\delta+\frac 1p<\frac 12-\rho $, $\rho>\frac 14$, and \eqref{HSnorm} to get
\DEQSZ\label{obenu0 12}
\EE \Big[ \|\mathfrak{S}_{r_v\DeltaA-\alpha I}\nabla v\|^{p}_{C^\beta_b([0,T];H^{2 \delta})} \Big]
&\le C\EE \big[\sup_{0\le s\le T}|\nabla v(s)|^p_{ H^{-2\rho}}\big]
\notag\\
&\le C\Big(\EE \big[\sup_{0\le s\le T}|\nabla v(s)|^{2p}_{L^2}\big]\Big)^{1/2}
\EEQSZ
By Theorem \ref{main.thm}, estimate \eqref{eeep} we observe that right hand side of \eqref{obenu0 12} is finite. Hence we conclude that
$$
\EE\Big[ \|\nabla v-e^{(\cdot)(r_v\DeltaA -\alpha I) }\nabla v_0\|^{
p}_{C^{\beta}_b([0,T];H^{2\delta})}\Big] <\infty.
$$

\medskip
This gives the assertion for $v$.
\end{proof}

\end{document}